\documentclass[11pt]{unc_dissertation}

\usepackage[T1]{fontenc}
\usepackage[latin1]{inputenc}

\usepackage[style=numeric]{biblatex}
\defbibheading{References}{References}

\usepackage{longtable}
\usepackage[acronym]{glossaries}
\usepackage[dvipsnames]{xcolor}

\usepackage{ragged2e}
\usepackage{amsmath}
\usepackage{amsthm}
\usepackage{amsfonts}
\usepackage{bbm}
\usepackage{amssymb}
\usepackage{geometry}
\usepackage{tikz, tikz-cd}
\usepackage{txfonts}
\usepackage{mathrsfs}
\usepackage[inline, shortlabels]{enumitem}
\usepackage{blkarray}
\usetikzlibrary{arrows,calc}
\usepackage{titleref}
\usepackage[normalem]{ulem} 
\usepackage{float}
\newtheorem{Lemma}{Lemma}[section]
\newtheorem{Theorem}[Lemma]{Theorem}
\newtheorem{Definition}[Lemma]{Definition}
\newtheorem{Corollary}[Lemma]{Corollary}
\newtheorem{Example}[Lemma]{Example}

\newtheorem{Remark}[Lemma]{Remark}
\newtheorem{Proposition}[Lemma]{Proposition}
\newtheorem{Conjecture}[Lemma]{Conjecture}

\newlength{\dhatheight}
\newcommand{\doublehat}[1]{%
    \settoheight{\dhatheight}{\ensuremath{\widehat{#1}}}%
    \addtolength{\dhatheight}{-0.35ex}%
    \widehat{\vphantom{\rule{1pt}{\dhatheight}}%
    \smash{\widehat{#1}}}}

\let\bs\boldsymbol

\def\llambda{{\bs \lambda}}

\def\sym{\text{Sym}}
\newcommand{\bL}{\reflectbox{$\mathsf{L}$}}

\newcommand{\bA}{\mathsf{A}}

\newcommand{\bT}{\mathsf{T}}

\newcommand{\qm}{\textnormal{\textsf{QM}}}

\newcommand{\lt}{\Gamma_{\llambda}}
\newcommand{\trwt}{\mathsf{W}}
\newcommand{\trwts}{\textsf{\textbf{W}}}

\newcommand{\tr}{\textsf{t}}
\newcommand{\trs}{\textsf{\textbf{t}}}
\newcommand{\dv}{\mathsf{v}}
\newcommand{\dw}{\mathsf{w}}
\newcommand{\ver}{\mathsf{V}}
\newcommand{\vrs}{\widehat{\mathcal{O}}_{\text{vir}}}

\newcommand{\ns}{\mathrm{ns}}
\newcommand{\rel}{\mathrm{rel}}
\newcommand{\const}{\mathscr{C}}
\newcommand{\ahat}{\hat{a}}
\newcommand{\eha}{\mathscr{U}_{\hbar}(\doublehat{\mathfrak{gl}}_1)}
\newcommand{\hopf}{\mathscr{U}_{\hbar}(\hat{\mathfrak{g}}_{Q})}
\newcommand{\hopfwall}{\mathscr{U}_{\hbar}(\mathfrak{g}_{w})}
\newcommand{\rprod}{\mathop{\overrightarrow{\prod}}}
\newcommand{\lprod}{\mathop{\overleftarrow{\prod}}}

\newcommand{\lotimes}{\mathop{\overleftarrow{\otimes}}}

\newcommand{\Lie}{\mathrm{Lie}}
\newcommand{\Wall}{\mathsf{Wall}}

\newcommand{\Pic}{\mathrm{Pic}}

\newcommand{\stab}{\mathrm{Stab}}

\newcommand{\Bw}{\textbf{B}}

\usepackage{titlesec}
\titlespacing*{\section}
{0pt}{-5pt}{0pt}
\titlespacing*{\subsection}
{0pt}{-5pt}{0pt}
\usepackage{indentfirst}   

\usepackage[hang,flushmargin]{footmisc}
\RequirePackage{ifpdf}

\ifpdf 
	\usepackage[]{graphicx} 
\else
	\usepackage[dvips]{graphicx} 
\fi

\title{Exotic Quantum Difference Equations and Integral Solutions}
\author{Hunter Dinkins}
\committee{Andrey Smirnov}{Jiuzu Hong}{Rich\'ard Rim\'anyi}{Lev Rozansky}{Alexander Varchenko}
\date{Date}

\abstract{

One of the fundamental objects in the $K$-theoretic enumerative geometry of Nakajima quiver varieties is known as the the capping operator. It is uniquely determined as the fundamental solution to a system of $q$-difference equations. Such difference equations involve shifts of two sets of variables, the variables arising as equivariant parameters for a torus that acts on the variety and an additional set of variables known as K\"ahler parameters. The difference equations in the former variables were identified with the qKZ equations in \cite{pcmilect}. The difference equations in the latter variables were identified representation theoretically in \cite{OS} using an analog of the quantum dynamical Weyl group. 

Once this representation theoretic description is known, there is an obvious generalization of these equations, which we refer to as exotic quantum difference equations. They depend on a choice of alcove in a certain hyperplane arrangement in $\Pic(X)\otimes \mathbb{R}$, with the usual difference equations corresponding to the alcove containing small anti-ample line bundles. As our main result, we relate the fundamental solution of these equations back to quasimap counts using the so-called vertex with descendants. The particular descendants of interest are given in terms of $K$-theoretic stable envelopes. 

In the case of the Hilbert scheme of points in $\mathbb{C}^2$, we write our exotic quantum difference equations explicitly using the algebra $\eha$. We use the results of \cite{dinkinselliptic} to obtain formulas for the $K$-theoretic stable envelopes of arbitrary slope. Using this, we are able to write explicit formulas for the solutions of the exotic difference equations. These formulas can be written as contour integrals. As a partially conjectural application of our results, we apply the saddlepoint approximation to these integrals to diagonalize the Bethe subalgebras of $\eha$ for arbitrary slope.
}

\newacronym{unc}{UNC}{The University of North Carolina at Chapel Hill}
\makeglossaries

\begin{document}
\RaggedRight

\ifpdf
	\DeclareGraphicsExtensions{.pdf,.jpg,.png}
\else
	\DeclareGraphicsExtensions{.eps}
\fi

\pagestyle{plain}


\frontmatter

\maketitle

\chapter*{Acknowledgements}
First of all, I would like to thank my advisor, Andrey Smirnov. Virtually none of this work would have been completed without his patience, guidance, and insight. Andrey has taught me an immense amount about learning math and conducting research. His emphasis on concrete computations and explicit examples helped make understanding abstract mathematics possible. I always walk away from meetings with Andrey feeling encouraged, inspired, and hopeful. I could not have asked for a better advisor, nor a kinder and more supportive one.

I am also grateful to the other members of my committee, Jiuzu Hong, Rich\'ard Rim\'anyi, Lev Rozansky, and Alexander Varchenko, whose varied knowledge and interests have enriched my time as a graduate student. On the administrative side of the UNC mathematics department, Sara and Laurie deserve thanks for their essential support. I have also learned much from conversations with Josh Kiers, Marc Besson, Sam Jeralds, Paul Kruse, Jay Havaldar, Collin Kofroth, Sam Crew, and Dan Zhang.

I owe a great debt of gratitude to my undergraduate professors, Robert Brabenec and Stephen Lovett, who first helped me to see the beauty of math and inspired me to pursue a PhD. Stephen's Calculus II course was the first mathematics course I took in college, and was instrumental in my decision to major in math. The many courses I took from him later played the same role in my decision to go to graduate school. Robert's holistic approach to math and his rich insights bridging math, life, and faith are things that I hope to carry with me the rest of my life.

I was privileged to have a very strong community supporting me during my time as a graduate student, including my friends Logan, Sara, Ryker, Julia, and the community of Harvest Church, as well as my siblings, Taylor, Juliet, and Marshall. I am grateful to my mom for her love, selflessness, and belief in me. I also thank my grandparents, who spent alot of hours with me as a child instilling strong study habits, a value of education, and good grammar. 

Finally, I would like to thank my wife Erica. Pursuing PhDs in mathematics at the same time has not come without its challenges, but it has also been a time of rich blessing and growth. She endured many practice talks and mathematical monologues over the last five years, which served to spare others from the same fate. It is no exaggeration to say that without her support and companionship, particularly during the isolation of the COVID pandemic, this thesis would not have been completed. For her patience, encouragement, and love I am truly grateful. 

%

\newlength{\oldbaselineskip}
\setlength{\oldbaselineskip}{\the\baselineskip}
\newlength{\oldparskip}
\setlength{\oldparskip}{\the\parskip}

\setlength{\baselineskip}{0.5\oldbaselineskip}
\setlength{\parskip}{0.5\oldbaselineskip}

\tableofcontents

\listoffigures

\setlength{\baselineskip}{1.0\oldbaselineskip}
\setlength{\parskip}{1.0\oldparskip}

\listofabbreviations

\mainmatter


\pagestyle{plain} 

\chapter{Introduction}\label{Intro}
The purpose of this thesis is to explore one aspect of the rich interplay between representation theory and algebraic geometry. Our study centers on Nakajima quiver varieties, which from their inception have bridged these two fields \cite{NakALE,NakQv}.

We first explain the algebro-geometric objects relevant for this thesis. Nakajima quiver varieties arise as the Higgs branch of the moduli space of vacua in certain quantum field theories. Let $X$ be a Nakajima quiver variety associated to a quiver $Q$ with vertex set $I$. The partition functions of the associated theory can be studied by considering counts of rational curves inside $X$. The latter can be formulated in a mathematically precise way using the moduli space of stable quasimaps from an algebraic curve to $X$ constructed in \cite{qm}. Degeneration and gluing arguments, see Section 6 of \cite{pcmilect}, allow one to reduce arbitrary quasimap counts to quasimap counts where the domain curve is $\mathbb{P}^1$. In turn, there are relatively few possibilities for quasimap counts with domain $\mathbb{P}^1$. One of these possibilities is known as the capping operator, written $\Psi(z)$. It is defined by a pushforward from the moduli space of quasimaps with relative condition at $0$ and nonsingular condition at $\infty$:
\begin{equation}\label{capintro}
\Psi(z)=\sum_{d}\left(\widehat{\text{ev}}_{0}\times\text{ev}_{\infty}\right)_{*}\left(\qm^{d}_{\substack{\ns \, \infty \\ \rel \, 0}}, \vrs^{d} \right) z^{d} \in K_{\bT\times\mathbb{C}^{\times}_{q}}(X)_{loc}^{\otimes 2}[[z]]
\end{equation}
In the above formula, we have:
\begin{itemize}
    \item $\vrs^{d}$ is the symmetrized virtual structure sheaf on the moduli space of quasimaps. Setting up quasimap counts requires of choice of polarization of $X$, which is a class $T^{1/2}\in K_{\bT}(X)$ explained in Section \ref{stabdef}. We assume a polarization is fixed.
    \item The subscript $loc$ stands for localized equivariant $K$-theory.
    \item The sum is taken over the cone in $\mathbb{Z}^I$ consisting of all possible degrees $d$ of quasimaps, and $z^{d}=\prod_{i \in I} z_i^{d_i}$.
    \item The maps $\widehat{\text{ev}}_{0}$ and $\text{ev}_{\infty}$ are the maps which evaluate quasimaps at $0$ and $\infty$.
    \item The torus $\bT$ is a maximal torus of $\text{Aut}(X)$, and $\mathbb{C}^{\times}_{q}$ is the natural torus acting on $\mathbb{P}^1$.
\end{itemize}  
\begin{sloppypar}
Using the natural bilinear pairing on equivariant $K$-theory, see (\ref{pair}), we view $\Psi(z)$ as an element of $\text{End}(K_{\bT\times\mathbb{C}^{\times}_{q}}(X)_{loc})[[z]]$.
\end{sloppypar}

As we will review in Section \ref{qdedef}, the power series $\Psi(z)$ satisfies the quantum difference equation, which is a system of equations of the form
\begin{equation}\label{qdeintro}
\Psi(z q^{\mathscr{L}})\mathscr{L}=\textbf{M}_{\mathscr{L}}(z) \Psi(z), \quad \Psi(z) \in \text{End}(K_{\bT\times\mathbb{C}^{\times}_{q}}(X)_{loc})[[z]]
\end{equation} 
for each $\mathscr{L} \in \Pic(X)$, where $\textbf{M}_{\mathscr{L}}(z)$ is some operator. Here, we abuse notation and write $\mathscr{L}$ for the operator of tensor multiplication by $\mathscr{L}$. In addition, $\Psi(z)$ is the fundamental solution of the above system of equations with the initial condition $\Psi(0)=1$. The equation (\ref{qdeintro}) is an easy consequence of the degeneration and gluing formulas for quasimaps counts, see Section 8.1 of \cite{pcmilect}. However, the operator $\textbf{M}_{\mathscr{L}}(z)$ is, like $\Psi(z)$, defined using quasimap counts, and so it is not immediately clear that the above equation provides any new information about $\Psi(z)$. The generating function $\Psi(z)$ also satisfies equations involving shifts of the equivariant parameters, which are similarly described using quasimap counts.

For (\ref{qdeintro}) to provide useful information about $\Psi(z)$, it is necessary to have a description of $\textbf{M}_{\mathscr{L}}(z)$ that does not involve quasimap counts. Such a description exists, and involves the representation theory of a certain algebra. The papers \cite{pcmilect} and \cite{OS}, building on the work of \cite{MO}, define a certain Hopf algebra $\hopf$ associated to a quiver $Q$, which acts on the equivariant $K$-theory of all quiver varieties built from $Q$. The difference equations for $\Psi(z)$ in the equivariant parameters were identified in Section 10 of \cite{pcmilect} with the quantum Knizhnik-Zamolodchikov equations for $\hopf$. The main result of \cite{OS} gives a similar representation theoretic identification of the quantum difference equation for quiver varieties. We will now explain this identification.

There is a periodic locally finite hyperplane arrangement inside of $\Pic(X)\otimes \mathbb{R}$, as in, for example, Figure \ref{hyp}. To each wall $w$ separating two alcoves of this hyperplane arrangement, \cite{OS} associates an element 
$$
\textbf{B}_{w}(z) \in \hopf[[z]]
$$
called the wall-crossing operator for the wall $w$.

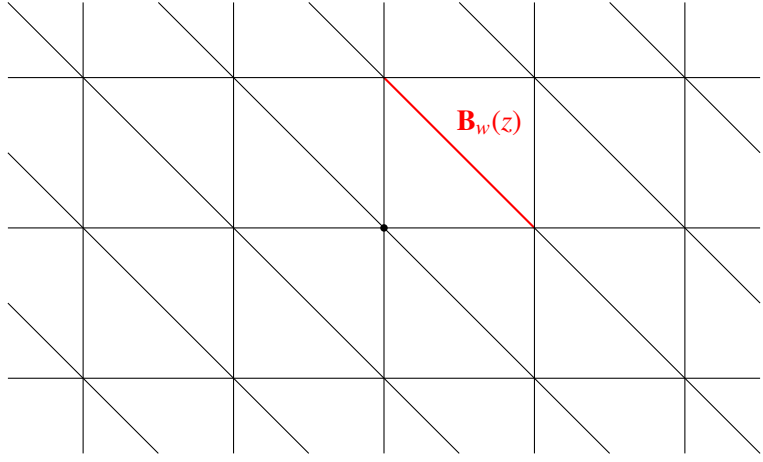
\begin{figure}[htbp]
    \centering
   \begin{tikzpicture}[scale=2,roundnode/.style={circle,fill,inner sep=1pt},roundnode2/.style={circle,fill,inner sep=0.8pt}]
     
    \node[roundnode] at (0,0){};

\draw (-1.5,1.5)--(1.5,-1.5);
\draw (-0.5,1.5)--(0,1);
\draw[red, thick] (0,1)--(1,0);
\draw (1,0)--(2.5,-1.5);
\draw (0.5,1.5)--(2.5,-0.5);
\draw (1.5,1.5)--(2.5,0.5);
\draw (-2.5,1.5)--(0.5,-1.5);
\draw(-2.5,0.5)--(-0.5,-1.5);
\draw (-2.5,-0.5)--(-1.5,-1.5);

\node[red] at (0.7,0.7){ $\textbf{B}_{w}(z)$};

\draw(-2.5,0)--(2.5,0);
\draw(-2.5,1)--(2.5,1);
\draw(-2.5,-1)--(2.5,-1);

\draw(0,-1.5)--(0,1.5);
\draw(1,-1.5)--(1,1.5);
\draw(2,-1.5)--(2,1.5);
\draw (-1,-1.5)--(-1,1.5);
\draw(-2,-1.5)--(-2,1.5);

\end{tikzpicture}  
    \caption{Elements of $\hopf(z)$ are associated to walls in a certain hyperplane arrangement in $\Pic(X)\otimes \mathbb{R}$.}
    \label{hyp}
\end{figure}

Let $\nabla_{0}$ be the alcove containing small real multiples of ample line bundles. For a line bundle $\mathscr{L}\in \Pic(X)$, choose a path between the alcoves $-\nabla_{0}$ and $-\nabla_{0}-\mathscr{L}$. Let $w_1,\ldots, w_n$ be the ordered sequence of walls crossed on this path. The walls can be crossed either positively or negatively with positive direction determined by the ample line bundle $\theta$ given by the GIT stability condition used in the construction of $X$. Let $m_i$ be the negative of the sign of the crossing of the wall $w_i$. Let
$$
\textbf{B}^{-\nabla_{0}}_{\mathscr{L}}(z)=\textbf{B}_{w_n}(z)^{m_n}\ldots \textbf{B}_{w_1}(z)^{m_1}
$$
Then the main result of \cite{OS} is that the quantum difference operator takes the form
$$
\textbf{M}_{\mathscr{L}}(z)=\const_{X} \mathscr{L} \textbf{B}^{-\nabla_{0}}_{\mathscr{L}}(z)
$$
where $\const_{X}$ is a constant. 
Once one has the above identification, it is natural to allow for different choices of the alcove. More precisely, we consider the following system of difference equations.
\begin{Definition}\label{exoticqdeintro}
Let $\nabla \subset \Pic(X)\otimes \mathbb{R}$ be an alcove. For each line bundle $\mathscr{L}\in \Pic(X)$, choose a path in $\Pic(X)\otimes \mathbb{R}$ from $\nabla$ to $\nabla-\mathscr{L}$. Let $w_1,\ldots,w_n$ be the ordered sequence of walls crossed, and let $m_i$ be the negative sign of the crossing of the wall $w_i$. Let
$$
\textbf{B}^{\nabla}_{\mathscr{L}}(z)=\textbf{B}_{w_n}(z)^{m_n}\ldots \textbf{B}_{w_1}(z)^{m_1}
$$
We define the exotic quantum difference equations for the alcove $\nabla$ to be the system of equations
$$
\Psi^{\nabla}(z q^{\mathscr{L}}) \mathscr{L}=\const_{X} \mathscr{L} \textbf{B}^{\nabla}_{\mathscr{L}}(z) \Psi^{\nabla}(z), \quad \Psi^{\nabla}(z) \in \text{End}(K_{\bT\times\mathbb{C}^{\times}_{q}}(X)_{loc})[[z]]
$$
with the initial condition $\Psi^{\nabla}(0)=1$. The constant $\const_{X}$ is the same constant that appears in the usual quantum difference equations.
\end{Definition}

This system of equations is the main object of study of this thesis. For an illustration of the construction such equations, see Figure \ref{hyp2}. Our main result is an identification of the solution $\Psi^{\nabla}(z)$ with an object arising from the enumerative geometry of quasimaps from $\mathbb{P}^1$ to $X$. 

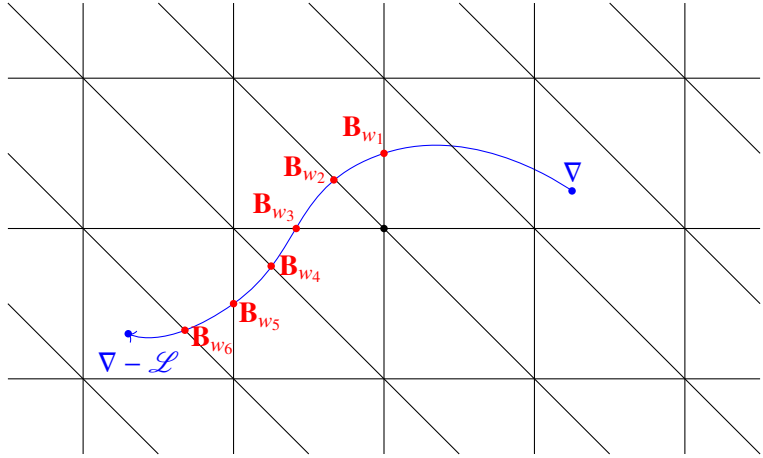
\begin{figure}[htbp]
    \centering
    \begin{tikzpicture}[scale=2,roundnode/.style={circle,fill,inner sep=1pt},roundnode2/.style={circle,fill,inner sep=1pt}]
     
    \node[roundnode] at (0,0){};
    \node[roundnode2, color=blue](a1) at (1.25,0.25){};
    \node[blue]  at ([shift={(95:0.1)}]a1.60) { $\nabla$};
      \node[roundnode2, color=blue](a2) at (-1.5-0.2,-0.5-0.2){};
          \node[blue]  at ([shift={(290:0.2)}]a2) { $\nabla-\mathscr{L}$};
         \draw [blue,->] plot [smooth, tension=1] coordinates {(a1) (0,1/2) (-1,-1/2) (a2)};
      
\draw (-1.5,1.5)--(1.5,-1.5);
\draw (-0.5,1.5)--(2.5,-1.5);
\draw (0.5,1.5)--(2.5,-0.5);
\draw (1.5,1.5)--(2.5,0.5);
\draw (-2.5,1.5)--(0.5,-1.5);
\draw(-2.5,0.5)--(-0.5,-1.5);
\draw (-2.5,-0.5)--(-1.5,-1.5);

\draw(-2.5,0)--(2.5,0);
\draw(-2.5,1)--(2.5,1);
\draw(-2.5,-1)--(2.5,-1);

\draw(0,-1.5)--(0,1.5);
\draw(1,-1.5)--(1,1.5);
\draw(2,-1.5)--(2,1.5);
\draw(-1,-1.5)--(-1,1.5);
\draw(-2,-1.5)--(-2,1.5);

\node[roundnode2, color=red] (n1) at (0,1/2){};
\node[roundnode2, color=red](n2) at(-1/3,1/3-0.01){};
\node[roundnode2, color=red](n3) at(-7/12,0){};
\node[roundnode2, color=red](n4) at(-9/12,-1/4){};
\node[roundnode2, color=red](n5) at(-1,-1/2){};
\node[roundnode2, color=red](n6) at(-1-4/12+1/100,-8/12-1/100){};

 \node[red]  at ([shift={(130:0.2)}]n1) { $\textbf{B}_{w_1}$};
 \node[red]  at ([shift={(160:0.2)}]n2) { $\textbf{B}_{w_2}$};
 \node[red]  at ([shift={(135:0.2)}]n3) { $\textbf{B}_{w_3}$};
  \node[red]  at ([shift={(-5:0.2)}]n4) { $\textbf{B}_{w_4}$};
   \node[red]  at ([shift={(-25:0.2)}]n5) { $\textbf{B}_{w_5}$};
    \node[red]  at ([shift={(-20:0.2)}]n6) { $\textbf{B}_{w_6}$};

\end{tikzpicture}  
    \caption{Exotic quantum difference equations for an alcove $\nabla$ are defined by choosing paths between $\nabla$ and $\nabla-\mathscr{L}$ for each line bundle $\mathscr{L}$. The ordinary quantum difference equation corresponds to the alcove containing small anti-ample line bundles.}
    \label{hyp2}
\end{figure}

To summarize what we have described so far, the usual quantum difference equations are defined using quasimaps and are identified using the algebra $\hopf$. The study of the exotic quantum difference equations follows the reverse logic. Namely, we define them using the algebra $\hopf$, and propose to relate them to quasimap counts.

We now provide a summary of our results. The main ideas are inspired by the work of Aganagic and Okounkov \cite{OkBethe}, which provides an alternative description of the capping operator $\Psi(z)$ from the one provided in (\ref{capintro}). Recall that in (\ref{capintro}), the operator $\Psi(z)$ is defined using the moduli space of quasimaps with relative condition of $0$ and nonsingular condition at $\infty$. The main result of \cite{OkBethe} allows one to replace a quasimap count with a relative insertion by a quasimap count with a descendant insertion. As we will review in Section \ref{stabdesc}, \cite{OkBethe} defines a map
\begin{align}\nonumber
   \textbf{f}^{\nabla}: K_{\bT}(X)& \to K_{\bT\times G_{\dv}}(pt)_{loc} \\ \label{descintro}
   \alpha &\mapsto \textbf{f}^{\nabla}_{\alpha}
\end{align}
where $G_{\dv}$ is the gauge group used in the definition of the quiver variety. The definition of the map uses stable envelopes, which are reviewed in Section \ref{stabdef}. Stable envelopes depend on a choice of chamber, polarization, and slope, each of which are explained in Section \ref{stabdef}. The map (\ref{descintro}) is built from a specific choice of chamber explained in Section \ref{stabdes}. Then \cite{OkBethe} proves that $\Psi(z)$ is the adjoint of the operator
\begin{align*}
\Omega: K_{\bT}(X) &\to  K_{\bT \times\mathbb{C}^{\times}_{q}}(X)_{loc}[[z]] \\
\alpha &\mapsto \sum_{d} \text{ev}_{\infty,*}\left(\qm_{\ns \, \infty}, \text{ev}_{0}^{*}(\textbf{f}^{\nabla_{0}}_{\text{opp},\alpha}) \otimes \vrs^{d} \right) z^d
\end{align*}
with respect to the pairing on equivariant $K$-theory. Here, $\textbf{f}^{\nabla_{0}}_{\text{opp},\alpha}$ refers to a map defined in the same way as $\textbf{f}^{\nabla_{0}}_{\alpha}$, but with the opposite choice of polarization and slope. Our main result is the following description of the solutions $\Psi^{\nabla}(z)$ of the exotic $q$-difference equations.

\begin{Theorem}[Theorem \ref{mainthm1}]\label{mainthmintro}
Let $\nabla \subset \Pic(X)\otimes \mathbb{R}$ be an alcove. Let $\Omega^{\nabla}$ be the operator 
\begin{align*}
\Omega^{\nabla}: K_{\bT}(X) &\to  K_{\bT \times\mathbb{C}^{\times}_{q}}(X)_{loc}[[z]] \\
\alpha &\mapsto \sum_{d} \text{ev}_{\infty,*}\left(\qm_{\ns \, \infty}, \text{ev}_{0}^{*}(\textbf{f}^{\nabla}_{\text{opp},\alpha}) \otimes \vrs^{d} \right) z^d
\end{align*}
We assume the polarization used to twist the virtual structure sheaf into the symmetrized virtual structure sheaf is $T^{1/2}_{\text{opp}}$. Let $(\Omega^{\nabla})^{\chi}$ be the adjoint of $\Omega^{\nabla}$ with respect to the bilinear pairing on equivariant $K$-theory. Then as elements of $\text{End}(K_{\bT\times\mathbb{C}^{\times}_{q}}(X)_{loc})[[z]]$, we have
$$
\left(\Omega^{\nabla}\right)^{\chi}=\Psi^{\nabla}
$$
\end{Theorem}
The proof of the above theorem uses rigidity arguments similar to those of \cite{OkBethe}, results about the interactions between various quasimap counts defined in \cite{pcmilect}, and properties of $\textbf{B}_{w}(z)$.

As an easy consequence of the definition of the exotic $q$-difference equations, we have
$$
\Psi^{\nabla}(z)=\textbf{B}_{w}(z) \Psi^{\nabla'}(z)
$$
if $\nabla$ and $\nabla'$ are two alcoves separated by the wall $w$, with $\nabla'$ on the positive side of the wall. Hence, at least as formal power series in $z$, the wall crossing operators possess a direct geometric interpretation as the ratio of two solutions of exotic quantum difference equations. 

This theorem can be made very explicit for type $A$ quiver varieties. In this case, formulas for elliptic stable envelopes (and, by degeneration, $K$-theoretic stable envelopes for all slopes) were developed in \cite{dinkinselliptic}. These formulas generalize the formula of Smirnov in \cite{SmirnovElliptic} for the Hilbert scheme of points in the plane, which in turn generalizes various abelianization formulas going back to Shenfeld in \cite{Shenfeld}. This allows us to write explicit expressions for $\textbf{f}^{\nabla}_{\text{opp},\alpha}$ in the $K$-theoretic fixed point basis. One can also use equivariant localization on the moduli space of quasimaps to write explicit formulas for the operator $\Omega^{\nabla}$. Putting these together provides an explicit description of $\Psi^{\nabla}$. 

While this holds for an arbitrary type $A$ quiver variety, we have chosen to write this out explicitly in this thesis only for the Hilbert scheme of points in $\mathbb{C}^2$, which is the Nakajima quiver variety arising from the quiver with one vertex and one loop. We specialize to the case of this quiver because, to the best of our knowledge, this is the only affine type $A$ case for which the alcove structure in $\Pic(X)\otimes \mathbb{R}$ is well-understood, see \cite{OS}. Explicit formulas for the wall-crossing operators also exist for the cotangent bundle of the Grassmannian, see for example Section 7 of \cite{OS}. In this case, however, all alcoves are conjugate under the action of $\Pic(X)$, and the exotic solutions turn out to be easily determined by the capping operator. For an arbitrary (finite or affine) type $A$ quiver variety, the wall-crossing operators each coincide with the wall-crossing operators of $\text{Hilb}^{n}(\mathbb{C}^2)$ or of the cotangent bundle of the Grassmannian. However, the alcove structure in $\Pic(X)\otimes \mathbb{R}$ can be complicated, which poses a challenge with making these results completely explicit in this generality.

For the Jordan quiver case and the variety $X=\text{Hilb}^{n}(\mathbb{C}^2)$, the algebra $\hopf$ is expected to be isomorphic to the quantum toroidal $\mathfrak{gl}_1$, which we write as $\eha$. Its definition will be reviewed in Section \ref{qta}. Here, $\Pic(X)\cong \mathbb{Z}$ is generated by the tautological line bundle, and the walls are given by rational numbers. The walls for $X$ are given by rational numbers with denominator at most $n$, which we denote by $\Wall_n$. For a fixed wall $w \in \mathbb{Q}$, there are elements $\alpha^{w}_k\in \eha$ for $k \in \mathbb{Z}$, which generated a Heisenberg subalgebra. The wall-crossing operator for $w$ is given in terms of these generators as
$$
\textbf{B}_{w}(z)= :\exp\left(\sum_{k=1}^{\infty} \frac{n_k \hbar^{d(w)k/2}}{1-z^{-d(w)k} q^{n(w)k} \hbar^{d(w)k/2}} \alpha^{w}_{-k} \alpha^{w}_{k}\right):
$$
where $::$ means to apply all $\alpha^{w}_k$ for $k> 0$ before $\alpha^{w}_{k}$ for $k<0$. For an alcove $\nabla \subset \mathbb{R}\setminus \Wall_n$, the exotic quantum difference equations of Definition \ref{exoticqdeintro} are given by
$$
\Psi^{\nabla}(z q^{\mathscr{L}}) \mathscr{L}= \left(\mathscr{L}\rprod_{w \in [s-\mathscr{L},s)\cap \Wall_{n}} \textbf{B}_{w}(z) \right)\Psi^{\nabla}(z)
$$
where $s \in \nabla$ and the product is taken in increasing order from left to right. 

We will give formulas for the solutions $\Psi^{\nabla}$ in the $K$-theoretic fixed point basis. We refer to the main body of this thesis, in particular Chapter \ref{CH3} and Chapter \ref{CH2}, for precise descriptions of the notations. The torus fixed points of $\text{Hilb}^{n}(\mathbb{C}^{2})$ are given by partitions $\lambda$ of $n$. The descendant insertions are given by the following.
\begin{Proposition}[Proposition \ref{hilbdesc}]
In the $K$-theoretic fixed point basis, the descendant insertions of (\ref{desc}) are given by\footnote{In the main body of the text, what is called $\textbf{f}^{\nabla}_{\text{opp},\lambda}$ here is denoted by $\textbf{g}^{\nabla}_{\text{opp},\lambda}$, and $\textbf{f}^{\nabla}_{\text{opp}}$ is reserved for a map $K_{\bT}(X) \to K_{\bT \times G_{\dv}}(T^* \text{Rep}_{Q}(\dv,\dw))_{loc}$. The two are related by postcomposing with the restriction to $0 \in T^* \text{Rep}_{Q}(\dv,\dw)$.}
$$
\textbf{f}^{\nabla}_{\text{opp},\lambda}=\sum_{\mu} \Lambda^{\bullet}\left(T_{\mu}X\right)^{-1} \textbf{S}^{\nabla}_{\lambda,\mu} (\det T^{1/2})^{1/2} \sym\left(S^{K}_{\lambda} \sum_{\tr \in \lt} \trwt_{\tr}^{K,\nabla} \right) \in K_{\bT \times G_{\dv}}(pt)_{loc}
$$
where  
$$
\textbf{S}^{\nabla}_{\lambda,\mu}=\stab_{-\mathfrak{C},T^{1/2}_{\text{opp}},-\nabla}(\lambda)|_{\mu}
$$
and
\begin{align*}
    S_{\lambda}^{K}&=\frac{\prod_{\substack{a,b, \in \lambda  \\ \rho_a+1 < \rho_b}} \ahat\left( \frac{ x_a}{t_{1} x_b}\right) \prod_{\substack{a,b, \in \lambda \\ \rho_b < \rho_a +1}} \ahat\left( \frac{ x_b}{t_{2} x_a}\right)  \prod_{\substack{a \in \lambda \\ \rho_a \leq \rho_{r} }} \ahat(x_a) \prod_{\substack{a \in \lambda \\  \rho_{r} < \rho_a }} \ahat\left(\frac{1}{t_1 t_2 x_a}\right) }{ \prod_{\substack{a,b\in\lambda \\ \rho_a<\rho_b}} \ahat\left(\frac{x_a}{x_b}\right) \ahat\left(\frac{x_a}{t_1 t_2 x_b}\right)}\\
    \trwt^{K,\nabla}_{\tr}&= (-1)^{\kappa(\tr)}\left(\frac{x_r}{\varphi^{\lambda}_{r}}\right)^{\lfloor |\lambda |s \rfloor +1/2}\ahat\left(\frac{x_r}{\varphi^{\lambda}_{r}}\right)^{-1} \prod_{e \in \tr}\ahat\left(\frac{x_{h(e)}  \varphi^{\lambda}_{t(e)}}{x_{t(e)}\varphi^{\lambda}_{h(e)}}\right)^{-1}  \left( \frac{x_{h(e)} \varphi^{\lambda}_{t(e)}}{x_{t(e)}\varphi^{\lambda}_{h(e)}}\right)^{\lfloor l_{e} s \rfloor +1/2}
    \end{align*} 
    for $s \in \nabla$.
\end{Proposition}

The solutions of the exotic $q$-difference equations are given as follows.

\begin{Theorem}[Theorem \ref{hilbsolution}]
In the $K$-theoretic fixed point basis, the matrix given by
$$
\Psi^{\nabla}(\mu)|_{\lambda}=\sum_{d \in C_{\mu}} \left(\sqrt{t_1 t_2} z\right)^{|d|} \textbf{f}^{\nabla}_{\text{opp},\lambda}(x)|_{x_i=\varphi_i^{\mu} q^{-d_i}} \prod_{i,j=1}^{n} \frac{\left(\frac{\varphi^{\mu}_j}{\varphi^{\mu}_i t_1}\right)_{d_i-d_j}}{\left(q t_2 \frac{ \varphi^{\mu}_j}{ \varphi^{\mu}_i}\right)_{d_i-d_j}} \frac{\left(q \frac{\varphi^{\mu}_j}{\varphi^{\mu}_i}\right)_{d_i-d_j}}{\left(\frac{\varphi^{\mu}_j}{ \varphi^{\mu}_i  t_1 t_2}\right)_{d_i-d_j}} \prod_{i=1}^{n}\frac{\left( \frac{1}{\varphi^{\mu}_i t_1 t_2}\right)_{d_i}}{\left(\frac{q}{ \varphi^{\mu}_i}\right)_{d_i}} 
$$
is a solution of the equation
$$
\Psi^{\nabla}(zq) \mathscr{L} =\mathscr{L} \textbf{B}^{\nabla}_{\mathscr{L}}(z) \Psi^{\nabla}(z)
$$
where $\mathscr{L}$ is the dual of the tautological line bundle, $\textbf{f}^{\nabla}_{\text{opp},\lambda}$ is given explicitly by (\ref{gdesc}), $\{\varphi^{\mu}_{i}\}_{i=1}^{n}$ are the weights of the fiber over $\mu$ of the tautological vector bundle on $X$, $(x)_{d}$ denotes the $q$-Pochammer symbol, and $C_{\mu}$ is the cone in $\mathbb{Z}^{n}$ consisting of the possible quasimap degrees.
\end{Theorem}

The above formula can be rewritten as a contour integral. Integral solutions to differential and difference equations were originally studied by Tarasov and Varchenko in \cite{TV6,TV3, TV2, TV5,TV4}. Integral formulas for solutions can be used to find eigenvectors of the operator $\left(\mathscr{L} \textbf{B}^{\nabla}_{\mathscr{L}}(z)\right)\big|_{q=1}$, as in \cite{TV4}. This operator is expected to be an element of the Bethe subalgebra of $\eha$ corresponding to the $R$-matrix $\mathscr{R}^{\nabla}$ for the alcove $\nabla$. The cohomological version of this statement is one of the main results of \cite{MO}. The eigenvectors can be obtained by the so-called saddle point approximation. This was explored in \cite{Pushk1} for the cotangent bundle of the Grassmannian, see also the introduction to \cite{OkBethe}. Although there are certain technical issues that we leave as conjectures, we show that, modulo these conjectures, the functions $\textbf{f}^{\nabla}_{\text{opp},\lambda}$, when evaluated at solutions of the Bethe equations, provide an eigenbasis for the Bethe subalgebra corresponding to the alcove $\nabla$. These eigenvectors provide a 1-parameter deformation of the modified Macdonald polynomials dependent on the choice of alcove.

We end this introduction by summarizing the contents of each chapter. In Chapter 2, we review various foundational constructions that we will need, including Nakajima quiver varieties, stable envelopes in equivariant $K$-theory, and the enumerative geometry of quasimaps. In chapter 3, we begin by reviewing the definition of the algebra $\hopf$, the construction of the wall-crossing operators $\textbf{B}_{w}(z)$, various properties of $\textbf{B}_{w}(z)$, and the definition of $\textbf{f}^{\nabla}_{\alpha}$. We then define the exotic $q$-difference equations and study their solutions. The main result, Theorem \ref{mainthmintro}, is stated precisely and proven. In Chapter 4, we begin the task of making Theorem \ref{mainthmintro} explicit for $\text{Hilb}^{n}(\mathbb{C}^{2})$ using the formulas for stable envelopes obtained in \cite{dinkinselliptic}. In Chapter 5, we specialize to $\text{Hilb}^{n}(\mathbb{C}^2)$ and provide explicit formulas for various quasimap counts, including the solutions of the exotic $q$-difference equations. We also review the definition of the algebra $\eha$ and write explicit formulas for the exotic $q$-difference equations. In Chapter 6, we apply our results to find the eigenvectors of the Bethe subalgebras of $\eha$ for arbitrary choice of slope.

\chapter{Quiver Varieties, Stable Envelopes, and Quasimap Counts}\label{CH0}
Before we begin our investigation of $q$-difference equations, we need to review some preliminaries.

\section{Nakajima quiver varieties}

Nakajima quiver varieties were first introduced in \cite{NakALE} and \cite{NakQv} to provide a geometric construction of representations of various algebras. We review their definition here. For references, see \cite{GinzburgLectures, MO, NakALE, NakQv}.

\subsection{Definition}\label{nakdef}

Let $Q$ be a quiver with vertex set $I$, which we understand to be a collection of vertices and arrows between them. If $i,j \in I$ and there is an arrow from $i$ to $j$, we write $i\to j$. Let $\dv,\dw \in \mathbb{Z}_{\geq 0}^{I}$ and for each $i \in I$ let $V_i$ and $W_i$ be complex vector spaces such that $\dim_{\mathbb{C}} V_i = \dv_i$ and $\dim_{\mathbb{C}} W_i=\dw_i$. The vectors $\dv$ and $\dw$ are called the dimension and framing dimension vectors, respectively. Let 
$$
\text{Rep}_{Q}(\dv,\dw)= \bigoplus_{i \to j} \text{Hom}(V_i, V_j) \oplus \bigoplus_{i \in I} \text{Hom}(W_i,V_i)
$$
The group $G_{\dv}:=\prod_{i \in I} GL(V_i)$ acts on $\text{Rep}(\dv,\dw)$ by change of basis. Explicitly, if $A_{i,j} \in \text{Hom}(V_i,V_j)$ is a component of a point in $\text{Rep}_{Q}(\dv,\dw)$, then an element $(g_i,g_j) \in GL(V_i)\times GL(V_j) \subset G_{\dv}$ acts via
$$
A_{i,j} \mapsto g_{j} A_{i,j} g_{i}^{-1}
$$
and the other components of $G_{\dv}$ act trivially. Additionally, if $I_i \in \text{Hom}(W_i,V_i)$, then $g_i \in GL(V_i)$ acts via
$$
I_i \mapsto g_i I_i
$$
and the remaining components act trivially. This action of $G_{\dv}$ on $\text{Rep}_{Q}(\dv,\dw)$ induces a Hamiltonian action of $G_{\dv}$ on $T^*\text{Rep}_{Q}(\dv,\dw)$. We denote the associated moment map by
$$
\mu: T^*\text{Rep}_{Q}(\dv,\dw)\to  \mathfrak{g}_{\dv}^*, \quad \mathfrak{g}_{\dv}:=\Lie(G_{\dv})
$$
and let $\mu^{-1}(0)=\mathscr{Z}(\dv,\dw)$. 

\begin{Definition}
Let $Q$, $\dv$, and $\dw$ be as above. Let $\theta: G_{\dv} \to \mathbb{C}^{\times}$ be a character of $G_{\dv}$. The Nakajima quiver variety associated to this data is
$$
\mathcal{M}_{Q,\theta}(\dv,\dw)=\mathscr{Z}(\dv,\dw) /\!\!/_{\theta} G_{\mathsf{v}}
$$
Here, the notation $\mathscr{Z}(\dv,\dw)/\!\!/_{\theta} G_{\dv}$ denotes the Geometric Invariant Theory quotient with stability condition $\theta$.
\end{Definition}

We unravel this definition to write it more explicitly. Using the perfect pairing
\begin{align*}
\text{Hom}(V,V') \times \text{Hom}(V',V)\to \mathbb{C}, \quad
   (f,f') \mapsto  \text{Tr}_{V}(f'f)
\end{align*}
we can identify
$$
\text{Hom}(V_i,V_j)^* \cong \text{Hom}(V_j,V_i)
$$
for each $i\to j$. Hence 
$$
T^*\text{Rep}_{Q}(\dv,\dw) \cong \bigoplus_{i \to j} \text{Hom}(V_i, V_j) \oplus \bigoplus_{i \to j} \text{Hom}(V_j, V_i) \oplus \bigoplus_{i \in I} \text{Hom}(W_i,V_i)   \oplus \bigoplus_{i \in I} \text{Hom}(V_i,W_i)
$$
So a point in $T^*\text{Rep}_{Q}(\dv,\dw)$ can be written as a tuple 
\begin{align*}
&\left(\{A_{i,j}\}_{i\to j}, \{B_{j,i}\}_{i\to j}, \{I_i\}_{i \in I},\{J_{i}\}_{i \in I}\right) \quad \text{where} \\
&A_{i,j} \in \text{Hom}(V_i,V_j), \quad B_{j,i} \in \text{Hom}(V_j,V_i),\quad  I_i \in \text{Hom}(W_i,V_i), \quad J_i \in \text{Hom}(V_i,W_i)
\end{align*}
We abbreviate this as $(A,B,I,J)$. Then the moment map can be explicitly written as
$$
\mu(A,B,I,J)=\left(\sum_{j\to i}A_{j,i}B_{i,j}-\sum_{i\to j}B_{j,i} A_{i,j} + I_i J_i \right)_{i \in I} \in \bigoplus_{i \in I} \mathfrak{gl}_{\dv_i}= \mathfrak{g}_{\dv} \cong \mathfrak{g}_{\dv}^*
$$
where the isomorphism $\mathfrak{g}_{\dv} \cong \mathfrak{gl}_{\dv_i}^*$ is given by the trace pairing.

The stability condition $\theta$ provides a set of $\theta$-stable points $\mathscr{Z}(\dv,\dw)^{s}\subset \mathscr{Z}(\dv,\dw)$, and set theoretically
$$
\mathcal{M}_{Q,\theta}(\dv,\dw)=\mathscr{Z}(\dv,\dw)^{s}/G_{\dv}
$$
An explicit description of the $\theta$-stable points can be found in \cite{GinzburgLectures} Proposition 5.1.5.

If the choice of $Q$ and $\theta$ has been made, we will often suppress the notation and just write $\mathcal{M}(\dv,\dw)$.

\subsection{Basic properties}\label{nakprop}
For generic choices of the stability condition $\theta$, the variety $\mathcal{M}_{Q,\theta}(\dv,\dw)$ is a smooth connected variety equipped with a symplectic form $\omega$, which is induced by the standard symplectic form on $T^*\text{Rep}_{Q}(\dv,\dw)$. Explicitly, the symplectic form is given by
\begin{equation}\label{symp}
\omega((A,B,I,J),(A',B',I',J'))=\sum_{i\to j} \text{Tr}\left(A_{i,j} B'_{j,i}-A'_{i,j} B_{j,i}\right)+\sum_{i \in I} \text{Tr}\left(I_iJ_i'-I_i'J_i \right)
\end{equation}
By the general theory of GIT quotients, there is a natural proper map
$$
\mathcal{M}_{Q,\theta}(\dv,\dw) \to \mathcal{M}_{Q,0}(\dv,\dw)
$$
to an affine variety, which makes $\mathcal{M}_{Q,\theta}(\dv,\dw)$ into an equivariant symplectic resolution.

The group $\prod_{i \in I} GL(W_i)$ acts naturally on the framing vector spaces $W_i$. This induces an action on $T^*\text{Rep}_{Q}(\dv,\dw)$. This action commutes with the action of $G_{\dv}$ and hence descends to an action on $\mathcal{M}_{Q,\theta}(\dv,\dw)$. The diagonal torus $\bA_{\dw} \subset \prod_{i \in I} GL(W_i)$ is called the framing torus. From (\ref{symp}), it is easy to see that it preserves the symplectic form.

There is also an action of $\mathbb{C}^{\times}$ on $T^*\text{Rep}(\dv,\dw)$ given by scaling of the cotangent fibers, which likewise descends to an action on $\mathcal{M}_{Q,\theta}(\dv,\dw)$. From (\ref{symp}), it is clear that this action scales the symplectic form. We denote by $\hbar$ the character of the symplectic form under this action. We denote this torus by $\mathbb{C}^{\times}_{\hbar}$.

Further torus actions can be obtained by scaling the linear maps corresponding to edges as in Section 2.1.3 of \cite{MO}.

A crucial property of quiver varieties is that fixed points of the framing torus are the disjoint union of products of quiver varieties. To write this precisely, choose a splitting $W\cong W'\oplus W''$ of the total framing space $W=\bigoplus_{i} W_i$. Let $\dw'$ and $\dw''$ be the dimension vectors of $W'$ and $W''$, respectively. Consider the subtorus $\bA:=\mathbb{C}^{\times} \subset \bA_{\dw}$ which acts with weight $1$ on $W'$ and acts trivially on $W''$. Then
\begin{equation}\label{tensor}
\mathcal{M}_{Q,\theta}(\dv,\dw)^{\bA}\cong \bigsqcup_{\substack{\dv',\dv''\\ \dv'+\dv''=\dv}} \mathcal{M}_{Q,\theta}(\dv',\dw') \times \mathcal{M}_{Q,\theta}(\dv'',\dw'')
\end{equation}
We abbreviate this by saying that the torus $\bA$ splits the framing as $\dw=u \dw'+\dw''$, where $u$ is the coordinate on $\bA$.

The vector spaces $V_i$ descend to a collection of tautological bundles $\mathscr{V}_i$ on $\mathcal{M}_{Q,\theta}(\dv,\dw)$. A crucial result of \cite{kirv} states that the Schur functors of these bundles generate the $K$-theory of $\mathcal{M}_{Q,\theta}(\dv,\dw)$.

\subsection{Equivariant $K$-theory}\label{ktheory}

Let $X$ be a variety, and let $\bT\cong (\mathbb{C}^{\times})^{n} \subset \text{Aut}(X)$ be a torus acting on $X$. The $\bT$-equivariant $K$-theory is a contravariant functor on $\bT$-spaces, and makes $K_{\bT}(X)$ a module over $K_{\bT}(pt)=\mathbb{Z}[a_1^{\pm 1},\ldots,a_n^{\pm 1}]$. The generator $a_i$ is viewed as the character of $\bT$ picking out the $i$th coordinate. We define maps $\Lambda^{\bullet}$ and $\ahat$ on $K_{\bT}(pt)$ by
\begin{align}\nonumber
\Lambda^{\bullet}(x_1+\ldots+x_k-y_1-\ldots-y_l)&=\frac{\prod_{i=1}^{k}(1-x_i)}{\prod_{j=1}^{l}(1-y_j)} \\ \label{ahatdef}
\ahat(x_1+\ldots+x_k-y_1-\ldots-y_l)&=\frac{\prod_{j=1}^{l}(y^{1/2}-y^{-1/2})}{\prod_{i=1}^{k}(x^{1/2}-x^{-1/2})}
\end{align}
where $x_i$ and $y_j$ are monomials in $a_1,\ldots,a_n$.

The fraction field $\text{Frac} \left(K_{\bT}(pt)\right)$ is the field of rational functions $\mathbb{Q}(a_1,\ldots,a_n)$. We define the localized $\bT$-equivariant $K$-theory of $X$ to be 
$$
K_{\bT}(X)_{loc}=K_{\bT}(X) \otimes_{K_{\bT}(pt)} \text{Frac}\left(K_{\bT}(pt)\right)
$$ 

Let $\iota_{X}: X^{\bT} \to X$ be the inclusion of the fixed locus. Let $X^{\bT}=\bigsqcup_{k} F_{k}$ be the decomposition into connected components. Then
$$
K_{\bT}(X^{\bT})\cong\bigoplus_{k} K_{\bT}(F_k) \cong \bigoplus_{k} K(F_k)\otimes K_{\bT}(pt)
$$
The inclusion $\iota_{X}$ induces pushforward and pullback (or restriction) maps
\begin{align*}
    \iota_{X,*}&: K_{\bT}(X^{\bT}) \to K_{\bT}(X) \\
    \iota_{X}^{*}&: K_{\bT}(X) \to K_{\bT}(X^{\bT})
\end{align*}
and similarly for the inclusion $\iota_{F_k}$ of a fixed component $F_k$. We have $\iota_{X,*}=\sum_{k} \iota_{F_k,*}$ and $\iota_{X}^{*}=\bigoplus_{k} \iota_{F_k}^{*}$. For a class $\alpha \in K_{\bT}(X)$, we will also write $\alpha|_{F_k}:= \iota_{F_k}^{*}(\alpha)$. The localization theorem in equivariant $K$-theory states that $\iota_{X,*}: K_{\bT}(X^{\bT})_{loc} \to K_{\bT}(X)_{loc}$ is an isomorphism. Its inverse is given explicitly by
$$
\iota_{X,*}^{-1}=\bigoplus_{k}\frac{1}{\Lambda^{\bullet}(N_{F_k}^{*})} \iota_{F_k}^*
$$
where $N_{F_k}$ is the normal bundle of $F_k$ in $X$, and $N_{F_k}^*$ denotes the dual bundle.

Given another $\bT$-variety $Y$ and a proper $\bT$-equivariant map $f: X\to Y$, there exists a pushforward
$$
f_{*}: K_{\bT}(X) \to K_{\bT}(Y)
$$
Let $f^{\bT}$ be the restriction of $f$ to $X^{\bT}$. The commutativity of the diagram
\begin{equation*}
\begin{tikzcd}
X^{\bT} \arrow[labels= above, "\iota_{X}", r]  \arrow[labels=left,"f^{\bT}",d] & X \arrow["f",d] \\
Y^{\bT} \arrow[r,"\iota_{Y}"] & Y^{}
\end{tikzcd}
\end{equation*}
implies that $f_{*}\circ \iota_{X,*}=\iota_{Y,*}\circ f^{\bT}_{*}$. Hence \begin{equation}\label{locpush}
    f_{*}= \iota_{Y,*} \circ f^{\bT}_{*} \circ \iota_{X,*}^{-1}
\end{equation}
If $f$ is not proper, but $f^{\bT}$ is, then we define the pushforward $f_{*}:K_{\bT}(X) \to K_{\bT}(Y)_{loc}$ by this formula.

\begin{Definition}
If $X$ is a proper $\bT$-variety, the $\bT$-equivariant Euler characteristic of $\alpha \in K_{\bT}(X)$ is $\chi(\alpha):=p_{*}(\alpha)\in K_{\bT}(pt)$, where $p:X\to pt$.
\end{Definition}

Even if $X$ is not proper, the $\bT$-equivariant Euler characteristic is well-defined so long as $X^{\bT}$ is proper. 

\begin{Proposition}[\cite{Nakfd} Theorem 7.3.5]
Let $X$ be a Nakajima quiver variety acted on by a torus $\bT$ containing the framing torus. Then $\bT$-equivariant Euler characteristics exist for $X$. Furthermore, the pairing
\begin{align}\nonumber
(\cdot,\cdot)_{X}: K_{\bT}(X)_{loc}\times K_{\bT}(X)_{loc} &\to K_{\bT}(pt)_{loc} \\ \label{pair}
(\alpha,\beta)_{X}= \chi(\alpha \otimes \beta)
\end{align}
is a non-degenerate pairing of vector spaces over $K_{\bT}(pt)_{loc}$.
\end{Proposition}

This gives an isomorphism $K_{\bT}(X)_{loc}\to K_{\bT}(X)_{loc}^*$. 

\begin{Definition}\label{adjoint}
Let $\Phi:K_{\bT}(X)_{loc}\to K_{\bT}(X)^*_{loc}$ be the isomorphism induced by the pairing $(\cdot,\cdot)_{X}$. The adjoint of a $K_{\bT}(pt)_{loc}$ linear map $A: K_{\bT}(X)_{loc}\to K_{\bT}(X)_{loc}$ with respect to $\chi$ is
$$
A^{\chi}= \Phi^{-1}\circ A^{T} \circ \Phi
$$
where $A^{T}:K_{\bT}(X)_{loc}^* \to K_{\bT}(X)_{loc}^*$ is the transpose of $A$.
\end{Definition}

\subsection{Stable envelopes}\label{stabdef}

We review the notion of stable envelopes. Stable envelopes can be studied in cohomology, $K$-theory, and elliptic cohomology. Cohomological stable envelopes were first introduced in \cite{MO}. A discussion of $K$-theoretic stable envelopes can be found in Section 9 of \cite{pcmilect} and Section 2 of \cite{OS}. In elliptic cohomology, stable envelopes were constructed for Nakajima quiver varieties in \cite{AOElliptic}. More recent generalizations in the elliptic and $K$-theoretic setting can be found in \cite{indstab1} and \cite{indstab2}. As stable envelopes in $K$-theory are the most important for this thesis, we will review only this setting.

Let $X$ be a Nakajima quiver variety acted on by a torus $\bT$. Let $\bA \subset \bT$ be a subtorus preserving the symplectic form and let $\hbar$ be the weight of the symplectic form. The $\bA$-fixed locus is a disjoint union of connected components
$$
X^{\bA}=\bigsqcup_{k} F_k
$$
which are each pointwise fixed by $\bA$. The localization theorem in equivariant $K$-theory states that the restriction map
$$
K_{\bT}(X) \hookrightarrow K_{\bT}(X^{\bA})=\bigoplus_{k} K_{\bT}(F_k)
$$
is an injective ring homomorphism which is an isomorphism after taking localized $K$-theory.  Stable envelopes, which exist when $X$ is a Nakajima quiver variety, provide a canonical map in the opposite direction. They depend on three additional pieces of data.

The first piece of data is known as a chamber. Let $\Lie_{\mathbb{R}}(\bA)=\text{cochar}(\bA)\otimes_{\mathbb{Z}}\mathbb{R}$. The natural pairing $\langle \cdot, \cdot \rangle$ on characters and cocharacters of $\bA$ extends to $\Lie_{\mathbb{R}}(\bA)$. The normal bundle $N_{F_k}$ of a fixed component $F_k$ is a module over $\bA$, and hence decomposes into the sum of its weight spaces. Each weight $w$ that appears in the decomposition provides a hyperplane
$$
H_{w}=\{\sigma\in \Lie_{\mathbb{R}}(\bA) \, \mid \, \langle w, \sigma \rangle=0 \}
$$
The hyperplanes determined in this way by all weights of the normal bundles of all fixed components provides a decomposition
$$
\Lie_{\mathbb{R}}(\bA)\setminus \bigcup_{w} H_{w}=\bigsqcup_{i} \mathfrak{C}_{i}
$$
into connected components which are called chambers. We fix a chamber and denote it by $\mathfrak{C}$. A choice of chamber provides a decomposition
$$
N_{F_k}= N_{F_k}^{+}\oplus N_{F_k}^{-}
$$
into weight spaces which pair positively and negatively with the chamber. The chamber also provides a partial order on the set of fixed components. Let $\text{Attr}(F_i)=\{x \in X \, \mid \, \lim_{z\to0} \sigma(z)\cdot x \in F_i\}$ where $\sigma$ lies in $\mathfrak{C}$. Let $\text{Attr}^{f}(F_i)$ be the smallest closed subset of $X$ containing $\text{Attr}(F_i)$ which is closed under taking attraction. Then the partial order on fixed components is defined by
$$
F_j \leq F_i \iff F_{j} \in \text{Attr}^{f}(F_i)
$$

The second piece of data is a choice of polarization, which is a $K$-theory class $T^{1/2} \in K_{\bT}(X)$ giving half of the tangent space, in the sense that
$$
TX= T^{1/2}+ \hbar^{-1} \left(T^{1/2}\right)^{*} \in K_{\bT}(X)
$$
We fix a polarization and denote it by $T^{1/2}$.

The final piece of data needed is a choice of generic slope, which is a line bundle $s \in \Pic(X)\otimes \mathbb{R}$. We explain the meaning of generic. The pairing between $H^{2}(X,\mathbb{R})\cong \Pic(X)\otimes \mathbb{R}$ and $H_{2}(X,\mathbb{R})$ gives a collection of hyperplanes
$$
\{s \in \Pic(X) \otimes \mathbb{R} \, \mid \, (s,\alpha)\in \mathbb{Z}  \}\subset \Pic(X) \otimes \mathbb{R}
$$
for each effective curve class $\alpha \in H_{2}(X,\mathbb{R})$ such that $0 < \alpha \leq \dv$. By Kirwan surjectivity \cite{kirv}, there is a surjection $\mathbb{Z}^{I}\twoheadrightarrow H^{2}(X,\mathbb{Z})$ which induces $H_{2}(X,\mathbb{Z})\hookrightarrow \mathbb{Z}^I$, and $\alpha\leq \dv$ means that, under this embedding, $\alpha_i\leq \dv_i$ for all $i$. A slope $s \in \Pic(X) \otimes \mathbb{R}$ is called generic if it lies in the complement of the union of these hyperplanes over all effective curve classes.


\begin{Proposition}[\cite{AOElliptic}, \cite{indstab1}]
Let $X$ be a Nakajima quiver variety with $\bT$ and $\bA$ as above. Choose a chamber $\mathfrak{C}$, polarization $T^{1/2}$, and generic slope $s \in \Pic(X)\otimes \mathbb{R}$. Then there exists a unique map
$$
\stab_{\mathfrak{C},T^{1/2},s}: K_{\bT}(X^{\bA}) \to K_{\bT}(X)
$$
of $K_{\bT}(pt)$-modules satisfying the following three conditions
\begin{enumerate}
    \item     $
    \stab_{\mathfrak{C},T^{1/2},s}(\mathcal{O}_{F_i})|_{F_i}= \left(\frac{\det N_{F_i}^{-} }{\det T^{1/2}_{F_i,\neq 0}} \right)^{1/2}\Lambda^{\bullet}\left( N_{F_i}^{-}\right)^{*}
    $ where $T^{1/2}_{F_i,\neq 0}$ is the part of the polarization at $F_i$ consisting of nontrivial $\bA$-weights.
    \item 
    $
    \stab_{\mathfrak{C},T^{1/2},s}(\mathcal{O}_{F_i})|_{F_j}=0$ unless $F_j \in \text{Attr}^{f}(F_i)$
    \item $\deg_{\bA}\left( \stab_{\mathfrak{C},T^{1/2},s}(\mathcal{O}_{F_i})|_{F_j}\otimes s|_{F_i} \right)\subset \deg_{\bA}\left( \stab_{\mathfrak{C},T^{1/2},s}(\mathcal{O}_{F_j})|_{F_j} \otimes s|_{F_j}\right)$ where $\deg_{\bA}$ denotes the Newton polytope of a Laurent polynomial in the weights of $\bA$.
\end{enumerate}
\end{Proposition}

For cotangent bundles of partial flag varieties, $K$-theoretic stable envelopes were also discussed in \cite{RTV, RTV2, RTV3}.

The dependence of stable envelopes on $s$ is locally constant, and hence only depends on the choice of alcove in the affine hyperplane arrangement. So we will sometimes write $\nabla$ for an alcove in place of $s$. 

There is a canonically defined slope for any quiver variety, which is given by a sufficiently small real multiple of an ample line bundle on $X$. We refer to this as the small ample slope and denote it throughout by $\epsilon$. We will denote the corresponding alcove by $\nabla_{0}$.

\section{Quasimap counts}

We briefly review the enumerative geometry of quasimaps to Nakajima quiver varieties. For more details, see \cite{qm,pcmilect, Pushk1}. We will discuss the definition of stable quasimaps, nonsingular conditions, relative conditions, descendant insertions, quantum difference equations, and localization of virtual classes.

\subsection{Stable quasimaps}
Let $X=\mathcal{M}_{Q,\theta}(\dv,\dw)$. Let $C$ be a nonsingular projective curve.

\begin{Definition}[\cite{qm} Definition 3.1.1]\label{qmrel}
A stable quasimap genus $0$ quasimap from $C$ to $X$ relative to $p_1,\ldots,p_m \in C$ is given by the data
$$
(C',p_1',\ldots,p_m',P,f,\pi)
$$
where 
\begin{itemize}
    \item $C'$ is a genus $0$ connected curve with at worst nodal singularities.
    \item The points $p_1', \ldots, p_m'$ are nonsingular points of $C'$.
    \item $P$ is a principal $G_{\mathsf{v}}$ bundle over $C'$.
    \item $f$ is a section of the bundle $P\times_{G_{\mathsf{v}}} \mathscr{Z}(\dv,\dw)$.
    \item $\pi: C' \to C$ is a regular map.
\end{itemize}
satisfying the following conditions
\begin{enumerate}
\item There is a distinguished component $C_0$ of $C'$ so that $\pi$ restricts to an isomorphism $\pi: C_0 \cong C$.
\item $\pi^{-1}(p_i)$ is a chain of rational curves connecting $p_i \in C_0$ and $p_i' \in C'$.
    \item $f(p)\in \mathscr{Z}(\dv,\dw)$ is $\theta$-stable for all but a finite set of points disjoint from $p_1',\ldots,p_m'$ and the nodes of $C'$.
    \item The automorphism group (see below) of the quasimap is finite.
\end{enumerate}
\end{Definition}

The points $p \in C'$ such that $f(p)$ lies in the $\theta$-unstable locus are called singularities of the quasimap. Otherwise, it is said to be nonsingular at $p$.

An isomorphism of quasimaps $(C_1',p_1',\ldots,p_m',P_1,f_1,\pi_1)$ and $(C_2',q_1',\ldots,q_m',P_2,f_2,\pi_2)$ is a pair $(\phi, \psi)$, where $\phi$ is an isomorphism of curves $\phi:C_1' \to C_2'$ such that $\phi(p_i)=q_i$ and $\pi_2 \circ \phi=\pi_1$, and $\psi: P_1\to P_2$ is an isomorphism of principal $G_{\dv}$-bundles such that $\psi \circ f_1=f_2$.

\begin{Definition}\label{degree}
The degree of a quasimap $(C',p_1',\ldots,p_m',P,f,\pi)$ to $X$ is the tuple $(d_i)_{i \in I}\in\mathbb{Z}^{I}$ where $d_i$ is the degree of the rank $\dv_{i}$ vector bundle $P\times_{G_{\dv}}V_i$ over $C'$.
\end{Definition}

\begin{Proposition}[\cite{qm} Theorem 7.2.2]
The stack $\qm^{d}_{\rel \, p_1,\ldots p_m}$ parameterizing stable degree $d$ quasimaps from a curve $C$ to $X$ relative $p_1,\ldots p_m$ is a Deligne-Mumford stack of finite type with a perfect obstruction theory.
\end{Proposition}

As a corollary of this Proposition, it follows that there exists a canonical class on this stack known as the virtual structure sheaf, denoted by $\mathcal{O}_{\text{vir}}$.

Instead of considering quasimaps relative to a set of points, it is conceptually simpler to consider quasimaps nonsingular at a set of points, which are defined as follows.

\begin{Definition}\label{qmns}
A stable genus $0$ quasimap from $C$ to $X$ nonsingular at $p_1,\ldots, p_m \in C$ is given by the data
$$
(C,p_1,\ldots,p_m,P,f)
$$
where
\begin{itemize}
    \item $C$ is a genus 0 connected curve with at worst nodal singularities.
    \item The points $p_1,\ldots, p_m$ are nonsingular points of $C$.
    \item $P$ is a principal $G$-bundle over $C$.
    \item $f$ is a section of the bundle $P \times_{G_{\dv}}\mathscr{Z}(\dv,\dw)$.
\end{itemize}
satisfying the following conditions
\begin{enumerate}
    \item $f(p)\in \mathscr{Z}(\dv,\dw)$ is $\theta$-stable for all but a finite set of points disjoint from $p_1,\ldots,p_m$ and the nodes of $C$.
    \item The automorphism group of the quasimap is finite.
\end{enumerate}
\end{Definition}

This definition is clearly a special case of Definition \ref{qmrel}, and nonsingular quasimaps form an open substack of the stack of relative quasimaps. One can also impose nonsingular conditions are some marked points and relative conditions at others.

More generally, one can consider genus $g$ quasimaps to $X$, where $g\neq 0$. However, by the degeneration and glue formulas of \cite{pcmilect} Section 6.5, quasimap counts for higher genus curves reduce to the genus $0$ case.

\subsection{Equivariance and evaluation maps}
We restrict our attention to quasimaps from $C= \mathbb{P}^1$ to $X$. There is a natural action of $\mathbb{C}^{\times}$ on $\mathbb{P}^1$, and we denote this torus by $\mathbb{C}^{\times}_{q}$. If a torus $\bT$ acts on $X$, there is an induced action of $\bT$ on quasimaps to $X$. We will work equivariantly with respect to $\bT\times\mathbb{C}^{\times}_{q}$. As a result, we will only consider quasimaps from $\mathbb{P}^1$ to $X$ with marked points $0$ or $\infty$ (or both). Hence we are interested in the following moduli spaces
\begin{equation}\label{qms}
\qm_{\ns \, 0}, \quad \qm_{\rel \, 0}, \quad \qm_{\substack{\ns \, 0 \\ \rel \, \infty}}, \quad \qm_{\substack{ \rel \, 0 \\ \rel \, \infty}}
\end{equation}
and the spaces obtained by permuting $0$ and $\infty$ above. When we want to avoid specifying a particular one of these spaces (for instance, to state a property which holds for all of them), we will use the notation $\qm$. So, if $\qm$ denotes any one of these spaces, we have a decomposition based on degree:
$$
\qm= \bigsqcup_{d \in \mathbb{Z}^{I}} \qm^{d}
$$
Furthermore, for any $p \in \mathbb{P}^1$, we have an evaluation map
$$
\text{ev}_{p}: \qm \to  \mathfrak{X}:=[\mathscr{Z}(\dv,\dw)/G_{\dv}]
$$
If $\qm$ imposes a relative condition at $p$, then, using the notation of Definition \ref{qmrel}, we will denote the evaluation map at $p'$ by $\widehat{\text{ev}}_{p}$. Pushforwards of the virtual structure sheaf are key objects of study in $K$-theoretic enumerative geometry. In equivariant $K$-theory, pushforwards exist for proper maps.

If $p$ is a relative (resp. nonsingular point), then $\widehat{\text{ev}}_{p}$ (resp. $\text{ev}_{p}$) lands in $X \subset \mathfrak{X}$. So, there is a commutative diagram
\begin{equation*}
\begin{tikzcd}
\qm_{\ns \, p}^{d} \arrow[labels= below left, "\text{ev}_p", rd]  \arrow[hookrightarrow]{r} & \qm_{\rel \, p}^{d} \arrow["\widehat{\text{ev}}_p",d] \\
 & X
\end{tikzcd}
\end{equation*}
It is known that the vertical map $\widehat{\text{ev}}_{p}$ is proper and we thus obtain a pushforward
$$
\widehat{\text{ev}}_{p,*}:K_{\bT\times \mathbb{C}^{\times}_q}(\qm^{d}_{\ns \, p}) \to K_{\bT\times \mathbb{C}^{\times}_{q}}(X)
$$
On the other hand, the map $\text{ev}_{p}$ is not proper. However, the restriction to the $\mathbb{C}^{\times}_{q}$ locus for a fixed degree $d$
$$
\text{ev}_{p}: \left(\qm^{d}_{\ns \, p}\right)^{\mathbb{C}^{\times}_q} \to X
$$
is proper, see Corollary 7.2.15 of \cite{pcmilect}. As a result, the pushforward $\text{ev}_{p,*}$ is well-defined in localized $K$-theory:
$$
\text{ev}_{p,*}: K_{\bT\times \mathbb{C}^{\times}_q}(\qm^{d}_{\ns \, p}) \to K_{\bT\times \mathbb{C}^{\times}_{q}}(X)_{loc}:= K_{\bT\times \mathbb{C}^{\times}_{q}}(X) \otimes \text{Frac}(K_{\bT \times \mathbb{C}^{\times}_{q}}(pt))
$$
by (\ref{locpush}). Similarly, we can consider pushforwards from quasimap moduli spaces with conditions at both $0$ and $\infty$.

\subsection{Generating functions for quasimap counts}\label{genfuncs}
Quasimap counts are packaged into certain generating functions in \cite{pcmilect}. Fix a polarization $T^{1/2}$ of $X$. This induces a virtual bundle $\mathscr{T}^{1/2}$ on $\mathbb{P}^1 \times \qm$, where $\qm$ is a moduli space of quasimaps from $\mathbb{P}^1$ to $X$. Restricting to $0$ and $\infty$ gives virtual bundles $\mathscr{T}^{1/2}_{0}$ and $\mathscr{T}^{1/2}_{\infty}$ on $\qm$. The symmetrized virtual structure sheaf is defined by 
\begin{equation}\label{symvrs}
\vrs:=\mathcal{O}_{\textrm{vir}} \otimes \left(\mathscr{K}_{\text{vir}} \otimes \frac{\det \mathscr{T}^{1/2}_{0}}{\det \mathscr{T}^{1/2}_{\infty}} \right)^{1/2}
\end{equation}
where $\mathscr{K}_{\text{vir}}=(\det T_{\text{vir}})^{-1}$ and $T_{\text{vir}}$ is the virtual tangent space of $\qm$. As explained in Section 6.1 of \cite{pcmilect}, using the symmetrized virtual structure sheaf is important for the existence of limits. 

As in Definition (\ref{degree}), the degree $d$ of a quasimap is a tuple $(d_i)_{i \in I}\in \mathbb{Z}^I$. We introduce variables $z_i$ for $i \in I$ to record the degree. We write
$$
z^{d}:=\prod_{i \in I} z_i^{d_i}
$$
The variables $z_i$ are known as the K\"ahler parameters. We will consider generating functions in $z$ below. We adopt the convention that in all such generating series, a sum $\sum\limits_{d}\ldots$ is assumed to be taken over the cone consisting of degrees for which the moduli space of quasimaps is nonempty. For such a generating series with coefficients in a ring $A$, we will abuse notation and write it as an element of the ring $A[[z]]$, with the understanding that some of the terms may involve negative powers of $z_i$.

For clarity, below we will write the moduli space of quasimaps as the first argument of evaluation maps with the second argument being the class pushed forward.
\begin{Definition}
The vertex with descendants is the map
$$
\ver: K_{\bT}(\mathfrak{X}) \to K_{\bT \times \mathbb{C}^{\times}_{q}}(X)_{loc}[[z]]
$$
defined by
\begin{equation}\label{verdef}
\ver(\tau)= \sum_{d}\text{ev}_{0,*}\left(\qm^{d}_{\ns \, 0},\text{ev}_{\infty}^{*}(\tau) \otimes \vrs^{d} \right) z^{d}
\end{equation}
\end{Definition}

The vertex with descendants provides a type of generalized $q$-hypergeometric function. For explicit computations in this regard, see \cite{dinksmir2,tRSKor, KorZeit,liu,Pushk1,dinksmir}. It also plays a key role in the enumerative perspective on 3d mirror symmetry, see \cite{AOElliptic,CDZ, msflag, dinkms1, dinksmir3}.

\begin{Definition}
The capping operator is the formal series
\begin{equation}\label{capdef}
\Psi= \sum_{d}\left(\widehat{\text{ev}}_{0}\times\text{ev}_{\infty}\right)_{*}\left(\qm^{d}_{\substack{\ns \, \infty \\ \rel \, 0}}, \vrs^{d} \right) z^{d} \in K_{\bT\times\mathbb{C}^{\times}_{q}}(X)_{loc}^{\otimes 2}[[z]]
\end{equation}
\end{Definition}
A vector space $V$ with a bilinear form $B$ allows one to define a map $V\otimes V \to \text{End}(V)$ via
$$
\sum_{i}a_i \otimes b_i: v \mapsto \sum_{i} B(b_i,v) a_i
$$
On localized equivariant $K$-theory, the equivariant Euler characteristic of (\ref{pair}) provides a bilinear form
\begin{equation}\label{capoperator}
(\cdot,\cdot)_{X}: K_{\bT\times\mathbb{C}^{\times}_{q}}(X)_{loc} \times  K_{\bT\times\mathbb{C}^{\times}_{q}}(X)_{loc} \to \text{Frac}(K_{\bT\times\mathbb{C}^{\times}_{q}}(pt))
\end{equation}
which we use to interpret $\Psi$ as an element of $\text{End}(K_{\bT \times \mathbb{C}^{\times}_{q}}(X)_{loc})[[z]]$. We will use this freely to interpret 2-tensors as operators.

\begin{Definition}
The capped vertex with descendants is the map
$$
\widehat{\ver}: K_{\bT}(\mathfrak{X}) \to K_{\bT \times \mathbb{C}^{\times}_{q}}(X)[[z]]
$$
defined by
\begin{equation}\label{cappeddef}
\widehat{\ver}(\tau)=\sum_{d}\widehat{\text{ev}}_{0,*}\left(\qm^{d}_{\rel \, 0},\text{ev}_{\infty}^{*}(\tau) \otimes \vrs^{d} \right) z^{d} \in K_{\bT \times \mathbb{C}^{\times}_{q}}(X)[[z]]
\end{equation}
\end{Definition}

Computing the capped vertex by localization, see Section 7.4 of \cite{pcmilect}, shows that
\begin{equation}\label{cpv}
\widehat{\ver}(\tau)= (\Psi \circ \ver)(\tau)
\end{equation}


\begin{Definition}
The glue matrix is the formal series
\begin{equation}\label{gluedef}
    G=\sum_{d} \left(\widehat{\text{ev}}_{0} \times \widehat{\text{ev}}_{\infty} \right)_{*} \left(\qm_{\substack{\rel \, 0 \\ \rel \, \infty}}, \vrs^{d} \right) z^{d} \in K_{\bT \times \mathbb{C}^{\times}_{q}}(X)^{\otimes 2}[[z]]
\end{equation}
\end{Definition}
As in (\ref{capoperator}), we view the glue matrix as an element of $\text{End}(K_{\bT\times\mathbb{C}^{\times}_{q}}(X))[[z]]$. Since the pushforwards in the definition of $G$ are relative pushforwards, we do not need to consider localized $K$-theory. Since $\qm_{\substack{\rel \, 0 \\ \rel \, \infty}}^{0}=X$, the constant term of the glue matrix is $1$. Hence this operator is invertible in $\text{End}(K_{\bT\times\mathbb{C}^{\times}_{q}}(X))[[z]]$.

All of these objects are series in the variables $z_i$. We will write or omit the argument $z$ depending on convenience.

\subsection{Quantum difference equation}\label{qdedef}

For a tautological line bundle $\mathscr{L}_i:=\det \mathscr{V}_i \in \Pic(X)$, we define the operator
$$
\textbf{M}_{\mathscr{L}_i}(z)=\left(\sum_{d} (\widehat{\text{ev}}_{0}\times \widehat{\text{ev}}_{\infty})_{*}\left(\qm^{d}_{\substack{\rel \, 0 \\ \rel \, \infty}} , \vrs^{d} \otimes \det H^*\left( \mathscr{V}_i\otimes \pi^*(\mathcal{O}_0) \right)\right) z^d \right)G^{-1} \in \text{End}(K_{\bT \times \mathbb{C}^{\times}_{q}}(X))[[z]]
$$
where $\pi: C \to \mathbb{P}^1$ is part of the quasimap data from Definition \ref{qmrel}.

By Kirwan surjectivity, \cite{kirv}, any line bundle $\mathscr{L}$ can be written in terms of the tautological line bundles as
$$
\mathscr{L}=\bigotimes_{i \in I} \mathscr{L}_i^{m_i}
$$
We denote
$$
z q^{\mathscr{L}}= (z_1 q^{m_1}, z_2 q^{m_2}, \ldots, z_n q^{m_n})
$$
where $n=|I|$.

\begin{Proposition}[\cite{pcmilect} Theorem 8.1.16]\label{qde}
For any tautological line bundle $\mathscr{L}_i$, the capping operator satisfies the equation
$$
\Psi(z q^{\mathscr{L}_i}) \mathscr{L}_i = \textbf{M}_{\mathscr{L}_i}(z) \Psi(z)
$$
where $\mathscr{L}_i$ denotes the operator of tensor multiplication by $\mathscr{L}_i$.
\end{Proposition}
The map
$$
\mathscr{L}_i \mapsto \textbf{M}_{\mathscr{L}_i}(z)
$$
extends to a map
\begin{align*}
\Pic(X) &\to \text{End}(K_{\bT\times \mathbb{C}^{\times}_{q}}(X))[[z]] \\
\mathscr{L} &\mapsto \textbf{M}_{\mathscr{L}}(z)
\end{align*}
such that
$$
\Psi(z q^{\mathscr{L}}) \mathscr{L}=\textbf{M}_{\mathscr{L}}(z) \Psi(z)
$$
This is known as the quantum difference equation. It provides a system of linear equations that determine each coefficient of $z$ in $\Psi(z)$ and provides an effective way to explicitly compute $\Psi(z)$.

\subsection{Localization of the virtual structure sheaf}

Equivariant localization is a powerful tool that can be used to compute pushfowards of $\vrs$ from various quasimap moduli spaces. For this purpose, it is important to understand the virtual tangent space on $\qm$, one of the moduli spaces of (\ref{qms}), at $\mathbb{C}^{\times}_{q}$-fixed points. 

Let $g \in \qm^{\mathbb{C}^{\times}_{q}}$. As at the beginning of Section \ref{genfuncs}, a polarization $T^{1/2}$ of $X$ determines a virtual vector bundle $\mathscr{T}^{1/2}$ on $\mathbb{P}^1$. By the definition of the deformation theory, see Section 4 of \cite{pcmilect}, the virtual tangent space on $\qm$ at the point $g$ is 
$$
T_{\text{vir},g}=H^{*}\left(\mathbb{P}^1, \mathscr{T}^{1/2}+\hbar^{-1} (\mathscr{T}^{1/2})^{*}\right)
$$

Let $\mathcal{K}$ be the canonical bundle on $\mathbb{P}^1$. Equivariantly, we have $\mathcal{K}-\mathcal{O}_{\mathbb{P}^1}=-\mathcal{O}_0-\mathcal{O}_{\infty}$. For later use, we will need an alternate formula for the virtual tangent space. 

\begin{Proposition}[\cite{pcmilect} Lemma 6.1.4]\label{virtan}
Let $\mathscr{T}^{1/2}$ be the class on the quasimap moduli space induced by $T^{1/2}$. The virtual tangent space is equal to
$$
T_{\text{vir}}= \mathscr{T}^{1/2}_{0} + \mathscr{T}^{1/2}_{\infty}+H-\hbar^{-1} H^{*}
$$
where $H=H^{*}(\mathbb{P}^1,\mathscr{T}^{1/2}\otimes \mathcal{K})$. 
\end{Proposition}

When computing the pushforward of $\vrs$ by equivariant localization, the main contributions arise from the virtual tangent space. Proposition \ref{virtan} is useful for analyzing the limits in $q$ of the terms that arise in localization computations. In particular, the terms $H-\hbar^{-1} H^{*}$ contribute to localization computations for $\vrs$ via rational functions like
$$
-\frac{1-\hbar x}{\sqrt{\hbar}(1- x)}
$$
whose limits exist as $x^{\pm} \to \infty$.

Later on, it will be important to understand the behavior of the $q^{\pm 1} \to 0$ limits of $\ver$. At a fixed quasimap $g \in (\qm^{d}_{\ns \, 0})^{\mathbb{C}^{\times}_{q}}$, the vector bundles for a quasimaps split
\begin{equation}\label{split}
\mathscr{V}_{i}\cong \bigoplus_{j=1}^{\dv_i}\mathcal{O}(d_{i,j})
\end{equation}
The bundles $\mathcal{O}(d_{i,j})$ acquire a natural $\mathbb{C}^{\times}_{q}$-equivariant structure.

\begin{Proposition}[\cite{pcmilect} Section 7.1]\label{linearization}
The bundles $\mathcal{O}(d_{i,j})$ arising from a $\mathbb{C}^{\times}_{q}$-fixed quasimap in $\qm^{d}_{\ns \, 0}$ are linearized so that the fibers of $\mathcal{O}(d_{i,j})$ over $0$ and $\infty$ have $q$-weights $1$ and $q^{-d_{i,j}}$, respectively.

Similarly, in the case of a $\mathbb{C}^{\times}_{q}$-fixed quasimap in $\qm^{d}_{\ns \,  \infty}$, the fibers of $\mathcal{O}(d_{i,j})$ over $0$ and $\infty$ have $q$-weights $q^{d_{i,j}}$ and $1$, respectively.
\end{Proposition}


In the twist giving $\vrs$ in (\ref{symvrs}), we assume that we are using the polarization $T^{1/2}_{\text{opp}}$ opposite some chosen polarization, so as to match notations that will appear in the next chapter. In Proposition \ref{virtan}, we can just as well use $\mathscr{T}^{1/2}_{\text{opp}}$, the virtual bundle on $\mathbb{P}^1$ induced by $T^{1/2}_{\text{opp}}$, instead of $\mathscr{T}^{1/2}$. So
$$
T_{\text{vir}}=\mathscr{T}^{1/2}_{\text{opp},0}+\mathscr{T}^{1/2}_{\text{opp},\infty}+H-\hbar^{-1}H^{*}, \quad H=H^{*}(\mathbb{P}^1,\mathscr{T}^{1/2}_{\text{opp}}\otimes \mathcal{K})
$$

Since the pushforward $\text{ev}_{0,*}\left(\qm^{d}_{\ns \, 0},\vrs^{d}\right)$ is defined by localization, by definition we have
\begin{align*}
\text{ev}_{0,*}\left(\qm^{d}_{\ns \, 0},\vrs^{d}\right)&= \sum_{g \in \left(\qm^{d}_{\ns\, 0}\right)^{\mathbb{C}^{\times}_{q}}}\frac{\text{ev}_{0,*}\left(\qm_{\ns \, 0}^{\mathbb{C}^{\times}_{q}},\vrs^{d}\big|_{g}\right)}{\Lambda^{\bullet}\left(T_{\text{vir},g}\right)}\\
&= \sum_{g \in \left(\qm^{d}_{\ns\, 0}\right)^{\mathbb{C}^{\times}_{q}}} \ahat\left(T_{\text{vir},g} \right) \left(\frac{\det \mathscr{T}^{1/2}_{\text{opp},0}\big|_{g}}{\det \mathscr{T}^{1/2}_{\text{opp},\infty}\big|_{g}} \right)^{1/2} \in K_{\bT \times \mathbb{C}^{\times}_{q}}(X)_{loc}= K_{\bT}(X)_{loc}\otimes \mathbb{C}(q)
\end{align*}
where $\ahat$ is the function from (\ref{ahatdef}).

We consider the limits in $q$. Using Proposition \ref{virtan}, we have
\begin{align}\label{ahateq}
\ahat\left(T_{\text{vir},g}\right) \left(\frac{\det \mathscr{T}^{1/2}_{\text{opp},0}\big|_{g}}{\det \mathscr{T}^{1/2}_{\text{opp},\infty}\big|_{g}} \right)^{1/2} = 
\ahat\left(\mathscr{T}^{1/2}_{\text{opp},0}\big|_{g}\right) \ahat\left(\mathscr{T}^{1/2}_{\text{opp},\infty}\big|_{g}\right) \ahat\left(H-\hbar^{-1}H^{*}\right) \left(\frac{\det \mathscr{T}^{1/2}_{\text{opp},0}\big|_{g}}{\det \mathscr{T}^{1/2}_{\text{opp},\infty}\big|_{g}}\right)^{1/2}
\end{align}
The $q^{\pm 1} \to 0$ limits of $\ahat\left(H-\hbar^{-1}H^*\right)$ clearly exist. The terms $\ahat\left(\mathscr{T}^{1/2}_{\text{opp},0}\big|_{g}\right)$ and $\det \mathscr{T}^{1/2}_{\text{opp},0}$ do not depend on $q$ because of Proposition \ref{linearization}. From the identity
$$
\frac{1}{(\hbar x)^{-1/2}((\hbar x)^{-1/2}-(\hbar x)^{1/2})}= \frac{1}{1-x^{-1}}\frac{-(1-x^{-1})}{1-(\hbar x)^{-1}}
$$
the remaining terms in (\ref{ahateq}) are
$$
\frac{\ahat\left(\mathscr{T}^{1/2}_{\text{opp},\infty}\big|_{g}\right)}{\left(\det\mathscr{T}^{1/2}_{\text{opp},\infty}\big|_{g}\right)^{1/2}}= \frac{(-1)^{\text{rank}(\mathscr{T}^{1/2}_{\text{opp},\infty}|_{g})}}{\Lambda^{\bullet}\left(\mathscr{T}^{1/2}_{\infty}\big|_{g}\right)} \frac{\Lambda^{\bullet}\left(\mathscr{T}^{1/2}_{\infty}\big|_{g}\right)}{\Lambda^{\bullet}\left(\mathscr{T}^{1/2}_{\text{opp},\infty}\big|_{g}\right)^{*}}
$$

The second fraction has finite limits at $q^{\pm} \to \infty$. So we have proved the following.

\begin{Lemma}\label{locterms}
Let $g \in \left(\qm^{d}_{\ns \, 0}\right)^{\mathbb{C}^{\times}_{q}}$. Then localization contributions from $\vrs^{d}$ to the vertex $\ver$ at $g$ take the form
$$
\frac{R_{g}}{\Lambda^{\bullet}\left(\mathscr{T}^{1/2}_{\infty}\big|_{g}\right)^{*}} 
$$
where $\lim_{q^{\pm} \to \infty} R_{g}$ is finite.
\end{Lemma}

\chapter{Exotic Difference Equations and their Solutions}\label{CH1}

In this chapter, we provide the main constructions. First, we review the definition of the algebra $\hopf$ associated to a quiver. Then, we define the wall-crossing operators and review the representation-theoretic description of the quantum difference equations. Then we define our so-called exotic quantum difference equations. We study the solutions of such equations and, in our main result, relate them to the $K$-theoretic enumerative geometry of quasimaps.



\section{Quantum difference equations}
We will review the representation-theoretic description of the quantum difference equation of Section \ref{qdedef} obtained in \cite{OS}. It utilizes a certain quantum group associated to the quiver, so we start with this.

\subsection{Quantum affine algebras}\label{algebras}

Stable envelopes in $K$-theory are the key ingredient used in \cite{OS} to construct the quantum group $\hopf$. Fix a quiver $Q$, a stability condition $\theta$, and a framing vector $\dw$. Let
$$
\mathcal{M}(\dw):=\bigsqcup_{\dv}\mathcal{M}(\dv,\dw)
$$
Let $\bT $ be a torus acting on $\mathcal{M}(\dw)$ containing the framing torus $\bA_{\dw}$ such that $\bT$ scales the symplectic form on each $\mathcal{M}(\dv,\dw)$ with character $\hbar$. For example, we could take $\bT=\bA_{\dw}\times \mathbb{C}^{\times}_{\hbar}$ where $\mathbb{C}^{\times}_{\hbar}$ scales the cotangent fibers of the prequotient as in Section \ref{nakprop}. We could add additional factors by considering the tori scaling the linear maps associated to edges in the quiver $Q$ as in Section 2.1.3 of \cite{MO}.

A subtorus $\mathbb{C}^{\times}_{u}\subset \bA_{\dw}$ of the framing torus that splits the framing as $\dw=u \dw'+ \dw''$ provides a decomposition
$$
\mathcal{M}(\dw)^{\mathbb{C}^{\times}_u}\cong \mathcal{M}(\dw')\times \mathcal{M}(\dw'')
$$
by (\ref{tensor}). Since the action of the torus $\bT$ commutes with the action of $\mathbb{C}^{\times}_{u}$, we have
$$
K_{\bT}(\mathcal{M}(\dw)^{\mathbb{C}^{\times}_{u}}) \cong K_{\bT}(\mathcal{M}(\dw'))\otimes K_{\bT}(\mathcal{M}(\dw''))
$$
Let $\mathsf{F}_{\dw}:=K_{\bT}(\mathcal{M}(\dw))$. The torus $\mathbb{C}^{\times}_{u}$ has two chambers, depending on whether the character $u$ is attracting or repelling. We denote these by $+$ and $-$, respectively. Nakajima quiver varieties have many canonical polarizations, which are obtained via universal formulas in the tautological bundles (i.e. formulas independent of $\dv$ and $\dw$), see Section 2.2.7 of \cite{MO}. We fix such a formula in the tautological bundles, and use the polarization it provides for each connected component $\mathcal{M}(\dv,\dw)$ of $\mathcal{M}(\dw)$. We denote this choice of polarization by $T^{1/2}$. As discussed in Section \ref{stabdef}, stable envelopes of a Nakajima variety $\mathcal{M}(\dv,\dw)$ depend on a choice of slope, which is a generic element of $\Pic(\mathcal{M}(\dv,\dw))\otimes \mathbb{R}$. In this context, generic means that it lies in the complement of the union of countably many affine hyperplanes. By Kirwan surjectivity, there is a surjective map $\mathbb{R}^{I}\to \Pic(\mathcal{M}(\dv,\dw)) \otimes \mathbb{R}$, where $I$ is the vertex set of the quiver. We select an element $s\in \mathbb{R}^{I}$ that maps to a generic slope under this map for all choices of $\dv$ and $\dw$. This is possible because there are only countably many hyperplanes to be avoided. 

Then the slope $s$ $R$-matrix is defined as
\begin{equation}\label{slopeR}
\mathscr{R}^{s}(u)=\stab_{-,T^{1/2},s}^{-1}\circ \stab_{+,T^{1/2},s} \in \text{End}(\mathsf{F}_{\dw'}\otimes \mathsf{F}_{\dw''}) \otimes \mathbb{C}(u)
\end{equation}
Rational functions of $u$ are required to invert the stable envelope.

The compatibility of $R$-matrices with further splittings is expressed by the following proposition.
\begin{Proposition}\label{stabsplit}
Let $\dw_1,\dw_2$, and $\dw_3$ be framing vectors. Let $\stab^{23}_{+,T^{1/2},s}: \mathsf{F}_{\dw_2}\otimes \mathsf{F}_{\dw_3} \to \mathsf{F}_{\dw_2+\dw_3}$. Then
$$
(1 \otimes \stab^{23}_{s,+})^{-1} \mathscr{R}^{s}_{\mathsf{F}_{\dw_1},\mathsf{F}_{\dw_2+\dw_3}} (1 \otimes \stab^{23}_{s,+})= \mathscr{R}^{s,13}_{\mathsf{F}_{\dw_1},\mathsf{F}_{\dw_3}} \mathscr{R}^{s,12}_{\mathsf{F}_{\dw_1},\mathsf{F}_{\dw_2}}
$$
where $\mathscr{R}^{s,ij}_{\mathsf{F}_{\dw_i},\mathsf{F}_{\dw_j}}$ stands for the $R$-matrix acting in the $i$th and $j$th components.
\end{Proposition}
\begin{proof}
This follows from the triangle lemma for stable envelopes. For details, see Section 4.2 of \cite{MO}.
\end{proof}

Using these $R$-matrices, \cite{OS} employs the so-called FRT procedure of \cite{FRT} to define a Hopf algebra $\hopf$ which acts on $\mathsf{F}_{\dw}$ for any choice of framing $\dw$, see also \cite{MO} Section 4 for a discussion of the FRT procedure in cohomology. We review the procedure here.

 As defined above, we have an $R$-matrix $\mathscr{R}^{s}_{\mathsf{F}_{\dw},\mathsf{F}_{\dw'}}(u)$ acting in the tensor product of two such spaces. We define $R$-matrices acting on arbitary tensor products by
\begin{equation}\label{tensors}
\mathscr{R}^{s}_{\lotimes_{j \in J} \mathsf{F}_{\dw_j},\lotimes_{k \in K} \mathsf{F}_{\dw_k}}:=\rprod_{j \in J} \lprod_{k \in K} \mathscr{R}^{s}_{\mathsf{F}_{\dw_j},\mathsf{F}_{\dw_k}}(u_i/u_k) \in \text{End}\left(\lotimes_{j \in J} \mathsf{F}_{\dw_j}\otimes\lotimes_{k \in K} \mathsf{F}_{\dw_k} \right)\otimes \mathbb{C}(\{u_j/u_k\}_{j\in J, k \in K})
\end{equation}
where the arrows over $\prod$ denotes the product ordered increasing from right to left over some choice of ordering of the index sets $J$ and $K$. In the above products, we understand $\mathscr{R}^{s}_{\mathsf{F}_{\dw_j},\mathsf{F}_{\dw_k}}(u_i/u_k)$ to be acting by the $R$-matrix in the sub-scripted components and by $1$ elsewhere. We similarly define $R$-matrices acting on the duals of such spaces by
$$
\mathscr{R}^{s}_{\mathsf{F}_{\dw},\mathsf{F}_{\dw'}^*}=\left(\left(\mathscr{R}^{s}_{\mathsf{F}_{\dw},\mathsf{F}_{\dw'}}\right)^{-1}\right)^{t_2}, \quad \mathscr{R}^{s}_{\mathsf{F}_{\dw}^*,\mathsf{F}_{\dw'}}=\left(\left(\mathscr{R}^{s}_{\mathsf{F}_{\dw},\mathsf{F}_{\dw'}}\right)^{-1}\right)^{t_1}, \quad \mathscr{R}^{s}_{\mathsf{F}_{\dw}^*,\mathsf{F}_{\dw'}^*}=\left(\mathscr{R}^{s}_{\mathsf{F}_{\dw},\mathsf{F}_{\dw'}}\right)^{t_{12}}
$$
where $t_i$ stands for the transpose with respect to the $i$th factor.

Let $\mathfrak{B}$ be the smallest set containing all $\mathsf{F}_{\dw}$, closed under taking duals and tensor products. The quantum group $\hopf$ is defined as a subalgebra of $\prod_{V \in \mathfrak{B}} \text{End}(V)$. For an element $a \in \prod_{V \in \mathfrak{B}}\text{End}(V)$, denote by $a_{V}$ the component of $a$ acting in $\text{End}(V)$ for some specific $V \in \mathfrak{B}$.

\begin{Definition}
Fix an auxillary space $V_0 \in \mathfrak{B}$ and a finite rank operator $m(u) \in \text{End}(V_0)\otimes \mathbb{C}(u)$. Let $E(m(u))\in \prod_{V \in \mathfrak{B}}\text{End}(V)$ be the operator such that
$$
E(m(u))_{V}=\text{Res}_{u=\infty}\text{Tr}((m(u)\otimes 1) \mathscr{R}^{s}_{V_0,V})
$$
Let $\hopf$ be the subalgebra of $\prod_{V \in \mathfrak{B}}\text{End}(V)$ generated by $E(m(u))$ over all auxiallary spaces $V_0$ and all $m(u) \in \text{End}(V_0)$.
\end{Definition}

\begin{Proposition}[\cite{FRT}]
The algebra $\hopf$ is a Hopf algebra.
\end{Proposition}

The algebra $\hopf$ depends on the choice of slope $s$ used in the construction of the $R$-matrices. 
\begin{Proposition}[\cite{OS}]
Different choices of slope lead to isomorphic algebras with nonisomorphic Hopf structures.
\end{Proposition}

By the above proposition, we are justified in writing $\hopf$ for the algebra constructed from any choice of slope. However, we will write $\Delta_{s}$, $S_{s}$, and $\epsilon_{s}$ for the coproduct, antipode, and counit structure labeled by $s$.

By construction, $\hopf$ acts on any of the elements of $\mathfrak{B}$. Any of these modules are isomorphic to a tensor product of the so-called fundamental representations
\begin{equation}\label{fundreps}
\mathsf{F}_i=K_{\bT}(\mathcal{M}(\delta_i))
\end{equation}
and their duals, where $\delta_i$ stands for the single framing at vertex $i$.

By definition, the slope $s$ coproduct $\Delta_{s}$ can be described via stable envelopes. 

\begin{Proposition}\label{coprod}
Given $a \in \hopf$, we write $a_{\dw}$ for the component of $a$ acting on $\mathsf{F}_{\dw}$ and $a_{\dw,\dw'}$ for the component acting on $\mathsf{F}_{\dw}\otimes \mathsf{F}_{\dw'}$. Let $\stab_{+,T^{1/2},s}:\mathsf{F}_{\dw} \otimes \mathsf{F}_{\dw'} \to \mathsf{F}_{\dw+\dw'}$. Then we have
$$
\Delta_{s}(a)_{\dw,\dw'}=\stab_{+,T^{1/2},s}^{-1} \circ a_{\dw+\dw'} \circ \stab_{+,T^{1/2},s}
$$
\end{Proposition}
\begin{proof}
Given a generator 
$$
E(m(u)) \in  \hopf
$$
where $m(u)\in \text{End}(V_0)$, we have
$$
E(m(u))_{\dw_1+\dw_2}=\text{Res}_{u=\infty} \text{Tr}_{V_0}\left((m(u)\otimes 1) \mathscr{R}^{s}_{V_0,\mathsf{F}_{\dw_1+\dw_2}} \right)
$$
by definition. So
\begin{align*}
\stab_{+,T^{1/2},s}^{-1} E(m(u))_{\dw_1+\dw_2} \stab_{+,T^{1/2},s}&= \text{Res}_{u=\infty} \text{Tr}_{V_0}\left((1 \otimes \stab_{+,T^{1/2},s}^{-1})(m(u)\otimes 1) \mathscr{R}^{s}_{V_0,\mathsf{F}_{\dw_1+\dw_2}} (1 \otimes \stab_{+,T^{1/2},s}) \right) \\
&=\text{Res}_{u=\infty} \text{Tr}_{V_0}\left((m(u)\otimes 1\otimes 1) (1 \otimes \stab_{+,T^{1/2},s}^{-1}) \mathscr{R}^{s}_{V_0,\mathsf{F}_{\dw_1+\dw_2}} (1 \otimes \stab_{+,T^{1/2},s}) \right)  \\
&=\text{Res}_{u=\infty} \text{Tr}_{V_0}\left((m(u)\otimes 1\otimes 1)  \mathscr{R}^{s}_{V_0, \mathsf{F}_{\dw_2}} \mathscr{R}^{s}_{V_0,\mathsf{F}_{\dw_1}}\right) 
\end{align*}
where the last line follows by Proposition \ref{stabsplit}. By (\ref{tensors}) and the definition of $E(m(u))_{\dw_1,\dw_2}$, we have
\begin{align*}
    \stab_{+,T^{1/2},s}^{-1} E(m(u))_{\dw_1+\dw_2} \stab_{+,T^{1/2},s}&=\text{Res}_{u=\infty} \text{Tr}_{V_0}\left((m(u)\otimes 1\otimes 1)  \mathscr{R}^{s}_{V_0, \mathsf{F}_{\dw_1} \otimes\mathsf{F}_{\dw_2}} \right) \\
    &=\Delta_{s}(E(m(u)))_{\dw_1,\dw_2}
\end{align*}
\end{proof}

Similarly, one can consider the change between stable envelopes on opposite sides of a wall. Let $w\subset \Pic(X) \otimes \mathbb{R}$ be a wall, and let $\nabla$ and $\nabla'$ be two alcoves on opposite sides of $w$. We assume that there is a real line bundle $\mathscr{L}_{w} \in \Pic(X) \otimes \mathbb{R}$ on the wall and a small ample slope $\epsilon$ such that $\mathscr{L}_{w} +\epsilon\in \nabla'$ and $\mathscr{L}_{w}-\epsilon \in \nabla$. In this case, we say that the wall $w$ is crossed in the positive direction from $s$ to $s'$. Given a splitting $\dw=u \dw'+ \dw''$, we define the wall $w$ $R$-matrix by
\begin{equation}\label{wallR}
R^{\pm}_{w}(u)=\stab_{\pm,T^{1/2},s'}^{-1} \circ \stab_{\pm,T^{1/2},s} \in \text{End}(\mathsf{F}_{\dw'}\otimes \mathsf{F}_{\dw''})\otimes \mathbb{C}(u)
\end{equation}
We renormalize the wall $R$-matrices by
$$
\mathsf{R}^{\pm}_{w}=\hbar^{\Omega} R^{\pm}_{w}
$$
where $\Omega$ is the diagonal operator which scales classes supported on a fixed component $\mathcal{M}(\dv_1,\dw_1)\times \mathcal{M}(\dv_2,\dw_2)$ by one-fourth of the codimension of the component. The codimension function can be written explicitly as a function of $\dv,\dw,\dv',$ and $\dw'$, see formula (34) of \cite{OS}. Again, the FRT procedure allows one to use these $R$-matrices to construct a Hopf algebra $\hopfwall$ called the wall subalgebra. As shown in \cite{OS}, each wall subalgebra is a subalgebra of $\hopf$.

\subsection{Wall crossing operators}

There are certain distinguished elements of $\hopf$ that are used to describe the quantum difference equation. 

If $\gamma \in K_{\bT}(\mathcal{M}(\dv_1,\dw_1)\times\ldots \times \mathcal{M}(\dv_m,\dw_m))$, let
$$
Z_{(k)}(\gamma)=\prod_{i \in I} z_i^{\dv_{k,i}} \gamma
$$
We denote by $Z_{(k)}$ the corresponding element of $\prod_{\dw_1,\ldots,\dw_m}\text{End}(\mathsf{F}_{\dw_1} \otimes \ldots \otimes \mathsf{F}_{\dw_m})\otimes \mathbb{C}(z)$.

\begin{Definition}\label{triangular}
Let $A$ be an operator acting on $\mathsf{F}_{\dw}\otimes \mathsf{F}_{\dw'}$ such that $A=\bigoplus_{\alpha} A_{\alpha}$ where
$$
A_{\alpha}: K_{\bT}(\mathcal{M}(\dv,\dw))\otimes K_{\bT}(\mathcal{M}(\dv',\dw'))\to K_{\bT}(\mathcal{M}(\dv+\alpha,\dw))\otimes K_{\bT}(\mathcal{M}(\dv'-\alpha,\dw'))
$$
Then $A$ is said to be strictly lower triangular if $\langle \alpha,\theta \rangle\leq 0$ for all $\alpha$ that appear and $A_{0}=1$. Similarly, $A$ is said to be strictly upper triangular if $\langle \alpha,\theta \rangle\geq 0$ for all $\alpha$ that appear and $A_{0}=1$.
\end{Definition}

A direct summand $K_{\bT}(\mathcal{M}(\dv_{1},\dw_{1})\times\ldots\times\mathcal{M}(\dv_{m},\dw_{m})$ in $\mathsf{F}_{\dw_1}\otimes \mathsf{F}_{\dw_{m}}$ is called a weight space. The sum $\dv_1+\ldots +\dv_m$ is called the total weight. Then it follows that an operator as in Definition \ref{triangular} preserves the total weight.

\begin{Proposition}[\cite{OS} Proposition 6]
For each wall $w$, there exists an element $J_{w}(z) \in \hopfwall \otimes \hopfwall \otimes \mathbb{C}(z)$ so that $J_{w}(z)_{\dw,\dw'}$ is a strictly lower triangular operator on $\mathsf{F}_{\dw}\otimes \mathsf{F}_{\dw'}$ satisfying the equation
\begin{equation}\label{abrr}
\mathsf{R}_{w}^{-} Z_{(1)}^{-1} J_{w}(z)= J_{w}(z) \hbar^{\Omega} Z_{(1)}^{-1}
\end{equation}
for all framing vectors $\dw$ and $\dw'$.
\end{Proposition}

Let $C$ be the Cartan matrix of the quiver, i.e. $C=2 I -(A_{Q}+A_{Q}^{T})$, where $A_{Q}$ is the adjacency matrix of the quiver. For dimension and framing dimension vectors $\dv$ and $\dw$, let $\kappa(\dv,\dw)=(C\dv-\dw)/2$ and let $\kappa(\dv,\dw)_{i}$ be the $i$th component of $\kappa(\dv,\dw)$. Let $\mathscr{L}_{w}$ be a real line bundle on the wall $w$. 

For each wall $w$, the formula
\begin{equation}
\textbf{m}\left((1 \otimes S_{w}) J_{w}(z q^{-\mathscr{L}_w})^{-1} \right)
\end{equation}
provides an element of $\hopfwall \otimes \mathbb{C}(z,q)$, where $S_{w}$ is the antipode of the Hopf algebra $\hopfwall$ and $\textbf{m}: \hopfwall \otimes \hopfwall \to \hopfwall$ is the multiplication map. Since $J_{w}(z)$ is lower triangular (and in particular, preserves the total weight in the tensor product of two representations), the above element preserves each weight space $K_{\bT}(\mathcal{M}(\dv,\dw))$. If $A\in \hopfwall\otimes \mathbb{C}(z,q)$ is an operator that preserves each weight space, we denote by $A|_{z\to z\hbar^{\kappa}}$ the result of substituting $z_i\to z_i \hbar^{\kappa(\dv,\dw)_i}$ in the component acting on $K_{\bT}(\mathcal{M}(\dv,\dw))$.

\begin{Definition}[\cite{OS} equation (66)]
For each wall $w$, let
\begin{equation}\label{wallcross}
\textbf{B}_{w}(z)= \textbf{m}\left((1 \otimes S_{w}) J_{w}(z q^{-\mathscr{L}_w})^{-1} \right)\big|_{z\to z \hbar^{\kappa}} \in \hopfwall \otimes \mathbb{C}(z,q)
\end{equation}
We refer to $\textbf{B}_{w}(z)$ as the wall-crossing operator for the wall $w$.
\end{Definition}

\subsection{Properties of wall-crossing operators}
We also introduce the shift operator $T_{\mathscr{L}}$, which acts on rational functions of $z$ by $z \to z q^{\mathscr{L}}$.

\begin{Proposition}[\cite{OS} Proposition 9]\label{conjugateB}
$$
\mathscr{L} T_{\mathscr{L}}^{-1}\textbf{B}_{w}(z)= \textbf{B}_{w+\mathscr{L}}(z) \mathscr{L} T_{\mathscr{L}}^{-1}
$$
\end{Proposition}

Given a wall $w$, we denote by $\mathscr{L}_w \in \Pic(X) \otimes_{\mathbb{Z}} \mathbb{R}$ a real line bundle on this wall. While there is no canonical choice of $\mathscr{L}_w$, for our purposes, it does not matter which line bundle is chosen.

\begin{Proposition}\label{Bshift}
The operator
$$
\textbf{B}_{w}(z q^{\mathscr{L}_w})
$$
does not depend on $q$.
\end{Proposition}
\begin{proof}
The only $q$-dependence in $\textbf{B}_{w}(z)$ arises from the shift $z\to z q^{-\mathscr{L}_w}$. Hence reversing this shift removes any $q$-dependence.
\end{proof}

\begin{Lemma}\label{Jdegrees}
$J_{w}(z)$ can be expanded as a formal power series in $z$ as
$$
J_{w}(z)=1+\sum_{\substack{d \\ \langle d, \theta \rangle> 0}} J_{w,d} z^{d}, \quad J_{w,d} \in \hopfwall
$$
where $\theta$ is the stability condition.
\end{Lemma}
\begin{proof}
This follows from the strictly lower triangularity of $J_{w}(z)$ and $\mathsf{R}_{w}^{-}$, along with the equation (\ref{abrr}).
\end{proof}

\begin{Proposition}\label{Bdegrees}
$\textbf{B}_{w}(z)$ can be expanded as a formal power series in $z$ as
$$
\textbf{B}_{w}(z)=1+\sum_{\substack{d \\ \langle d, \theta \rangle> 0}} \textbf{B}_{w}^{d} z^{d}, \quad \textbf{B}_{w}^{d} \in \hopfwall[q^{\pm 1}]
$$
where $\theta$ is the stability condition.
\end{Proposition}
\begin{proof}
This follows from the definition of $\textbf{B}_{w}(z)$ and the previous lemma.
\end{proof}

\begin{Proposition}\label{Blim}
$$
\lim_{q \to \infty} \textbf{B}_{w}(z q^{\mathscr{L}}) = 1
$$
if $\mathscr{L}_{w}-\mathscr{L}$ is ample.
\end{Proposition}
\begin{proof}
Shifting $z\to z q^{\mathscr{L}-\mathscr{L}_w}$ in $J_{w}(z)$ gives
$$
J_{w}(z q^{\mathscr{L}-\mathscr{L}_{w}})=1+ \sum_{\substack{d \\ \langle d, \theta \rangle >0}}  J_{w,d} q^{\langle \mathscr{L}-\mathscr{L}_{w},d\rangle} z^{d}
$$
If $\mathscr{L}_{w}-\mathscr{L}$ is ample, then the powers of $q$ appearing above are all negative. Hence
$$
\lim_{q\to \infty} J_{w}(z q^{\mathscr{L}-\mathscr{L}_w})=1
$$
From (\ref{wallcross}), we have
$$
\textbf{B}_{w}(z q^{\mathscr{L}})=\textbf{m}\left((1 \otimes S_{w}) J(zq^{\mathscr{L}-\mathscr{L}_w})^{-1}\right)\big|_{z \to z \hbar^{\kappa}}
$$
So 
$$
\lim_{q\to \infty} \textbf{B}_{w}(z q^{\mathscr{L}})=1
$$
\end{proof}

\subsection{Representation-theoretic description}
Now we fix a particular quiver variety $X=\mathcal{M}(\dv,\dw)$. 

\begin{Definition}
Let $\nabla\subset \Pic(X)\otimes\mathbb{R}$ be an alcove and let $\mathscr{L} \in \Pic(X)$. Choose a path from $\nabla$ to $\nabla-\mathscr{L}$ in $\Pic(X)\otimes\mathbb{R}$. Let $w_1,\ldots, w_k$ be the ordered sequence of walls that this path crosses. If $w_i$ is crossed negatively, let $m_i=1$ and otherwise let $m_i=-1$. We say that $\{(w_i,m_i)\}_{i=1}^{k}$ is a sequence of oriented walls from $s$ to $s-\mathscr{L}$. Furthermore, we define
$$
\textbf{B}^{\nabla}_{\mathscr{L}}(z):=\textbf{B}_{w_n}(z)^{m_k} \ldots \textbf{B}_{w_1}(z)^{m_1}
$$
\end{Definition}

One of the main results of \cite{OS} is that this operator is well-defined. Namely, it does not depend on the choice of path. Recall from Proposition \ref{qde} that the capping operator $\Psi(z)$ of $X$ satisfies the equation
$$
\Psi(z q^{\mathscr{L}}) \mathscr{L}=\textbf{M}_{\mathscr{L}}(z) \Psi(z)
$$
for any $\mathscr{L} \in \Pic(X)$.

The main result of \cite{OS} is the following.
\begin{Theorem}[\cite{OS} Theorem 5]\label{qdeOSthm}
Let $\nabla_{0} \in \Pic(X)\otimes\mathbb{R}$ be the alcove containing small ample line bundles. For any $\mathscr{L} \in \Pic(X)$, 
\begin{equation}\label{qdeOS}
 \textbf{M}_{\mathscr{L}}(z)= \const_{X}  \mathscr{L} \textbf{B}^{-\nabla_{0}}_{\mathscr{L}}(z) \in \text{End}(K_{\bT \times \mathbb{C}^{\times}_{q}}(X))[[z]]
\end{equation}
where $\const_{X}\in K_{\bT \times \mathbb{C}^{\times}_{q}}(pt)[[z]]$ is a constant depending on $X$.
\end{Theorem}

\begin{Remark}
This is equivalent to Theorem 5 of \cite{OS}. In \cite{OS}, however, this result is written slightly differently to emphasize the compatibility of $\mathscr{L} \textbf{B}^{-\nabla_0}_{\mathscr{L}}(z)$ with the qKZ equations. Theorem \ref{qdeOSthm} is equivalent to Theorem 5 of \cite{OS} by the relationship between stable envelopes and coproducts expressed in Proposition \ref{coprod}.
\end{Remark}

\begin{Remark}
The form of the constant $\const_{X}$ is unknown. It will show up as a simple prefactor in our formulas below.
\end{Remark}

For later use, we record the following lemma.

\begin{Lemma}\label{limconst}
The constant $\const_{X}$ satisfies
$$
\const_{X}=1+\sum_{\substack{d \\ \langle d, \theta \rangle >0}} \const_{X,d} z^{d}
$$
and
$$
\lim_{q\to 1} \const_{X}=1
$$
\end{Lemma}
\begin{proof}
By (\ref{qdeOS}) and Proposition \ref{Bdegrees}, the only degrees that can show up in $\const_{X}$ satisfy the condition of the lemma. This proves the first part.

Let $\mathscr{L}$ be an ample line bundle. Then
$$
\textbf{B}^{-\nabla_0}_{\mathscr{L}}(z)=\textbf{B}_{w_n}(z)\ldots \textbf{B}_{w_1}(z)
$$
where $w_1,\ldots, w_n$ is the ordered sequence of walls crossed on some path from $-\nabla_{0}$ to $-\nabla_{0}-\mathscr{L}$ that crosses all walls in the negative direction. The powers are all positive because the walls are crossed in the negative direction. By the proof of Proposition \ref{Blim},
$$
\lim_{q\to 0} \textbf{B}_{w}(z q^{\mathscr{L}})=1
$$
if $\mathscr{L}-\mathscr{L}_w$ is ample. Here, each $-\mathscr{L}_{w_i}$ is ample, and thus
$$
\lim_{q\to 0} \textbf{B}_{w_i}(z)=1
$$
Corollary 8.1.10 of \cite{pcmilect} reads $\lim_{q\to 0} \textbf{M}_{\mathscr{L}}(z)=\mathscr{L}$. So taking the $q\to 0$ limit in (\ref{qdeOS}) gives
$$
\lim_{q\to 0} \const_{X}=1
$$
\end{proof}

\section{Exotic quantum difference equations}

Now we turn to our exotic quantum difference equations. In this section, will define them and study their solutions.

\subsection{Definition and basic properties}

As we have set up our notation, there is an obvious generalization of (\ref{qdeOS}).

\begin{Definition}\label{exoticqde1}
Let $\nabla\subset \Pic(X)\otimes \mathbb{R}$ be an alcove. The exotic $q$-difference system for the alcove $\nabla$ is the system of equations
\begin{equation}\label{exoticqde2}
\Psi^{\nabla}(z q^{\mathscr{L}}) \mathscr{L} = \const_{X}  \mathscr{L} \textbf{B}^{\nabla}_{\mathscr{L}}(z)  \Psi^{\nabla}(z), \quad \Psi^{\nabla}(0)=1
\end{equation}
for $\mathscr{L} \in \Pic(X)$. We consider this as an equation for $\Psi^{\nabla}(z) \in \text{End}(K_{\bT\times\mathbb{C}^{\times}_{q}}(X)_{loc})[[z]]$.
\end{Definition}

We remark that the constant $\const_{X}$ is assumed to be the same constant present in (\ref{qdeOS}). Throughout, we will refer to $\Psi^{\nabla}(z)$ that solves (\ref{exoticqde2}) as an exotic solution for the alcove $\nabla$.

\begin{Proposition}
Solutions of (\ref{exoticqde2}) are unique.
\end{Proposition}
\begin{proof}
Choose a line bundle $\mathscr{L}$, and we write
$$
\const_{X} \mathscr{L} \textbf{B}^{\nabla}_{\mathscr{L}}(z)=\mathscr{L}+\sum_{\substack{d \\ \langle d, \theta \rangle>0}} A_{d} z^{d}
$$
where $\theta$ is the stability condition. The degrees of the expansion are as above because of Proposition \ref{Bdegrees} and Lemma \ref{limconst}. Along with (\ref{qdif}), this implies that
$$
   \Psi^{\nabla}(z)=1+\sum_{\substack{d \\ \langle d, \theta \rangle > 0}} \Psi^{\nabla}_{d} z^d 
$$

Then (\ref{exoticqde2}) is equivalent to the system of equations
$$
q^{\langle \mathscr{L},d\rangle} \Psi^{\nabla}_{d} \mathscr{L}=\Psi^{\nabla}_{d} + \sum_{\substack{e+f=d \\ f \neq d \\ \langle e ,\theta \rangle >0 \\ \langle f, \theta \rangle >0}} A_{e} \Psi^{\nabla}_{f}, \quad \Psi^{\nabla}_{0}=1
$$
for $ d \in \mathbb{Z}^{I}$ with $ \langle d, \theta \rangle >0$. By induction, we can assume that $\Psi^{\nabla}_{f}$ in the right hand side has been determined. Assuming now that $\mathscr{L}$ is ample, we have that $\langle \mathscr{L},d\rangle \neq 0$, and hence the operator $q^{\langle \mathscr{L},d\rangle} \mathscr{L}-1$ is invertible. So $\Psi^{\nabla}_{d}$, and hence $\Psi^{\nabla}(z)$, is uniquely determined. 
\end{proof}

\begin{Proposition}\label{adjacent}
Let $\nabla$ and $\nabla'$ be two alcoves connected by a wall $w$ such that moving from $\nabla'$ to $\nabla$ crosses $w$ negatively. Assume that there exists a solution $\Psi^{\nabla'}(z)$ of the exotic difference equation for the alcove $\nabla'$. Then there exists a solution of the exotic difference equation for alcove $\nabla$ and
$$
\Psi^{\nabla}(z)=\textbf{B}_{w}(z) \Psi^{\nabla'}(z)
$$
\end{Proposition}
\begin{proof}
By definition, $\Psi^{\nabla'}(z)$ satisfies
$$
\Psi^{\nabla'}(z q^{\mathscr{L}}) \mathscr{L}=\const_{X} \cdot \mathscr{L} \textbf{B}^{\nabla'}_{\mathscr{L}}(z) \Psi^{\nabla'}(z)
$$
Let $\{(w_i,m_i)\}_{i=1}^{n}$ be an oriented sequence of walls from $\nabla'$ to $\nabla'-\mathscr{L}$ such that $w_1=w$ and hence $m_1=1$. Then
$$
\textbf{B}^{\nabla'}_{\mathscr{L}}(z)= \textbf{B}_{w_n}(z)^{m_n} \ldots \textbf{B}_{w_2}(z)^{m_2} \textbf{B}_{w}(z)
$$
So
\begin{align*}
\textbf{B}_{w}(z q^{\mathscr{L}}) \Psi^{\nabla'}(z q^{\mathscr{L}})  \mathscr{L}
&=\const_{X}  \textbf{B}_{w}(z q^{\mathscr{L}}) \textbf{B}^{\nabla'}_{\mathscr{L}}(z) \Psi^{\nabla'}(z) \\
&=\const_{X}  \textbf{B}_{w}(z q^{\mathscr{L}}) \mathscr{L} \textbf{B}_{w_n}(z)^{m_n} \ldots \textbf{B}_{w}(z) \Psi^{\nabla'}(z) \\
&= \const_{X}  \mathscr{L}\textbf{B}_{w-\mathscr{L}}(z)\textbf{B}_{w_n}(z)^{m_n} \ldots \textbf{B}_{w_2}(z)^{m_2} \textbf{B}_{w}(z) \Psi^{\nabla'}(z)\\
&= \const_{X}  \mathscr{L} \textbf{B}^{\nabla}_{\mathscr{L}}(z) \textbf{B}_{w}(z) \Psi^{\nabla'}(z)
\end{align*}
where the first equality is from the $q$-difference equation, the third is from Proposition \ref{conjugateB}, and the second and fourth are from the definition of $\textbf{B}^{\nabla'}_{\mathscr{L}}(z)$ and $\textbf{B}^{\nabla'}_{\mathscr{L}}(z)$.

So $\textbf{B}_{w}(z) \Psi^{\nabla'}(z)$ satisfies the same $q$-difference equation as $\Psi^{\nabla}(z)$. By uniqueness, they must be equal.
\end{proof}

Repeatedly applying this proposition, we can relate any exotic solution to the solution of the usual $q$-difference equation.

\begin{Corollary}\label{prop3}
Let $\nabla$ be an alcove and let $\{(w_i,m_i)\}_{i=1}^{n}$ be an oriented sequence of walls from $-\nabla_{0}$ to $\nabla$. Then
$$
\Psi^{\nabla}(z)=\textbf{B}_{w_n}(z)^{m_n} \ldots \textbf{B}_{w_1}(z)^{m_1} \Psi(z)
$$
where $\Psi(z)$ is the capping operator.
\end{Corollary}

As a result, we deduce the following.
\begin{Corollary}
Solution of (\ref{exoticqde2}) exist for any alcove $\nabla$.
\end{Corollary}

For later use, we record the following lemma.
\begin{Lemma}\label{psilem}
Let $\nabla$ be an alcove and let $w$ be a wall such that $\mathscr{L}_w+\epsilon \in \nabla$, for some small ample slope $\epsilon$. Then the limit
$$
\lim_{q\to 0} \Psi^{\nabla}(z q^{\mathscr{L}_w})
$$
exists
\end{Lemma}
\begin{proof}
This can be proven by considering the defining $q$-difference system for $\Psi^{\nabla}(z)$. Let $\mathscr{L}$ be an ample line bundle, and consider
\begin{equation}\label{qdif}
\Psi^{\nabla}(z q^{\mathscr{L}}) \mathscr{L}= \const_{X} \mathscr{L} \textbf{B}^{\nabla}_{\mathscr{L}}(z) \Psi^{\nabla}(z)
\end{equation}
Write
$$
   \const_{X} \textbf{B}^{\nabla}_{\mathscr{L}}(z)=1+\sum_{\substack{d \\ \langle d, \theta \rangle > 0}} B_{d} z^{d}
$$
The two sum runs over degrees pairing positively with the stability condition $\theta$ because of Proposition \ref{Bdegrees} and Lemma \ref{limconst}. By (\ref{qdif}), we have
$$
   \Psi^{\nabla}(z)=1+\sum_{\substack{d \\ \langle d, \theta \rangle > 0}} \Psi^{\nabla}_{d} z^d 
$$

Also, (\ref{qdif}) is equivalent to the system
\begin{equation}\label{sys}
q^{\langle \mathscr{L}_w,d\rangle}\left(\Psi^{\nabla}_{d} \mathscr{L} q^{\langle\mathscr{L},d\rangle}- \mathscr{L} \Psi^{\nabla}_{d}\right) =  \sum_{\substack{e+f=d \\ f\neq d \\  \langle e ,\theta \rangle >0 \\ \langle f, \theta \rangle >0}} B_{e} q^{\langle \mathscr{L}_w,e\rangle} \Psi^{\nabla}_{f} q^{\langle \mathscr{L}_w,f\rangle}
\end{equation}
Inductively, we can assume that $\lim_{q\to 0}\Psi^{\nabla}_{f} q^{\langle \mathscr{L}_w,f\rangle}$ exists, since it is obviously true for $\Psi^{\nabla}_{0}=1$. From Proposition \ref{Blim}, Lemma \ref{limconst}, and since $\mathscr{L}$ is assumed to be ample, we have that $\lim_{q\to 0} B_{d} q^{\langle \mathscr{L}_w,d\rangle}$ exists. So we deduce that the $q\to 0$ limit of the right hand side of (\ref{sys}) exists. Hence the same is true of the left hand side. 

Consider the restriction left hand side of (\ref{sys}) to the fixed locus $X^{\bT} \subset X$. 
Since $\mathscr{L}$ is diagonal after restriction to the $\bT$-fixed locus, it does not affect the the order of the pole at $q=0$. Hence $\Psi^{\nabla}_{d} \mathscr{L}$ and $ \mathscr{L} \Psi^{\nabla}_{d}$ have poles of the same order at $q=0$. Since $\langle\mathscr{L},d\rangle>0$ for $d\neq 0$, the terms $\Psi^{\nabla}_{d} \mathscr{L} q^{\langle\mathscr{L},d\rangle}$ and $ \mathscr{L} \Psi^{\nabla}_{d}$ have different orders of poles at $q=0$, and the latter has a strictly worse pole. Hence there is no possibility of cancellation between these two terms, and so the $q\to 0$ limit of the left hand side of (\ref{sys}) exists only if $\lim_{q\to 0} q^{\langle \mathscr{L}_w,d\rangle} \Psi^{\nabla}_{d}$ exists. Thus we conclude that $\lim_{q\to 0} \Psi^{\nabla}(z q^{\mathscr{L}_w})$ exists.

\end{proof}

\begin{Example}\label{hilbqdes}
Consider $X=\text{Hilb}^{3}(\mathbb{C}^2)$, see Chapter \ref{CH2}. In this case, $\Pic(X)\cong \mathbb{Z}$ is generated by the tautological bundle $\mathscr{L}$, and $\Pic(X) \otimes \mathbb{R} \cong \mathbb{R}$. The walls are given by rational numbers with denominator at most $3$. The usual $q$-difference equation corresponds to the alcove $(-1/3,0)\subset \Pic(X)\otimes \mathbb{R}$. The walls crossed between $(-1/3,0)$ and $(-1/3-1,-1)$ are (in order) $-1/3$, $-1/2$, $-2/3$, and $-1$. So the usual $q$-difference equation reads
$$
\Psi(z q) \mathscr{L}= \mathscr{L} \textbf{B}_{-1}(z) \textbf{B}_{-2/3}(z) \textbf{B}_{-1/2}(z) \textbf{B}_{-1/3}(z) \Psi(z)
$$
Similarly, the exotic $q$-difference equation for the alcove $(1/3,1/2)$ is given by
$$
\Psi(z q) \mathscr{L}= \mathscr{L}  \textbf{B}_{-1/2}(z) \textbf{B}_{-1/3}(z) \textbf{B}_{0}(z) \textbf{B}_{1/3}(z) \Psi(z)
$$
See Figure \ref{qdeex} for an illustration.
\end{Example}

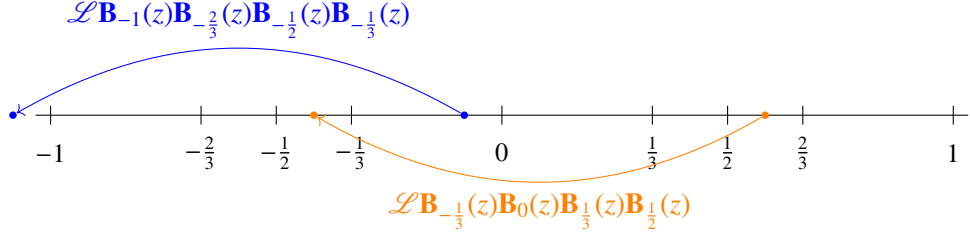
\begin{figure}
\centering
     \begin{tikzpicture}[scale=1,roundnode/.style={circle,fill,inner sep=1pt},roundnode2/.style={circle,fill,inner sep=0.8pt}]
     \draw (-6.2,0)--(6.2,0);
     \draw (2,-1/8)--(2,1/8);
     \node at (2,-1/2){\small $\frac{1}{3}$};
        \draw (3,-1/8)--(3,1/8);
     \node at (3,-1/2){\small $\frac{1}{2}$};
      \draw (4,-1/8)--(4,1/8);
     \node at (4,-1/2){\small $\frac{2}{3}$};
      \draw (6,-1/8)--(6,1/8);
     \node at (6,-1/2){\small $1$};
      \draw (0,-1/8)--(0,1/8);
     \node at (0,-1/2){\small $0$};
        \draw (-2,-1/8)--(-2,1/8);
     \node at (-2,-1/2){\small $-\frac{1}{3}$};
        \draw (-3,-1/8)--(-3,1/8);
     \node at (-3,-1/2){\small $-\frac{1}{2}$};
        \draw (-4,-1/8)--(-4,1/8);
     \node at (-4,-1/2){\small $-\frac{2}{3}$};
        \draw (-6,-1/8)--(-6,1/8);
     \node at (-6,-1/2){\small $-1$};
     
     \node[roundnode,blue](n1) at (-1/2,0){};
     \node[roundnode,blue](n2) at (-6-1/2,0){};
         \draw[->,blue] (n1) to [bend right] node[above] {$\mathscr{L} \textbf{B}_{-1}(z) \textbf{B}_{-\frac{2}{3}}(z) \textbf{B}_{-\frac{1}{2}}(z) \textbf{B}_{-\frac{1}{3}}(z)$} (n2);
          
      \node[roundnode,orange](n7) at (3.5,0){};
     \node[roundnode,orange](n8) at (-6+3.5,0){};
     \draw[->,orange] (n7) to [bend left] node[below] {$\mathscr{L}  \textbf{B}_{-\frac{1}{3}}(z) \textbf{B}_{0}(z)\textbf{B}_{\frac{1}{3}}(z) \textbf{B}_{\frac{1}{2}}(z)$} (n8);
\end{tikzpicture}  
\label{qdeex}
\end{figure}

\subsection{Descendant insertions}\label{stabdes}
Our ultimate goal is to describe solutions to the exotic difference equations for a quiver variety $X=\mathcal{M}(\dv,\dw)$ using the quasimap counts of Section \ref{genfuncs}. Recall the notations $\mathfrak{X}$ and $\mathscr{Z}(\dv,\dw)$ from Section \ref{nakdef}. By precomposing the vertex with descendants $\ver$ with the pullback of the inclusion $\mathfrak{X} \subset \mathfrak{R}:=[T^*\text{Rep}(\dv,\dw)/G_{\dv}]$, we can view the vertex with descendants as an operator with domain $K_{\bT}(\mathfrak{R})$. We will define a map
\begin{align*}
    \textbf{f}^{\nabla}: K_{\bT}(X) &\to K_{\bT}(\mathfrak{R}) \\
    \alpha \mapsto \textbf{f}^{\nabla}_{\alpha}
\end{align*}
depending on an alcove $\nabla$. This construction is the same as the one given in \cite{OkBethe}, but we will use it for arbitrary alcoves, rather than just the small ample alcove. We repeat the construction here for completeness.

Fix a quiver $Q$ and a choice of stability condition $\theta$. For the remainder of this subsection, all quiver varieties correspond to $Q$ and $\theta$. Let $\bT \subset \text{Aut}(\mathcal{M}(\dv,\dw))$ be a torus containing the framing torus. We assume that $\bT$ scales the symplectic form with weight $\hbar$, and we write $\bA=\ker(\hbar)\subset\bT$ for the subtorus preserving the symplectic form.

As a GIT quotient, a quiver variety is naturally an open subset of the stack quotient:
$$
\mathcal{M}(\dv,\dw)=[\mathscr{Z}(\dv,\dw)^{\theta-ss}/G_{\dv}] \subset [\mathscr{Z}(\dv,\dw)/G_{\dv}] =: \mathfrak{X}
$$

We will consider quiver varieties constructed with framing dimension $\dw+\dv$. Let $V_{i}'$ be a vector space with $\dim_{\mathbb{C}} V_i'= \dv_i$ so that the framing spaces used in such a construction are of the form $W_{i} \oplus V_{i}'$. Although they are isomorphic as vector spaces, we distinguish $V_i$ and $V_i'$. Let $G'_{\dv}=\prod_{i \in I} GL(V_i')$. The group $\bT \times G_{\dv}'$ acts on $\mathcal{M}(\dv,\dw+\dv)$.

As shown in \cite{OkBethe}, there exists an embedding
\begin{equation}\label{embed}
\iota: \mathfrak{R} \hookrightarrow [(\mathscr{Z}(\dv,\dw+\dv)_{\text{iso}}/G')/G] = [\mathcal{M}(\dv,\dw+\dv)_{\text{iso}}/G'_{\dv}] \hookrightarrow [\mathcal{M}(\dv,\dw+\dv)/G'_{\dv}]
\end{equation}
where the subscript $\text{iso}$ denotes the locus of points where $V_i'\to V_i$ is an isomorphism.

Let $\mathsf{U} =\mathbb{C}^{\times}$ be the torus which acts on the vector spaces $V_i'$ with weight $u$. It induces an action on $\mathcal{M}(\dv,\dw+\dv)$. The disjoint union $\mathcal{M}(\dv,\dw) \bigsqcup \mathcal{M}(\dv,\dv)$ lies in the fixed locus $\mathcal{M}(\dv,\dw+\dv)^{\mathsf{U}}$. We will use the stable envelopes of the $\mathsf{U}$-action, which means that we must specify a chamber of $\mathsf{U}$, a polarization of $\mathcal{M}(\dv,\dw+\dv)$, and a slope $s \in \Pic(\mathcal{M}(\dv,\dw+\dv))\otimes_{\mathbb{Z}} \mathbb{R}$.

We choose the chamber $\mathfrak{C}$ of $\mathsf{U}$ such that $\mathcal{M}(\dv,\dv)$ lies in the full attracting set of $\mathcal{M}(\dv,\dw)$. This is equivalent to choosing $u$ to be an attracting weight.

Assume that a polarization $T^{1/2}_{\mathcal{M}(\dv,\dw)}$ of $\mathcal{M}(\dv,\dw)$ has been chosen. Then we choose the polarization of $\mathcal{M}(\dv,\dv+\dw)$ given by
$$
T^{1/2}_{\mathcal{M}(\dv,\dw+\dv)} = T^{1/2}_{\mathcal{M}(\dv,\dw)}+ \hbar \sum_{i \in I} \text{Hom}(V_i,V_i')
$$

Similarly, we assume a generic choice of slope $s \in \Pic(\mathcal{M}(\dv,\dw))\otimes_{\mathbb{Z}}\mathbb{R}$ has been made. We write $\nabla$ for the alcove containing $s$. By Kirwan surjectivity \cite{kirv}, we can assume that $s$ is a real multiple of a product of the tautological line bundles of $\mathcal{M}(\dv,\dw)$. This gives a perfectly good real line bundle on $\mathcal{M}(\dv,\dw+\dv)$, and we perturb $s$ slightly so that it is a generic slope for both varieties. This is possible since there are only countably many affine hyperplanes to be avoided.

Since the action of $\bT \times G_{\dv}'$ on $\mathcal{M}(\dv,\dw+\dv)$ commutes with the action of $\mathsf{U}$, the stable envelope corresponding to this choice of chamber, polarization, and slope is a map
\begin{equation}\label{stabu}
\stab_{\mathsf{U},\nabla}: K_{\bT \times G_{\dv}'}(\mathcal{M}(\dv,\dw+\dv)^{\mathsf{U}}) \to K_{\bT \times G_{\dv}'}(\mathcal{M}(\dv,\dw+\dv))
\end{equation}
We denote the restriction of this map to $K_{\bT \times G_{\dv}'}(\mathcal{M}(\dv,\dw))$ by the same notation.

Since
$$
K_{\bT \times G'_{\dv}}(\mathcal{M}(\dv,\dw+\dv)) \cong K_{\bT}([\mathcal{M}(\dv,\dw+\dv)/G'_{\dv}])
$$
we can use the embedding $\iota$ from (\ref{embed}) to pullback $\stab_{\mathsf{U},\nabla}(\alpha)$ to a class on $\mathfrak{R}$. Since $G_{\dv}'$ acts trivially on $\mathcal{M}(\dv,\dw)$, we have an inclusion 
$$
K_{\bT}(\mathcal{M}(\dv,\dw))\hookrightarrow K_{\bT}(\mathcal{M}(\dv,\dw))\otimes K_{G_{\dv}'}(pt)= K_{\bT \times G_{\dv}'}(\mathcal{M}(\dv,\dw))
$$
\begin{Definition}
For $\alpha \in K_{\bT}(\mathcal{M}(\dv,\dw))\subset K_{\bT\times G_{\dv}'}(\mathcal{M}(\dv,\dw))$, let
\begin{equation}\label{stabdesc}
\textbf{s}^{\nabla}_{\alpha}:=\iota^*(\stab_{\mathsf{U},\nabla}(\alpha)) \in K_{\bT}(\mathfrak{R})
\end{equation}
\end{Definition}

\begin{Proposition}[\cite{OkBethe} Proposition 1]
The class $\textbf{s}_{\alpha}$ is supported on $\mathfrak{X} \subset \mathfrak{R}$. 
\end{Proposition}

We have
$$
\textbf{s}^{\nabla}_{\alpha}\in K_{\bT}(\mathfrak{R}) = K_{\bT}([T^*\text{Rep}(\dv,\dw)/G_{\dv}])= K_{\bT \times G_{\dv}}(T^*\text{Rep}(\dv,\dw))
$$
The origin of $T^*\text{Rep}(\dv,\dw)$ gives a distinguished point $0 \in \mathfrak{R}$, and $\textbf{s}^{\nabla}_{\alpha}$ is determined by its restriction to $0$. We will need a bound on the $G_{\dv}$ weights of $\textbf{s}^{\nabla}_{\alpha}\big|_{0} \in K_{\bT \times G_{\dv}}(pt)$.

The map $\iota$ is induced by a map on the prequotient with respect to $G_{\dv}$. The origin in $T^*\text{Rep}(\dv,\dw)$ maps to the point of $\mathscr{Z}(\dv,\dw+\dv)_{\text{iso}}/G_{\dv}'$ whose representative in $\mathscr{Z}(\dv,\dw+\dv)$ is given by the tuple of linear maps such that $V' \to V$ is an isomorphism and all other maps are zero. This provides an isomorphism $G_{\dv}\cong G_{\dv'}$. This quiver data naturally represents a point of $\mathcal{M}(\dv,\dv+\dw)$ fixed by $G_{\dv}'$. We denote this point by $\star \in \mathcal{M}(\dv,\dv+\dw)^{G_{\dv}'}$. Bounding the $G_{\dv}$ weights of $\textbf{s}^{\nabla}_{\alpha}\big|_{0}$ is equivalent to bounding the $G_{\dv}'$ weights of $\stab_{\mathsf{U},\nabla}(\alpha)|_{\star}$, which leads to the following proposition.

\begin{Proposition}[\cite{OkBethe} Proposition 2]\label{window}
The $G_{\dv}$-weights of $\textbf{s}^{\nabla}_{\alpha}\big|_0$ are contained in 
$$
\mathscr{L}+\text{convex hull}\left(\text{weights of } \Lambda^{\bullet}\left(T^{1/2}_{\mathcal{M}(\dv,\dw+\dv)} \right)_{\star}^{*}\right) \in \text{char}(G_{\dv})\otimes \mathbb{R}
$$
for any $\mathscr{L} \in \nabla$.
\end{Proposition}

\begin{Definition}
Let 
\begin{equation}\label{desc}
\textbf{f}^{\nabla}_{\alpha}=\Delta_{\hbar}^{-1} \textbf{s}^{\nabla}_{\alpha} \in K_{\bT}(\mathfrak{R})_{loc}
\end{equation}
where
$$
\Delta_{\hbar}=\Lambda^{\bullet} \left( \hbar \bigoplus_{i \in I}\text{Hom}(V_i,V_i)\right)
$$
Similarly, let
\begin{equation}\label{descopp}
\textbf{f}^{\nabla}_{\text{opp},\alpha}=\Delta_{\hbar}^{-1} \textbf{s}^{\nabla}_{\text{opp},\alpha} \in K_{\bT}(\mathfrak{R})_{loc}
\end{equation}
where the subscript $\text{opp}$ means to use the same construction as for $\textbf{f}^{\nabla}_{\alpha}$, but using $T^{1/2}_{\text{opp},\mathcal{M}(\dv,\dw)}$ and $-\nabla$ instead.
\end{Definition}

\begin{Proposition}[\cite{OkBethe} Definition 2, Proposition 3]\label{intdesc}
The insertion of the class $\textbf{f}^{\nabla}_{\alpha} \in K_{\bT}(\mathfrak{R})_{loc}$ into the capped vertex with descendants lies in integral $K$-theory. More precisely, 
$$
\widehat{\ver}(\textbf{f}^{\nabla}_{\alpha}) \in K_{\bT \times \mathbb{C}^{\times}_{q}}(X)[[z]]
$$
The same holds for insertion of the class $\textbf{f}^{\nabla}_{\text{opp},\alpha}$.
\end{Proposition}

\subsection{Solutions of exotic difference equations}
We will not be changing the dimension and framing dimension for the remainder of this chapter. So let $X=\mathcal{M}(\dv,\dw)$. Fix a polarization $T^{1/2}$ of $X$. Recall that for each choice of alcove $\nabla$ we have defined a map
\begin{align*}
    K_{\bT}(X) &\to K_{\bT}(\mathfrak{X})_{loc} \\
  \alpha  & \mapsto \textbf{f}^{\nabla}_{\alpha}
\end{align*}
using stable envelopes for the alcove $\nabla$ and the polarization $T^{1/2}$. Although $\textbf{f}^{\nabla}_{\alpha}$ lies in localized $K$-theory, when used as a descendant in a quasimap count, it behaves as an integral descendant insertion.

\begin{Definition}
Let $\ver^{\nabla}$ be the map
\begin{align*}
\ver^{\nabla}: K_{\bT}(X) &\to K_{\bT \times \mathbb{C}^{\times}_{q}}(X)_{loc}[[z]] \\
\alpha &\mapsto \ver(\textbf{f}^{\nabla}_{\alpha})
\end{align*}
where $\ver$ is the vertex where the polarization in (\ref{symvrs}) is the opposite of the polarization $T^{1/2}$ used for stable envelopes in $\textbf{f}^{\nabla}_{\alpha}$.
\end{Definition}






The main result of \cite{OkBethe}, which inspired much of this thesis, is as follows.

\begin{Proposition}[\cite{OkBethe}]
Let $\nabla_{0}$ be the alcove containing small ample slopes. Let $\Omega$ be the operator
\begin{align*}
\Omega: K_{\bT}(X) &\to  K_{\bT \times\mathbb{C}^{\times}_{q}}(X)_{loc}[[z]] \\
\alpha &\mapsto \sum_{d} \text{ev}_{\infty,*}\left(\qm_{\ns \, \infty}, \text{ev}_{0}^{*}(\textbf{f}^{\nabla_{0}}_{\text{opp},\alpha}) \otimes \vrs^{d} \right) z^d
\end{align*}
Then
$$
\Omega(\alpha)=\sum_{d}\text{ev}_{\infty,*}\left(\qm_{\substack{\rel \, 0\\ \ns \, \infty}}, \widehat{\text{ev}}_{0}^{*}(\alpha)\otimes \vrs^{d} \right) z^{d}
$$
Equivalently, $\Omega$ is the adjoint of $\Psi$:
$$
\Omega^{\chi}=\Psi 
$$
of Definition \ref{adjoint}.

\end{Proposition}
We will give a similar identification of the solutions to the exotic difference equations.

\begin{Lemma}\label{integral}
Let $\Psi^{\nabla}$ be the solution of the exotic difference equation for an alcove $\nabla$. Let $\nabla'$ be another alcove. For any $\alpha \in K_{\bT}(X)$, we have
$$
(\Psi^{\nabla} \circ \ver^{\nabla'})(\alpha) \in K_{\bT \times \mathbb{C}^{\times}_{q}}(X)[[z]]
$$
i.e. it lives in integral $K$-theory.
\end{Lemma}
\begin{proof}
Choose a path from $-\nabla_{0}$ to $\nabla$ that crosses the oriented sequence of walls $\{(w_i,m_i)\}_{i=1}^{n}$. Then 
$$
\Psi^{\nabla}(z)=\textbf{B}_{w_n}(z)^{m_n} \ldots \textbf{B}_{w_1}(z)^{m_1} \Psi(z)
$$
by Corollary \ref{prop3}.

By (\ref{cpv}) and Proposition \ref{intdesc}, $(\Psi \circ \ver^{\nabla'})(\alpha)=\widehat{\ver}(\textbf{f}^{\nabla}_{\alpha}) \in K_{\bT\times\mathbb{C}^{\times}_{q}}(X)[[z]]$, and since the wall crossing operators act on the nonlocalized $K$-theory, the result follows.
\end{proof}

Our first main result is the following, which relates solutions of the exotic $q$-difference equations to quasimap counts.
\begin{Theorem}\label{mainthm1}
Let $\nabla\subset \Pic(X)\otimes \mathbb{R}$ be an alcove. Then $\Psi^{\nabla} \circ \ver^{\nabla}=1$.
\end{Theorem}
\begin{proof}
We split the proof into several steps. First, we assume that $\nabla$ is contained in the negative ample cone. Let $w$ be a wall such that $\nabla$ is on the positive side of $w$. Let $\mathscr{L}_w$ be a line bundle on the wall $w$ and let $\epsilon$ be a small ample slope such that $\mathscr{L}_w+\epsilon \in \nabla$. We know that $\Psi^{\nabla} \circ \ver^{\nabla}$ takes values in Laurent polynomials in $q$. Hence $(\Psi^{\nabla} \circ \ver^{\nabla})|_{z=zq^{\mathscr{L}_w}}$ is also a Laurent polynomial in $q$. We will show that is does not depend on $q$ by showing that its limits exist as $q^{\pm 1} \to 0$.

We first consider $\ver^{\nabla}(\alpha)$. By localization and Proposition \ref{locterms},
\begin{align*}
\ver^{\nabla}(\alpha)&= \sum_{d} \text{ev}_{0,*}\left(\qm^{d}_{\ns \, 0}, \text{ev}_{\infty}^{*}(\textbf{f}^{\nabla}_{\alpha}) \otimes \vrs^{d} \right) z^{d} \\
&=\sum_{d} z^{d} \sum_{g \in \left(\qm^{d}_{\ns \,  0}\right)^{\mathbb{C}^{\times}_{q}}} \frac{\text{ev}_{\infty}^{*}( \textbf{f}_{\alpha}^{\nabla})|_{g} R_{g} }{\Lambda^{\bullet}\left(\mathscr{T}^{1/2}_{\infty}\big|_{g}\right)^{*}} \\ 
&=\sum_{d} z^{d} \sum_{g \in \left(\qm^{d}_{\ns \,  0}\right)^{\mathbb{C}^{\times}_{q}}} \frac{\text{ev}_{\infty}^{*}( \textbf{s}_{\alpha}^{\nabla})|_{g} R_{g} }{\Lambda^{\bullet}\left(\mathscr{T}^{1/2}_{\mathcal{M}(\dv,\dw+\dv),\infty}\big|_{g}\right)^{*}}
\end{align*}
where $R_{g}$ has finite limits in $q$.

By Proposition \ref{window}, the $\mathbb{C}^{\times}_{q}$-weights of $\text{ev}_{\infty}^{*}(\textbf{s}^{\nabla}_{\alpha})|_{g}$ are contained in $$
-\langle d,s\rangle+\text{Convex Hull of }\mathbb{C}^{\times}_{q}-\text{weights of } \Lambda^{\bullet}\left(\mathscr{T}^{1/2}_{\mathcal{M}(\dv,\dv+\dw),\infty}\big|_{g}\right)
$$
for any $s \in \nabla$. The sign on the first term is negative because of the first part of Proposition \ref{linearization}.


This implies that the limits
$$
\lim_{q^{\pm}\to \infty} \frac{q^{\langle d,s\rangle}\text{ev}_{\infty}^{*}( \textbf{s}_{\alpha}^{\nabla})|_{g} R_{g} }{\Lambda^{\bullet}\left(\mathscr{T}^{1/2}_{\mathcal{M}(\dv,\dw+\dv),\infty}\big|_{g}\right)^{*}}
$$
exist for any $s \in \nabla$. Let $s=\mathscr{L}_{w}+\epsilon \in \nabla$ and rewrite the above limits as 
$$
\lim_{q^{\pm}\to \infty} \frac{q^{\langle d,\mathscr{L}_w+\epsilon\rangle}\text{ev}_{\infty}^{*}( \textbf{s}_{\alpha}^{\nabla})|_{g} R_{g} }{\Lambda^{\bullet}\left(\mathscr{T}^{1/2}_{\mathcal{M}(\dv,\dw+\dv),\infty}\big|_{g}\right)^{*}} =\lim_{q^{\pm}\to \infty} \frac{q^{\langle d,\mathscr{L}_w\rangle} q^{\langle d,\epsilon\rangle}\text{ev}_{\infty}^{*}( \textbf{s}_{\alpha}^{\nabla})|_{g} R_{g} }{\Lambda^{\bullet}\left(\mathscr{T}^{1/2}_{\mathcal{M}(\dv,\dw+\dv),\infty}\big|_{g}\right)^{*}}
$$
Since $\epsilon$ is ample, $\langle d,\epsilon\rangle \geq 0$, and we deduce that $$
\lim_{q \to 0} \frac{q^{\langle d,\mathscr{L}_w\rangle}\text{ev}_{\infty}^{*}( \textbf{s}_{\alpha}^{\nabla})\big|_{g} R_{g} }{\Lambda^{\bullet}\left(\mathscr{T}^{1/2}_{\mathcal{M}(\dv,\dw+\dv),\infty}\big|_{g}\right)^{*}} 
$$
exists and
$$
\lim_{q \to \infty} \frac{q^{\langle d,\mathscr{L}_w\rangle }\text{ev}_{\infty}^{*}( \textbf{s}_{\alpha}^{\nabla})|_{g} R_{g} }{\Lambda^{\bullet}\left(\mathscr{T}^{1/2}_{\mathcal{M}(\dv,\dw+\dv),\infty}\big|_{g}\right)^{*}}=0
$$
Hence $\lim_{q\to 0} \ver^{\nabla}(\alpha)|_{z \to z q^{\mathscr{L}_w}}$ exists and $\lim_{q\to \infty} \ver^{\nabla}(\alpha)|_{z \to z q^{\mathscr{L}_w}}=\alpha$.

Next we consider $\Psi^{\nabla}$. Let $w_1, \ldots, w_n$ be an ordered sequence of walls crossed from $-\nabla_0$ to $\nabla$. We assume they are all crossed negatively, which is possible since $\nabla$ is in the negative ample cone. By Corollary \ref{prop3}, 
$$
\Psi^{\nabla}(z)= \textbf{B}_{w_n}(z ) \ldots \textbf{B}_{w_1}(z ) \Psi^{-\nabla_0}(z)
$$
\begin{sloppypar}
It is known that $\lim_{q\to \infty} \Psi^{-\nabla_{0}}(z)$ exists, see \cite{pcmilect} Section 7. Since $w$ is anti-ample, it follows that $\lim_{q\to \infty} \Psi^{-\nabla_{0}}(z q^{\mathscr{L}_w})$ also exists, and actually equals $1$. Similarly, since all walls $w_i$ are negatively crossed and since $\mathscr{L}_{w_i}-\mathscr{L}_w$ is ample, Proposition \ref{Blim} implies that 
$$\lim_{q \to \infty} \textbf{B}_{w_n}(z q^{\mathscr{L}_w}) \ldots \textbf{B}_{w_1}(z q^{\mathscr{L}_w} )=1$$
\end{sloppypar}

So $\lim_{q\to \infty} \Psi^{\nabla}(z q^{\mathscr{L}_w})=1$. The last thing to do is show that $\lim_{q\to 0} \Psi^{\nabla}(z q^{\mathscr{L}_w})$ exists. But this is exactly the content of Lemma \ref{psilem}.

Hence the limits $\lim_{q^{\pm}\to 0} (\Psi^{\nabla}\circ \ver^{\nabla})|_{z=zq^{\mathscr{L}_w}}$ exist, and so this quantity is independent of $q$ and is equal to its value as $q\to \infty$. In fact, in the above arguments, we have already computed this to be $1$.

Finally, we need to relax the assumption that $\nabla$ is contained in the negative ample cone. By Proposition \ref{conjugateB}, any wall crossing operator is conjugate to one in the negative ample cone, up to a shift of $z$. Let $\nabla$ be an alcove not contained in the negative ample cone, and let $\mathscr{L}_0 \in \Pic(X)$ be such that $\nabla-\mathscr{L}_0$ is in the negative ample cone. Proposition \ref{conjugateB} implies that
$$
\Psi^{\nabla}(z)=\mathscr{L}_0^{-1}\Psi^{\nabla-\mathscr{L}_0}(z q^{-\mathscr{L}_0})
$$
By the defining properties for stable envelopes, we have
$$
\ver^{\nabla}(z)=\ver^{\nabla-\mathscr{L}_0}(z q^{-\mathscr{L}_0}) \mathscr{L}_0
$$
So
\begin{align*}
\Psi^{\nabla} \circ \ver^{\nabla}&= \mathscr{L}_0^{-1} \Psi^{\nabla-\mathscr{L}_0}(z q^{-\mathscr{L}_0}) \circ \ver^{\nabla-\mathscr{L}_0}(z q^{-\mathscr{L}_0}) \mathscr{L}_0 \\
&= \mathscr{L}_0^{-1} \mathscr{L}_0 \\
&=1
\end{align*}
Hence $\Psi^{\nabla} \circ \ver^{\nabla}=1$ holds for any alcove.
\end{proof}

Our second main theorem, Theorem \ref{mainthm2} below, gives an explicit identification of $\Psi^{\nabla}$ via enumerative geometry. Let $\Omega^{\nabla}: K_{\bT}(X) \to K_{\bT \times \mathbb{C}^{\times}_{q}}(X)_{loc}[[z]]$ be the operator
\begin{equation}\label{omega}
\Omega^{\nabla}(\alpha)=\sum_{d} \text{ev}_{\infty,*}\left(\qm_{\ns \, \infty},\text{ev}_{0}^{*}(\textbf{f}_{\text{opp},\alpha}^{\nabla}) \otimes \vrs \right) z^{d}
\end{equation}
where $\textbf{f}^{\nabla}_{\text{opp},\alpha}$ is from (\ref{descopp}). We assume the polarization used to define the symmetrized virtual structure sheaf is $T^{1/2}_{\text{opp}}$. In contrast to $\ver^{\nabla}$, the polarizations used in the descendant insertion and the quasimap count of $\Omega^{\nabla}$ agree. As it involves pushing forward from a nonsingular point, $\Omega^{\nabla}$ is defined by $\mathbb{C}^{\times}_{q}$-localization. As in (\ref{split}), the data of a $\mathbb{C}^{\times}_{q}$-fixed quasimap gives a decomposition of the bundles $\mathscr{V}_i=\bigoplus_{j} \mathcal{O}(d_{i,j})$ provided by the quasimap. Analogous to Lemma \ref{locterms}, we have the following.


\begin{Lemma}\label{locterms2}
Let $g \in (\qm_{\ns \, \infty})^{\mathbb{C}^{\times}_{q}}$. The contributions of $\vrs$ on $\qm_{\ns \, \infty}$ to localization formulas for $\Omega$ at $g$ take the form
$$
\frac{R'_{g}}{\Lambda^{\bullet}\left(\mathscr{T}^{1/2}_{\text{opp},0}\big|_{g}\right)^{*}}
$$
where
$$
\lim_{q^{\pm 1} \to 0} R'_{g} \quad \text{exists}
$$
\end{Lemma}
\begin{proof}
The proof is similar to that of Lemma \ref{locterms}.
\end{proof}

\begin{Theorem}\label{mainthm2}

As elements of $\text{End}(K_{\bT\times\mathbb{C}^{\times}_{q}}(X)_{loc})[[z]]$, we have
$$
\left(\Omega^{\nabla}\right)^{\chi}= \Psi^{\nabla}
$$
where $\left(\Omega^{\nabla}\right)^{\chi}$ is the adjoint of $\Omega^{\nabla}$ from Definition \ref{adjoint}.
\end{Theorem}

\begin{proof}

A-priori, $\left(\Omega^{\nabla}\right)^{\chi} \circ \ver^{\nabla}:K_{\bT}(X)\to K_{\bT\times\mathbb{C}^{\times}_{q}}(X)_{loc}[[z]]$, i.e. the composition takes values in localized $K$-theory. We first show that it actually takes values in integral $K$-theory. 

We use the degeneration and glue formulas for quasimap counts, see \cite{pcmilect} Section 6.5, to see that 
$$
\Omega^{\nabla}=R \circ G^{-1} \circ S
$$
where
\begin{itemize}
\item $S$ is the operator defined by
$$
S: \alpha \mapsto \sum_{d} \widehat{\text{ev}}_{\infty,*}\left(\qm_{\rel \, \infty} ,\text{ev}_{0}^{*} (\textbf{f}^{\nabla}_{\text{opp},\alpha}) \otimes \vrs \right) z^{d}
$$
    \item $G$ is the glue matrix, which is defined by
    $$
    G: \alpha \mapsto \sum_{d} \widehat{ev}_{\infty,*}\left(\qm_{\substack{\rel \, 0 \\ \rel \, \infty}},\vrs^{d}\right) z^{d}
    $$
    \item $R$ is the operator defined by 
$$
R: \alpha \mapsto \sum_{d} \text{ev}_{\infty,*}\left( \qm_{\substack{\ns \, \infty \\ \rel \, 0}}, \widehat{\text{ev}}_{0}^{*}(\alpha) \otimes \vrs^{d} \right) z^{d}
$$
\end{itemize} 
From  (\ref{capoperator}), the operator $R^{\chi}$ is exactly the capping operator. Taking adjoints, we find 
\begin{align*}
\left(\Omega^{\nabla}\right)^{\chi} \circ \ver^{\nabla}&= S^{\chi} \circ (G^{-1})^{\chi} \circ R^{\chi} \circ \ver^{\nabla} \\
&= S^{\chi} \circ (G^{-1})^{\chi} \circ \widehat{\ver}^{\nabla}
\end{align*}
where the second line follows from (\ref{cpv}). The capped vertex $\widehat{\ver}^{\nabla}$, the operator $S^{\chi}$, and $(G^{-1})^{\chi}$ are all defined via pushforwards from quasimap spaces with relative conditions, and hence all take values in Laurent polynomials in $q$.

It follows that $\left(\Omega^{\nabla}\right)^{\chi} \circ \ver^{\nabla}$ is really a map $K_{\bT}(X) \to K_{\bT \times \mathbb{C}^{\times}_{q}}(X)[[z]]$, i.e. it takes values in non-localized $K$-theory.

By Theorem \ref{mainthm1}, to prove that $\left(\Omega^{\nabla}\right)^{\chi}=\ver^{\nabla}$, it is sufficient to show that 
\begin{equation}\label{op}
\left(\Omega^{\nabla}\right)^{\chi}\circ \ver^{\nabla}=1
\end{equation}
We apply the same strategy as in the proof of Theorem \ref{mainthm1}: we will compute the $q^{\pm 1} \to 0$ limits.

Fix $\alpha \in K_{\bT}(X)$. Let $w$ be a wall such that $\nabla$ is on the positive side of $w$. Let $\mathscr{L}_w$ be a line bundle on $w$. Then $\mathscr{L}_w+\epsilon\in \nabla$, where $\epsilon$ is a small ample slope. In the proof of Theorem \ref{mainthm1}, we showed that $\lim_{q\to 0} \ver^{\nabla}(\alpha)|_{z\to z q^{\mathscr{L}_w}}$ exists and $\lim_{q\to\infty} \ver^{\nabla}(\alpha)|_{z\to z q^{\mathscr{L}_w}}=\alpha$. We apply similar reasoning to $\Omega^{\nabla}$.

By Lemma \ref{locterms2}, the contribution of $\vrs$ on $\qm_{\ns \, \infty}$ to localization formulas at a fixed point $g \in \left(\qm_{\ns \, \infty}^{d}\right)^{\mathbb{C}^{\times}_{q}}$ looks like
$$
\frac{R'_{g}}{\Lambda^{\bullet}\left(\mathscr{T}^{1/2}_{\text{opp},0}\big|_{g}\right)^{*}}
$$
where
$$
\lim_{q^{\pm 1} \to 0} R'_{g} \quad \text{exists}
$$
By Proposition \ref{window}, we have
$$
\lim_{q^{\pm}\to 0} \frac{q^{\langle d,s\rangle} \textbf{s}^{\nabla}_{\text{opp},\alpha} R'_{g}}{\Lambda^{\bullet}\left(\mathscr{T}^{1/2}_{\text{opp},\mathcal{M}(\dv,\dv+\dw),0}\big|_{g}\right)^{*}} \quad \text{exists}
$$
for any slope $s \in \nabla$. Setting $s=\mathscr{L}_w+\epsilon$ gives
$$
\lim_{q^{\pm}\to 0} \frac{q^{\langle d,w\rangle} q^{\langle d,\epsilon\rangle} \textbf{s}^{\nabla}_{\text{opp},\alpha} R'_{g}}{\Lambda^{\bullet}\left(\mathscr{T}^{1/2}_{\text{opp},\mathcal{M}(\dv,\dv+\dw),0}\big|_{g}\right)^{*}} \quad \text{exists}
$$
which implies that $\lim_{q \to 0} \Omega^{\nabla}(\alpha)|_{z\to z q^{\mathscr{L}_{w}}}$ exists and $\lim_{q \to \infty} \Omega^{\nabla}(\alpha)|_{z\to z q^{\mathscr{L}_{w}}}=\alpha$.

As $\left(\Omega^{\nabla}\right)^{\chi}$ is the adjoint of $\Omega^{\nabla}$, we deduce that 
$$
\lim_{q\to 0} \left(\Omega^{\nabla}\right)^{\chi}(\alpha)|_{z\to z q^{\mathscr{L}_{w}}} \quad \text{exists}
$$
and 
$$
\lim_{q\to \infty} \left(\Omega^{\nabla}\right)^{\chi}(\alpha)|_{z\to z q^{\mathscr{L}_w}}=\alpha
$$
Since $\left(\left(\Omega^{\nabla}\right)^{\chi} \circ \ver^{\nabla}\right)(\alpha)$ is a Laurent polynomial in $q$ whose $q^{\pm 1} \to 0$ limits exist, it follows that it is independent of $q$. Sending $q\to \infty$ shows
$$
\left(\Omega^{\nabla}\right)^{\chi}\circ \ver^{\nabla}=1
$$

\end{proof}

\chapter{Stable Envelopes of Type A Quiver Varieties}\label{CH3}

For a quiver variety of finite or affine type $A$, it is possible to write explicit formulas for solutions of exotic $q$-difference equations using Theorem \ref{mainthm2}. For this, we need to know how to calculate the descendants of Definition \ref{desc}, which are defined in terms of stable envelopes. In this section, we review the known formulas for elliptic stable envelopes of the variety $\mathcal{M}(n,r)$ arising from the quiver with one vertex and one loop. This variety is known as the instanton moduli space, and $\mathcal{M}(n,1)=\text{Hilb}^{n}(\mathbb{C}^{2})$. We use these formulas to write formulas for $K$-theoretic stable envelopes of arbitrary slope. This section follows \cite{dinkinselliptic}, which is based on the earlier work of Smirnov \cite{SmirnovElliptic} for the Hilbert scheme of points in the plane.

\section{Combinatorial background}\label{vwtuples}

To state our formula for stable envelopes, we first need various combinatorial constructions.

\subsection{$r$-tuples of Partitions}
Let $\lambda=(\lambda_1\geq \lambda_2 \geq \ldots \geq \lambda_{l}> 0)$ be a partition. We identify it with its Young diagram, which is the set of points
\begin{equation}\label{Yng}
\{(x,y)\in \mathbb{Z}^{2}_{> 0} \, \mid \, 1\leq x \leq l , 1 \leq y \leq \lambda_i\} 
\end{equation}
As is standard, we refer to the points in the Young diagram as ``boxes." We call the box with coordinates $(1,1)$ the corner box. Let $a\in \lambda$ be a box with coordinates $(x,y)$. The content and height of $a$ are defined to be
\begin{align*}
    c_{a}&=x-y \\
    h_{a}&=x+y-2
\end{align*}
We will often think of our Young diagrams as being rotated to the left by $45^{\circ}$. With this convention, the content and height give the horizontal and vertical coordinates, respectively, normalized so that the corner box of $\lambda$ has content 0 and height 0. We denote $|\lambda|=\sum_{j} \lambda_{j}$, and refer to it as the total number of boxes of $\lambda$. See Figure \ref{Yng0} for an example of our conventions.

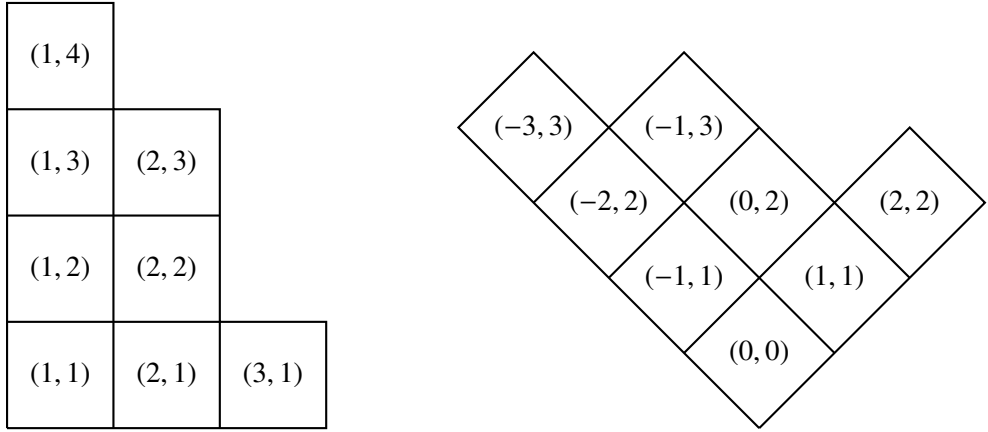
\begin{figure}[htbp]
\centering
   \begin{tikzpicture}
   
   \begin{scope}[shift={(-5,0)},rotate=-45]
\draw[thick,-](0,0)--(-4,4)--(-3,5)--(-2,4)--(-1,5)--(1,3)--(2,4)--(3,3)--(0,0);
\draw[thick,-](1,1)--(-2,4);
\draw[thick,-](2,2)--(1,3);
\draw[thick,-](-1,1)--(1,3);
\draw[thick,-](-2,2)--(0,4);
\draw[thick,-](-3,3)--(-2,4);
\node at (0,1){$(1,1)$};
\node at (1,2){$(2,1)$};
\node at (-1,2){$(1,2)$};
\node at (2,3){$(3,1)$};
\node at (0,3){$(2,2)$};
\node at (-2,3){$(1,3)$};
\node at (-1,4){$(2,3)$};
\node at (-3,4){$(1,4)$};
\end{scope}

   \begin{scope}[shift={(5,0)}]
 \draw[thick,-](0,0)--(-4,4)--(-3,5)--(-2,4)--(-1,5)--(1,3)--(2,4)--(3,3)--(0,0);
\draw[thick,-](1,1)--(-2,4);
\draw[thick,-](2,2)--(1,3);
\draw[thick,-](-1,1)--(1,3);
\draw[thick,-](-2,2)--(0,4);
\draw[thick,-](-3,3)--(-2,4);
\node at (0,1){$(0,0)$};
\node at (1,2){$(1,1)$};
\node at (-1,2){$(-1,1)$};
\node at (2,3){$(2,2)$};
\node at (0,3){$(0,2)$};
\node at (-2,3){$(-2,2)$};
\node at (-1,4){$(-1,3)$};
\node at (-3,4){$(-3,3)$};
  \end{scope}
   \end{tikzpicture}
 
    \caption{Our convention for the coordinates of boxes in a Young diagram. The left figure shows the $(x,y)$-coordinates and the right figure shows the $(c,h)$-coordinates.}
    \label{Yng0}
\end{figure}


The torus $\bT$ that acts on $\mathcal{M}(n,r)$ is the product of the framing torus with $\mathbb{C}^{\times}_{t_1}\times\mathbb{C}^{\times}_{t_2}$, where the latter acts by scaling the linear maps corresponding to the two loops in the doubled Jordan quiver. The $\bT$-fixed points of this variety are indexed by $r$-tuples of partitions $\llambda=(\lambda^{(1)},\ldots,\lambda^{(r)})$ such that $\sum_{i} |\lambda^{i}|=n$, see for instance Section 3.4 of \cite{dinkinselliptic}.



\begin{Remark}
An $r$-tuple of partitions $\llambda=(\lambda^{(1)},\ldots,\lambda^{(r)})$ is sensitive to the order in which the $\lambda^{(i)}$ appear. In other words, changing the order gives a different $\bT$-fixed point.
\end{Remark}

Typically, $r$-tuples of partitions are rotated and depicted above the Jordan quiver in order from bottom to top. See Figure \ref{fig1} for an example of a $\bT$-fixed point.

\begin{figure}[htbp]
    \centering
\begin{tikzpicture}[scale=0.98,roundnode/.style={circle,fill,inner sep=2.5pt},refnode/.style={circle,inner sep=0pt},squarednode/.style={rectangle,fill,inner sep=3pt}] 
\begin{scope}[shift={(0,-1)}]
\node[roundnode,label=above:{15}] (V1) at (0,0) {};
\node[squarednode,label=below:{$3$}] (F1) at (0,-1){};
\draw[thick, ->] (V1) to [out=210,in=120,looseness=8] (V1);
\end{scope}

\begin{scope}[shift={(0,0)}]
\draw[thick,-](0,0)--(-4,4)--(-3,5)--(-2,4)--(-1,5)--(1,3)--(2,4)--(3,3)--(0,0);
\draw[thick,-](1,1)--(-2,4);
\draw[thick,-](2,2)--(1,3);
\draw[thick,-](-1,1)--(1,3);
\draw[thick,-](-2,2)--(0,4);
\draw[thick,-](-3,3)--(-2,4);
\node[] at (0,1){$u_1$};
\node[] at (-1,2){$t_1 u_1$};
\node[] at (-2,3){$t_1^2 u_1$};
\node[] at (-3,4){$t_1^3 u_1$};
\node[] at (1,2){$t_2 u_1$};
\node[] at (0,3){$t_1 t_2 u_1$};
\node[] at (-1,4){$t_1^2 t_2 u_1$};
\node[] at (2,3){$t_2^2 u_1$};
\end{scope}

\begin{scope}[shift={(0,4.5)}]
\draw[thick,-](0,0)--(-2,2)--(0,4)--(1,3)--(2,4)--(3,3)--(0,0);
\draw[thick,-](1,1)--(-1,3);
\draw[thick,-](2,2)--(1,3);
\draw[thick,-](-1,1)--(1,3);
\node[] at (0,1){$u_2$};
\node[] at (-1,2){$t_1 u_2$};
\node[] at (1,2){$t_2 u_2$};
\node[] at (0,3){$t_1 t_2 u_2$};
\node[] at (2,3){$t_2^2 u_2$};
\end{scope}

\begin{scope}[shift={(-1,9)}]
\draw[thick,-](1,0)--(-1,2)--(0,3)--(2,1)--(1,0);
\draw[thick,-](0,1)--(1,2);
\node[] at (1,1){$u_3$};
\node[] at (0,2){$t_1 u_3$};
\end{scope}

\draw[thick, ->] (V1) -- (F1);

\draw[dotted, -](0,0)--(0,12);
\draw[dotted, -](1,0)--(1,12);
\draw[dotted, -](2,0)--(2,12);
\draw[dotted, -](3,0)--(3,12);
\draw[dotted, -](-1,0)--(-1,12);
\draw[dotted, -](-2,0)--(-2,12);
\draw[dotted, -](-3,0)--(-3,12);
\draw[dotted, -](-4,0)--(-4,12);
\end{tikzpicture}

 \caption{The $\bT$-fixed point on $\mathcal{M}(15,3)$ indexed by $\llambda=((4,3,1),(2,2,1),(2))$. The boxes represent basis vectors of $\mathscr{V}|_{\llambda}$ and the monomials in the boxes give the $\bT$-weights.}\label{fig1}
\end{figure}

\subsection{Trees in partitions}
Two boxes in a partition $\lambda$ with coordinates $(x_1,y_1)$ and $(x_2,y_2)$ are adjacent if
$$
x_1=x_2 \text{ and } y_1=y_2 \pm 1
$$
or 
$$
y_1=y_2 \text{ and } x_1=x_2 \pm 1
$$

\begin{Definition}
A tree $\tr$ in a partition $\lambda$ is a tree (a graph with no cycles) whose vertices consist of boxes in $\lambda$ and whose edges connect only adjacent boxes. A rooted tree in $\lambda$ is a tree in $\lambda$ with a distinguished box, called the root box.
\end{Definition}

A tree $\tr$ is totally determined by its set of edges, which, abusing notation, we will also denote by $\tr$.

An orientation on a tree is a choice of direction on each edge, or equivalently, a choice of two functions $h,t$ from edges to vertices, which sends an edge $e$ to one of the two vertices attached to it such that $h(e)\neq t(e)$. We refer to $h$ and $t$ as the head and tail functions and think of edges as proceeding from tail to head.

\begin{Definition}
A canonically oriented rooted tree in a partition $\lambda$ is a rooted tree $\tr$ in $\lambda$ with all edges oriented away from the root box. In other words, the root box does not appear in the image of the head function, and for all edges $e_1$ and $e_2$, if $h(e_1)=h(e_2)$, then $e_1=e_2$.
\end{Definition}

Any subtree of an oriented tree inherits an orientation. Given a canonically oriented rooted tree $\tr$ in $\lambda$ and a box $a \in \lambda$, there is a natural canonically oriented rooted subtree of $\lambda$ rooted at $a$ with orientation induced by that of $\tr$.

\begin{Definition}\label{subtree}
Given a rooted canonically oriented tree $\tr$ inside a partition $\lambda$ and a box $a \in \lambda$, we define $[a,\tr]$ to be the set of boxes in $\lambda$ appearing as vertices in the natural canonically oriented subtree of $\tr$ rooted at $a$. So $a \in [a, \tr]$, as well as any boxes that can be obtained by following edges that start at $a$ from tail to head.
\end{Definition}


\begin{Definition}
The skeleton of $\lambda$, denoted $\Sigma_{\lambda}$, is the graph with vertex set given by all boxes in $\lambda$ and edge set given by all edges that connect adjacent boxes.
\end{Definition}

\begin{Definition}
A $\bL$-shaped subgraph of $\Sigma_{\lambda}$ is a subgraph of $\Sigma_{\lambda}$ with two edges, which are of the form 
$$
e_1=\{(x,y),(x+1,y)\} \quad e_2=\{(x+1,y),(x+1,y+1)\}
$$
\end{Definition}

\begin{Proposition}
There are 
$$
\sum_{i}(m_i-1)
$$
$\bL$-shaped subraphs of $\Sigma_{\lambda}$, where $m_i:=\left|\{a \in \lambda \, \mid \, c_a=i\}\right|$
\end{Proposition}
\begin{proof}
This can be proven easily by, for example, induction on the number of boxes of $\lambda$.
\end{proof}

We index the $\bL$-shaped subgraphs in $\lambda$ as $\gamma_1,\ldots, \gamma_m$. 
\begin{Proposition}
Let $e_i$, $i \in \{1, \ldots, m\}$ be edges such that $e_i$ is an edge of $\gamma_i$. Then $\Sigma_{\lambda}\setminus \{e_1,e_2,\ldots e_n\}$ is a tree in $\lambda$.
\end{Proposition}
\begin{proof}
This follows easily from the construction of $e_i$ and $\bL$-shaped subgraphs.
\end{proof}

We refer to the trees obtained as in the previous proposition as \textit{admissible trees}.
\begin{Definition}\label{adtr}
Let $\Gamma_{\lambda}$ be the set of $2^m$ canonically oriented trees in $\lambda$ rooted at the corner box whose underlying edge set forms an admissible tree.
\end{Definition}


\begin{Remark}
In what follows, we will assume that all trees in partitions are rooted at the corner box and are canonically oriented. We will henceforth refer to such trees simply as trees, and will call the corner box the root box.
\end{Remark}

Later we will need the following function.
\begin{Definition}\label{kappa}
Let $\tr$ be a tree in $\lambda$. With the orientation of the Young diagram of $\lambda$ as in (\ref{Yng}), we define $\kappa(\tr)$ to be the number of vertical arrows in $\tr$ directed down plus the number of horizontal edges in $\tr$ directed left.
\end{Definition}

An example of the admissible trees inside a partition is given in Figure \ref{fig2}.
\begin{figure}[ht]
    \centering
    \begin{tikzpicture}
    
   \begin{scope}[shift={(-4,0)}]
   \node at (0,1)[circle,fill,inner sep=1.5pt]{};
\draw[thick,-](0,0)--(-4,4)--(-3,5)--(-2,4)--(-1,5)--(1,3)--(2,4)--(3,3)--(0,0);
\draw[thick,-](1,1)--(-2,4);
\draw[thick,-](2,2)--(1,3);
\draw[thick,-](-1,1)--(1,3);
\draw[thick,-](-2,2)--(0,4);
\draw[thick,-](-3,3)--(-2,4);
\draw[thick,red,->](0,1)--($(0,1)+(1,1)$);
\draw[thick,red,->]($(0,1)+(1,1)$)--($(0,1)+(2,2)$);
\draw[thick,red,->](0,1)--($(0,1)+(-1,1)$);
\draw[thick,red,->](-1,2)--($(0,1)+(0,2)$);
\draw[thick,red,->](-1,2)--(-2,3);
\draw[thick,red,->](-2,3)--(-3,4);
\draw[thick,red,->](-2,3)--(-1,4);
\end{scope}

   \begin{scope}[shift={(3,0)}]
    \node at (0,1)[circle,fill,inner sep=1.5pt]{};
\draw[thick,-](0,0)--(-4,4)--(-3,5)--(-2,4)--(-1,5)--(1,3)--(2,4)--(3,3)--(0,0);
\draw[thick,-](1,1)--(-2,4);
\draw[thick,-](2,2)--(1,3);
\draw[thick,-](-1,1)--(1,3);
\draw[thick,-](-2,2)--(0,4);
\draw[thick,-](-3,3)--(-2,4);
\draw[thick,red,->](0,1)--(1,2);
\draw[thick,red,->](1,2)--(2,3);
\draw[thick,red,->](0,1)--(-1,2);
\draw[thick,red,->](-1,4)--(0,3);
\draw[thick,red,->](-1,2)--(-2,3);
\draw[thick,red,->](-2,3)--(-3,4);
\draw[thick,red,->](-2,3)--(-1,4);
\end{scope}

 \begin{scope}[shift={(-4,-5)}]
  \node at (0,1)[circle,fill,inner sep=1.5pt]{};
\draw[thick,-](0,0)--(-4,4)--(-3,5)--(-2,4)--(-1,5)--(1,3)--(2,4)--(3,3)--(0,0);
\draw[thick,-](1,1)--(-2,4);
\draw[thick,-](2,2)--(1,3);
\draw[thick,-](-1,1)--(1,3);
\draw[thick,-](-2,2)--(0,4);
\draw[thick,-](-3,3)--(-2,4);
\draw[thick,red,->](0,3)--(1,2);
\draw[thick,red,->](1,2)--(2,3);
\draw[thick,red,->](0,1)--(-1,2);
\draw[thick,red,->](-1,2)--(0,3);
\draw[thick,red,->](-1,2)--(-2,3);
\draw[thick,red,->](-2,3)--(-3,4);
\draw[thick,red,->](-2,3)--(-1,4);
\end{scope}

 \begin{scope}[shift={(3,-5)}]
  \node at (0,1)[circle,fill,inner sep=1.5pt]{};
\draw[thick,-](0,0)--(-4,4)--(-3,5)--(-2,4)--(-1,5)--(1,3)--(2,4)--(3,3)--(0,0);
\draw[thick,-](1,1)--(-2,4);
\draw[thick,-](2,2)--(1,3);
\draw[thick,-](-1,1)--(1,3);
\draw[thick,-](-2,2)--(0,4);
\draw[thick,-](-3,3)--(-2,4);
\draw[thick,red,->](0,3)--(1,2);
\draw[thick,red,->](1,2)--(2,3);
\draw[thick,red,->](0,1)--(-1,2);
\draw[thick,red,->](-1,4)--(0,3);
\draw[thick,red,->](-1,2)--(-2,3);
\draw[thick,red,->](-2,3)--(-3,4);
\draw[thick,red,->](-2,3)--(-1,4);
\end{scope}

\end{tikzpicture}
    \caption{The four admissible trees inside the partition $(4,3,1)$. The root box is indicated by the black dot, and the canonical orientation is shown. From top to bottom and left to right, the values of the function $\kappa$ are $1, -1, -1$, and $1$.}
    \label{fig2}
\end{figure}
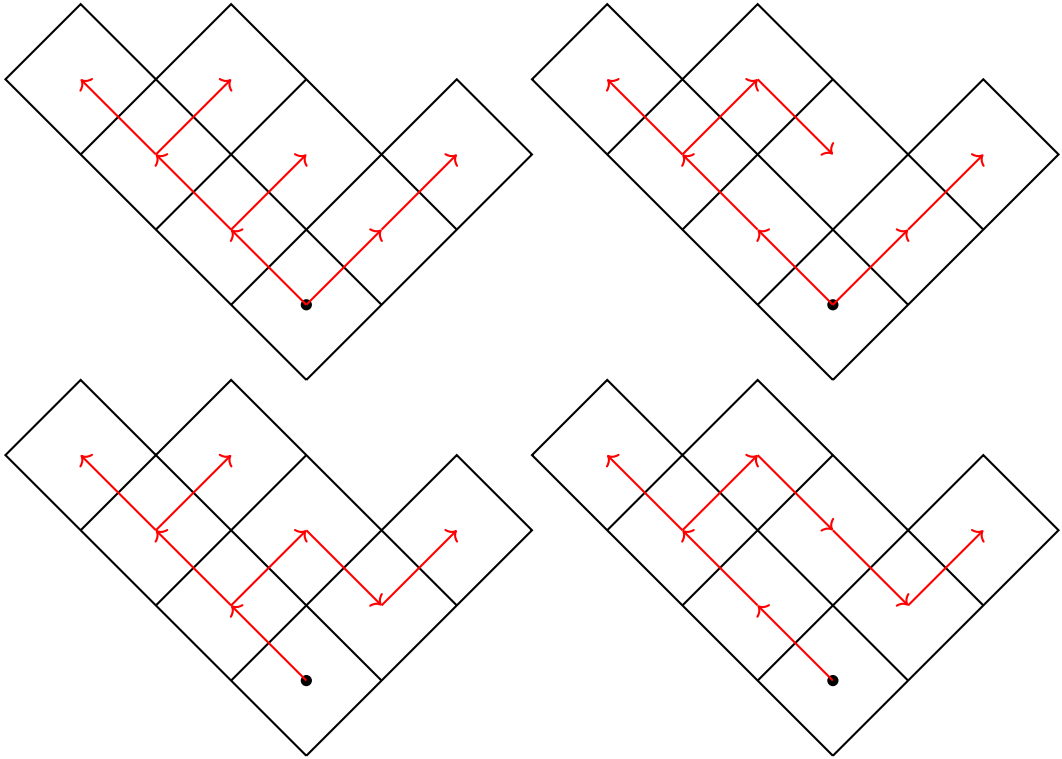

\subsection{Chern roots and torus weights labeled by boxes}
Fix $\llambda=(\lambda^{(1)},\ldots,\lambda^{(r)}) \in \mathcal{M}(n,r)^{\bT}$. If $a$ is a box of some $\lambda^{(i)}$, we write $a \in \llambda$. Choose some bijection
$$
\{\text{Chern roots of } \mathscr{V}\} \longleftrightarrow \{a \in \llambda\}
$$
where $\mathscr{V}$ is the tautological bundle on $\mathcal{M}(n,r)$.

Since our eventual formula for the elliptic stable envelope will involve a symmetrization over all the Chern roots of the tautological bundle, the choice of bijection does not matter. Hence we have a Chern root $x_a$ for each box $a \in \llambda$. Additionally, there is a natural bijection
$$
\{a \in \llambda \} \longleftrightarrow \{\bT-\text{weights of } \mathscr{V}|_{\llambda}\}
$$
which sends the box $a=(x,y)\in \lambda^{(i)}$ to the weight $\varphi^{\llambda}_{a}:=t_1^{y-1} t_2^{x-1} u_i$, where $u_1,\ldots,u_r$ are the framing parameters. We refer to $\varphi^{\llambda}_{a}$ as the restriction of the Chern root $x_{a}$ to the fixed point $\llambda$. See Figure \ref{fig1} for an example.

\subsection{Chamber}
We next specify a choice of chamber $\mathfrak{C} \subset \Lie_{\mathbb{R}}(\bA)$, where $\bA$ is the subtorus of $\bT$ preserving the symplectic form. We allow here for a fairly general choice. A chamber is represented by a real cocharacter 
$$
\sigma \in \text{cochar}(\bA)\otimes_{\mathbb{Z}} \mathbb{R}
$$
Recall that $\bA=\bA_{\mathsf{w}}\times \mathbb{C}^{\times}_{a}$, where $\bA_{\mathsf{w}}$ is the framing torus, and $a=\sqrt{t_1/t_2}$. With respect to this decomposition,
$$
\sigma =(\sigma_{\mathsf{w}},\sigma_a)
$$


We assume that $\sigma_{a}$ is infinitesimally small compared to the rest of the components, by which we mean that the presence of $a$ in a weight never affects its attracting/repelling properties unless it is the only coordinate that appears. In other words, we assume that the $\sigma$-attracting/repelling properties of tangent weights with nonzero $\bA_{\mathsf{w}}$-weight are totally determined by $\sigma_{\mathsf{w}}$. 


We can uniformly positively rescale each component of our cocharacter $\sigma$ without changing the corresponding chamber. So, without loss of generality, we assume that $\sigma_a:u \mapsto u^{\pm 1}$.

\subsection{Ordering on boxes}
We define a real-valued function on $\bT$-characters by
$$
\rho: w \mapsto \langle \sigma, w \rangle - \epsilon \deg_{\hbar}( w)
$$
where $0<\epsilon\ll 1$. Given a box $a \in \llambda$, we abuse notation and write
\begin{equation}\label{rho}
\rho_{a}:= \rho(\varphi^{\llambda}_a)
\end{equation}
Due to the genericity assumption on $\sigma$, the function $\rho$ defines a total order on the set of all boxes.

\section{Stable envelopes for the instanton moduli space}
As remarked above, the fixed points on $\mathcal{M}(n,r)$ are indexed by $r$-tuples of partitions $\llambda=(\lambda^{(1)},\ldots,\lambda^{(r)})$ with total number of boxes equal to $n$. Each of these partitions has a root box, and we denote the root box of $\lambda^{(j)}$ by $r_{j}$. A product $\prod\limits_{a \in \llambda}$ denotes a product taken over all boxes in all of the partitions $\lambda^{(j)}$. We will need the transcendental functions
\begin{align*}
\vartheta(x)&=(x^{1/2}-x^{-1/2})\prod_{i=1}^{\infty}(1-x q^i)(1-q^i/x) \\
\phi(x,y) &= \frac{\vartheta(xy)}{\vartheta(x)\vartheta(y)}
\end{align*}

We define the following three functions:
\begin{align}\nonumber
S_1(\vec{x};q,t_1,t_2)&= \prod_{\substack{a,b, \in \llambda \\ \rho_a+1 < \rho_b}} \vartheta\left(  \frac{x_a}{t_1 x_b}\right) \prod_{\substack{a,b, \in \llambda \\ \rho_b < \rho_a +1}} \vartheta\left( \frac{x_b}{t_2 x_a}\right) \\ \nonumber
S_2(\vec{x},\vec{u};q,t_1,t_2)&= \prod_{j=1}^{r}\left( \prod_{\substack{a \in \llambda \\ \rho_a \leq \rho_{r_{j}} }} \vartheta\left(\frac{x_a}{u_{j}}\right) \prod_{\substack{a \in \llambda \\ \rho_{r_{j}} < \rho_a }} \vartheta\left(\frac{ u_{j}}{t_1 t_2 x_a}\right)\right)\\ \label{nontree}
S_3(\vec{x};q,t_1,t_2)&= \prod_{\substack{a,b\in\llambda \\ \rho_a<\rho_b}} \vartheta\left(\frac{x_a}{x_b}\right) \vartheta\left(\frac{ x_a}{t_1 t_2 x_b}\right)
\end{align}
where $u_1,\ldots,u_r$ are the equivariant parameters of the framing torus. Let
$$
S_{\llambda}(\vec{x},\vec{u};q,t_1,t_2)=\frac{S_1(\vec{x};q,t_1,t_2) S_2(\vec{x},\vec{u};q,t_1,t_2)}{S_3(\vec{x};q, t_1,t_2)}
$$
We will sometimes omit the arguments and just write $S_{\llambda}$.

Let $T^{1/2}_{\rho>0,\llambda}$ be the virtual sub-bundle of $T^{1/2}$ consisting of terms that, when restricted to $\llambda$, give $\bT$-weights on which $\rho$ is positive. In other words, $T^{1/2}_{\rho>0,\llambda}$ consists of terms that restrict to attracting weights or negative powers of $\hbar$ at $\llambda$. We assume, as in Section \ref{algebras}, that the polarization is given by some formula in the tautological bundle, so that
$$
\det T^{1/2}_{\rho>0,\llambda} = \prod_{a \in \llambda} x_a^{d^{\llambda}_a}
$$
for some integers $d^{\llambda}_a$.

\begin{Definition}\label{treepart}
For a tree $\tr^{(j)}$ in a partition $\lambda^{(j)}$, we define
\begin{align*}
\trwt^{\lambda^{(j)}}_{\tr^{(j)}}(\vec{x},\vec{u},z;q,t_1,t_2) = (-1)^{\kappa(\tr_{j})}\phi\left(\frac{x_{r_{j}}}{\varphi^{\llambda}_{r_{j}}}, \prod_{a \in [r_{j}, \tr_j]} z_{c_a} (t_1t_2)^{-d^{\llambda}_a}\right) 
\prod_{e \in \tr_j} \phi\left(\frac{x_{h(e)} \varphi^{\llambda}_{t(e)}}{x_{t(e)}\varphi^{\llambda}_{h(e)}}, \prod_{a \in [h(e),\tr_{j}]} z_{c_a} (t_1t_2)^{-d^{\llambda}_a}\right)
\end{align*}
where $\kappa$ is the function in Definition \ref{kappa} and $[,]$ is as defined in Definition \ref{subtree}. Here, the product $\prod\limits_{e \in \tr_{j}}$ denotes the product over edges in the tree $\tr_{j}$.
\end{Definition}

Let
$$
\lt:= \Gamma_{\lambda^{(1)}} \times \ldots \times \Gamma_{\lambda^{(r)}}
$$
where $\Gamma_{\lambda^{(j)}}$ is from Definition \ref{adtr} so that an element of $\lt$ is a tuple of trees, one inside of each partition in $\llambda$.

\begin{Definition}
For $\llambda=(\lambda^{(1)},\ldots,\lambda^{(r)}) \in \mathcal{M}(n,r)^{\bT}$ and $\trs=(\tr^{(1)},\ldots,\tr^{(r)}) \in \lt$, let 
$$
\trwts^{\llambda}_{\trs}(\vec{x},\vec{u},z;q,t_1,t_2)= \prod_{j=1}^{r}  \trwt^{\lambda^{(j)}}_{\tr^{(j)}}(\vec{x},\vec{u},z;q,t_1,t_2)
$$
\end{Definition}
We will sometimes omit the arguments and just write $\trwts^{\llambda}_{\trs}$. The following is a special case of the formula for stable envelopes of affine type $A$ quiver varieties given in \cite{dinkinselliptic}.

\begin{Theorem}[\cite{dinkinselliptic}]\label{treeformula}
The elliptic stable envelope of $\llambda \in X^{\bT}$ is given by
$$
\stab_{T^{1/2},\mathfrak{C}}(\llambda)=\sym\left( S_{\llambda} \sum_{\trs \in \lt} \trwts^{\llambda}_{\trs} \right)
$$
where $\sym$ is the symmetrization over the Chern roots of $\mathscr{V}$.
\end{Theorem}
In particular, for $r=1$ we have a formula for the stable envelopes of $\text{Hilb}^{n}(\mathbb{C}^2)$.

\subsection{$K$-theory limit}
$K$-theoretic stable envelopes can be reproduced from elliptic stable envelopes via a limiting procedure, see Section 4.5 of \cite{AOElliptic}, Section 4.3 of \cite{dinkinselliptic}, and \cite{indstab1}. Since we only need these formulas for $\mathcal{M}(n,1)$ later, we specialize to that case here. In this case, $\bT$-fixed points are indexed by partitions $\lambda$ of $n$. The Picard group is generated by the tautological line bundle, and so the space of slopes is $\Pic(X)\otimes \mathbb{R}\cong \mathbb{R}$. As discussed in Section 8 of \cite{OS}, the set of walls, written $\Wall_n$, are given by rational numbers with denominator at most $n$.

Let $s\in \mathbb{R}\setminus \Wall_n$. The slope $s$ $K$-theoretic stable envelope of the fixed point $\lambda$ is
\begin{equation}\label{stablimit}
\stab_{\mathfrak{C},T^{1/2},s}(\lambda)= \left(\det T^{1/2}\right)^{1/2}\lim_{q\to 0} \stab^{Ell}_{\mathfrak{C},T^{1/2}}(\lambda)|_{z=zq^{-s}}
\end{equation}
Let $\ahat(x)=\lim_{q\to 0} \vartheta(x)=x^{1/2}-x^{-1/2}$. It is easy to show that
$$
\lim_{q\to 0}\frac{\vartheta(x z q^{s})}{\vartheta(z q^{s}) } =\frac{1}{x^{\lfloor s \rfloor +1/2}},\quad s \in \mathbb{R}\setminus \mathbb{Z}
$$
For an edge $e$ in a tree $\tr$ in $\lambda$, let $l_{e}=|[h(e),\tr]|$ and $m_{e}=\sum_{a \in [h(e),\tr]} -d^{\lambda}_{a}$ so that
$$
 \prod_{a \in [h(e),\tr]} z (t_1t_2)^{-d^{\lambda}_a}= z^{l_e} (t_1t_2)^{m_e}
$$
Then \ref{stablimit} implies the following. 

\begin{Theorem}\label{Ktreeformula}
$$
\stab_{\mathfrak{C},T^{1/2},s}(\lambda)= \left( \det T^{1/2} \right)^{1/2}\sym \left(S_{\lambda}^{K} \sum_{\tr \in \Gamma_{\lambda}} \trwt_{\tr}^{K,s}\right)
$$
where
\begin{align*}
    S_{\lambda}^{K}=\frac{\prod\limits_{\substack{a,b, \in \lambda  \\ \rho_a+1 < \rho_b}} \ahat\left( \frac{ x_a}{t_{1} x_b}\right) \prod\limits_{\substack{a,b, \in \lambda \\ \rho_b < \rho_a +1}} \ahat\left( \frac{ x_b}{t_{2} x_a}\right)  \prod\limits_{\substack{a \in \lambda \\ \rho_a \leq \rho_{r} }} \ahat(x_a) \prod\limits_{\substack{a \in \lambda \\  \rho_{r} < \rho_a }} \ahat\left(\frac{1}{t_1 t_2 x_a}\right) }{ \prod\limits_{\substack{a,b\in\lambda \\ \rho_a<\rho_b}} \ahat\left(\frac{x_a}{x_b}\right) \ahat\left(\frac{x_a}{t_1 t_2 x_b}\right)}
\end{align*}
and
$$
\trwt^{K,s}_{\tr}= (-1)^{\kappa(\tr)}\left(\frac{x_r}{\varphi^{\lambda}_{r}}\right)^{\lfloor |\lambda |s \rfloor +1/2}\ahat\left(\frac{x_r}{\varphi^{\lambda}_{r}}\right)^{-1} \prod_{e \in \tr}\ahat\left(\frac{x_{h(e)}  \varphi^{\lambda}_{t(e)}}{x_{t(e)}\varphi^{\lambda}_{h(e)}}\right)^{-1}  \left( \frac{x_{h(e)} \varphi^{\lambda}_{t(e)}}{x_{t(e)}\varphi^{\lambda}_{h(e)}}\right)^{\lfloor l_{e} s \rfloor +1/2}
$$
\end{Theorem}

As usual, we can replace $s$ by an alcove $\nabla$, since stable envelopes depend only on the alcove containing the slope.

\chapter{Hilbert Scheme of Points in the Plane}\label{CH2}

One of the most important Nakajima quiver varieties is the Hilbert scheme of $n$ points on $\mathbb{C}^{2}$, written $\text{Hilb}^{n}(\mathbb{C}^{2})$. The quantum difference equation for this variety was discovered in \cite{OS}. It is described in terms of the quantum toroidal $\mathfrak{gl}_{1}$. In this section, we will work through our results for this important example. 

\section{Solutions of exotic difference equations}

We first write formulas for solutions of the exotic $q$-difference equations.

\subsection{Quiver variety description}\label{hilbqv}
Set theoretically, the variety $\text{Hilb}^{n}(\mathbb{C}^{2})$ coincides with the set of codimension $n$ ideals of $\mathbb{C}[x,y]$:
$$
\text{Hilb}^{n}(\mathbb{C}^{2})= \{I \subset \mathbb{C}[x,y] \, \mid \, \dim_{\mathbb{C}} \mathbb{C}[x,y]/I = n\}
$$

It can be realized as a Nakajima quiver variety as follows. We consider the Jordan quiver, which is the quiver with one vertex and one loop. Let $n \in \mathbb{N}$, and let $V$ be a vector space of dimension $n$. We let the framing dimension be $1$ and we choose the positive stability condition, given by the cocharacter
\begin{align*}
    \theta:G:=GL(V) &\to \mathbb{C}^{\times} \\
     g &\mapsto \det(g)
\end{align*}
By definition, the quiver variety is
$$
\mathcal{M}(n,1)=\mu^{-1}(0)/\!\!/_{\theta} G = \mu^{-1}(0)^{\theta-ss}/G
$$
where $\mu$ is the moment map and $\mu^{-1}(0)^{\theta-ss}$ denotes the $\theta$-semistable points in $\mu^{-1}(0)$.
Points on the quiver variety are represented by tuples $(X,Y,v,w)$ such that $X,Y \in \text{End}(V)$, $v \in V$, and $w \in V^*$. The stability condition precisely states that $w=0$ and $v$ is a cyclic vector under the action of $X$ and $Y$. The moment map condition implies that $[X,Y]=0$. Then the map
\begin{align*}
    \mu^{-1}(0)^{\theta-ss}/GL(V) &\to \text{Hilb}^{n}(\mathbb{C}^{2}) \\
    (X,Y,v) &\mapsto \{f \in \mathbb{C}[x,y] \, \mid \, f(X,Y)v=0 \}
\end{align*}
is an isomorphism.

The natural action of $(\mathbb{C}^{\times})^{2}$ on $\mathbb{C}^2$ induces on action on $\text{Hilb}^n(\mathbb{C}^2)$. The framing torus acts trivially, hence we disregard it. We replace this $(\mathbb{C}^{\times})^{2}$ by its double cover $\bT$ so that if the coordinates on $(\mathbb{C}^{\times})^2$ are $t_1$ and $t_2$, then the coordinates on $\bT$ are $t_1^{1/2}$ and $t_2^{1/2}$. We also denote $\hbar=t_1t_2$. Let $t_1=a \hbar^{1/2}$ and $t_2=\hbar^{1/2}/a$ so that $t_1 t_2=\hbar$ and $t_1/t_2=a^2$. The coordinate $a$ is the coordinate on the subtorus of $\bT$ preserving the symplectic form. We choose the chamber such that the attracting directions are given by negative powers of $a$. Let $\mathfrak{C}$ denote this chamber. 

The $\bT$-fixed points are indexed by partitions $\lambda$ of $n$. The $\bT$-weights of the tautological vector bundle at the fixed point $\lambda$ are given in terms of the boxes of $\lambda$ as in Figure \ref{yng0}. The pushforwards of the structures sheaves of the fixed points under the inclusion $\iota: \text{Hilb}^{n}(\mathbb{C}^{2})^{\bT} \hookrightarrow \text{Hilb}^{n}(\mathbb{C}^{2})$ provides a basis of localized $K$-theory, which we denote by $[\lambda]:=\iota_{*}(\mathcal{O}_{\lambda})$.

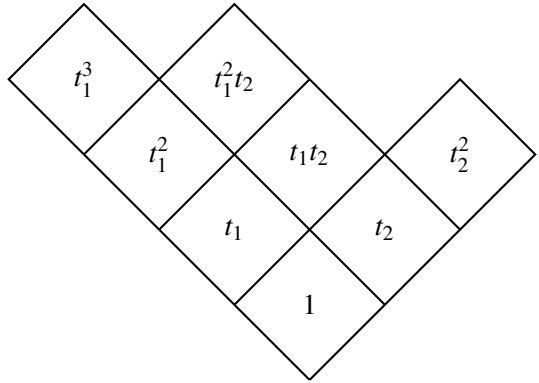
\begin{figure}[ht]
\centering
\begin{tikzpicture}[roundnode/.style={circle, draw=black, very thick, minimum size=5mm},squarednode/.style={rectangle, draw=black, very thick, minimum size=5mm}] 
\draw[thick,-](0,0)--(-4,4)--(-3,5)--(-2,4)--(-1,5)--(1,3)--(2,4)--(3,3)--(0,0);
\draw[thick,-](1,1)--(-2,4);
\draw[thick,-](2,2)--(1,3);
\draw[thick,-](-1,1)--(1,3);
\draw[thick,-](-2,2)--(0,4);
\draw[thick,-](-3,3)--(-2,4);
\node[] at (0,1){$1$};
\node[] at (-1,2){$t_1$};
\node[] at (-2,3){$t_1^2$};
\node[] at (-3,4){$t_1^3$};
\node[] at (1,2){$t_2$};
\node[] at (0,3){$t_1t_2$};
\node[] at (-1,4){$t_1^2 t_2$};
\node[] at (2,3){$t_2^2$};
\end{tikzpicture}
\caption{The $\bT$-weights of $\mathscr{V}$ on $\text{Hilb}^{8}(\mathbb{C}^2)$ at the fixed point $\lambda=(4,3,1)$.} \label{yng0}
\end{figure}

We fix once and for all the polarization of $\text{Hilb}^{n}(\mathbb{C}^2)$ given by 
\begin{equation}\label{hilbpol}
T^{1/2}= \frac{1}{t_2} \mathscr{V}^{*} \otimes \mathscr{V}+ \mathscr{V}-\mathscr{V}^{*}\otimes \mathscr{V}
\end{equation}
where $\mathscr{V}$ is the tautological bundle.

The quiver variety $\mathcal{M}(n,r)$ is known as the instanton moduli space. Although it is not the focus of our study here, it will appear below in the construction of the appropriate descendant insertions. The linear maps representing points in this variety correspond to arrows in Figure \ref{Jordan}

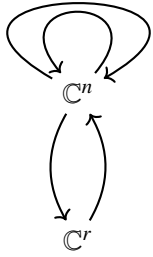
\begin{figure}[htbp]
    \centering
\begin{tikzpicture}[scale=0.98,roundnode/.style={circle,fill,inner sep=2.5pt},refnode/.style={circle,inner sep=0pt},squarednode/.style={rectangle,fill,inner sep=3pt}]

\node[] (V1) at (0,0) {$\mathbb{C}^n$};
\node[] (F1) at (0,-2){$\mathbb{C}^r$};
\draw[thick, ->] (V1) to [out=150,in=30,looseness=10] (V1);
\draw[thick, ->] (V1) to [out=50,in=130,looseness=8] (V1);
\draw[thick, ->] (V1) to [bend right] (F1);
\draw[thick, ->] (F1) to [bend right] (V1);

\end{tikzpicture}

 \caption{The quiver giving rise to $\mathcal{M}(n,r)$.}\label{Jordan}
\end{figure}


\subsection{Descendant insertions}

We briefly revert back to the notations of Section \ref{stabdes}. Denote the Hilbert scheme of $n$ points in $\mathbb{C}^2$ by $\mathcal{M}(n,1)$. Consider also $\mathcal{M}(n,1+n)$. Let $G'$ be the general linear group associated to the last $n$ framing dimensions, and let $\bA'$ be the maximal torus of $G'$. We also have the diagonal torus $\mathsf{U} \subset \bA'$ which acts on the last $n$ framing dimensions with weight $u$. The descendant insertions are defined in (\ref{stabu}) using the stable envelope
$$
\stab_{\mathsf{U},\nabla}: K_{\bT \times G'}(\mathcal{M}(n,1)) \to K_{\bT \times G'}(\mathcal{M}(n,1+n))
$$
where $\nabla$ is a choice of alcove for $\mathcal{M}(n,1)$ and the chamber is the one for which $u$ is an attracting weight. Recall that the descendant defined in Definition \ref{desc} is given by
$$
\textbf{f}^{\nabla}_{\alpha}=\Delta_{\hbar}^{-1} \textbf{s}^{\nabla}_{\alpha} \in K_{\bT}(\mathfrak{R})_{loc}
$$ 
where
$$
\textbf{s}^{\nabla}_{\alpha}=\iota^{*}\stab_{\mathsf{U},\nabla}(\alpha)
$$
and $\iota$ is the map (\ref{embed}). As
$$
K_{\bT}(\mathfrak{R})=K_{\bT\times G}(T^*\text{Rep}_{Q}(n,1))\cong K_{\bT\times G}(pt)
$$
we will describe $\textbf{f}^{\nabla}_{\alpha}\big|_{0}$ as a rational function in the weights of $\bT$ and $G$. 

\begin{Proposition}\label{hilbdesc}
In the fixed point basis, the descendant insertions of Definition \ref{desc} are given by
$$
\textbf{f}^{\nabla}_{[\lambda]}\big|_{0}=\sum_{\mu} \Lambda^{\bullet}\left(T_{\mu}X\right)^{-1} \textbf{S}^{\nabla}_{\lambda,\mu} (\det T^{1/2})^{1/2} \sym\left(S^{K}_{\lambda} \sum_{\tr \in \lt} \trwt_{\tr}^{K,\nabla} \right) \in K_{\bT \times G}(pt)_{loc}
$$
where  
$$
\textbf{S}^{\nabla}_{\lambda,\mu}=\stab_{-\mathfrak{C},T^{1/2}_{\text{opp}},-\nabla}(\lambda)|_{\mu}
$$
and
\begin{align*}
    S_{\lambda}^{K}&=\frac{\prod_{\substack{a,b, \in \lambda  \\ \rho_a+1 < \rho_b}} \ahat\left( \frac{ x_a}{t_{1} x_b}\right) \prod_{\substack{a,b, \in \lambda \\ \rho_b < \rho_a +1}} \ahat\left( \frac{ x_b}{t_{2} x_a}\right)  \prod_{\substack{a \in \lambda \\ \rho_a \leq \rho_{r} }} \ahat(x_a) \prod_{\substack{a \in \lambda \\  \rho_{r} < \rho_a }} \ahat\left(\frac{1}{t_1 t_2 x_a}\right) }{ \prod_{\substack{a,b\in\lambda \\ \rho_a<\rho_b}} \ahat\left(\frac{x_a}{x_b}\right) \ahat\left(\frac{x_a}{t_1 t_2 x_b}\right)}\\
    \trwt^{K,\nabla}_{\tr}&= (-1)^{\kappa(\tr)}\left(\frac{x_r}{\varphi^{\lambda}_{r}}\right)^{\lfloor |\lambda |s \rfloor +1/2}\ahat\left(\frac{x_r}{\varphi^{\lambda}_{r}}\right)^{-1} \prod_{e \in \tr}\ahat\left(\frac{x_{h(e)}  \varphi^{\lambda}_{t(e)}}{x_{t(e)}\varphi^{\lambda}_{h(e)}}\right)^{-1}  \left( \frac{x_{h(e)} \varphi^{\lambda}_{t(e)}}{x_{t(e)}\varphi^{\lambda}_{h(e)}}\right)^{\lfloor l_{e} s \rfloor +1/2}
    \end{align*} 
    for $s \in \nabla$.
\end{Proposition}

We remark that Theorem \ref{Ktreeformula} provides an explicit formula for $\textbf{S}^{\nabla}_{\lambda,\mu}$.
\begin{proof}



Recall the distinguished point $\star \in \mathcal{M}(n,1+n)$ described in the paragraph before Proposition \ref{window}. We have $\textbf{f}^{\nabla}_{[\lambda]}\big|_{0}= \Delta_{\hbar}^{-1} \stab_{\mathsf{U},\nabla}([\lambda])|_{\star}$ as elements of $K_{\bT \times G}(pt)_{loc}$. The inclusion $\bT\times \bA' \subset \bT \times G'$ induces a map $K_{\bT \times G'}(\cdot) \to K_{\bT \times \bA'}(\cdot)$ compatible with stable envelopes of $\mathsf{U}$. So, we abuse notation and also write $\stab_{\mathsf{U},\nabla}$ for the map $\stab_{\mathsf{U},\nabla}:K_{\bT\times \bA'}(\mathcal{M}(n,1)) \to K_{\bT \times \bA'}(\mathcal{M}(n,1+n))$.

By the triangle lemma for stable envelopes for the pair $\mathsf{U}\subset \bA'$, see Section 9.2 of \cite{pcmilect}, we have a commutative diagram
\begin{equation*}
\begin{tikzcd}
K_{\bT\times \bA'}(\mathcal{M}(n,1)) \arrow[labels= below left, "\stab_{\bA',\nabla}", rd]  \arrow[labels=above, "\stab_{\bA'/\mathsf{U},\nabla}"]{r} & K_{\bT\times \bA'}(\mathcal{M}(n,1)) \arrow["\stab_{\mathsf{U},\nabla}",d] \\
 & K_{\bT \times \bA'}(\mathcal{M}(n,1+n))
\end{tikzcd}
\end{equation*}
Since the torus $\bA'$ fixes $\mathcal{M}(n,1)$, the top arrow is the identity map. Thus $\stab_{\mathsf{U},\nabla}=\stab_{\bA',\nabla}$. Using the triangle lemma for the pair $\bA' \subset \bT\times \bA'$, we have a commutative diagram
\begin{equation*}
\begin{tikzcd}
K_{\bT\times \bA'}(\mathcal{M}(n,1)^{\bT}) \arrow[labels= below left, "\stab_{\bT\times \bA',\nabla}", rd]  \arrow[labels=above, "\stab_{\bT,\nabla}"]{r} & K_{\bT\times \bA'}(\mathcal{M}(n,1)) \arrow["\stab_{\bA',\nabla}",d] \\
 & K_{\bT \times \bA'}(\mathcal{M}(n,1+n))
\end{tikzcd}
\end{equation*}
Thus
$$
\stab_{\mathsf{U},\nabla}=\stab_{\bA',\nabla}=\stab_{\bT\times\bA',\nabla} \circ \stab_{\bT,\nabla}^{-1}
$$

In the fixed point basis, we have
$$
\stab_{\bA',\nabla}([\lambda])|_{\star}=\stab_{\bT \times \bA',\nabla}(\stab_{\bT}^{-1}([\lambda]))|_{\star}= \sum_{\mu} \stab_{\bT}^{-1}([\lambda])|_{\mu} \stab_{\bT\times\bA'}([\mu])|_{\star}
$$
Now, $\stab_{\bT \times\bA',\nabla}([\mu])$ is given by the $K$-theory limit of the tree formula in Theorem \ref{treeformula}. The restriction to $\star$ is given by substituting $x_i=u_{i+1}$, $i \geq 1$ for the Chern roots of the tautological bundle $\mathscr{V}$. Viewing this as an element of $K_{\bT\times G}(pt)$ is equivalent to substituting $u_{i+1}=x_i$, $i \geq 1$. So, the formulas from (\ref{nontree}) show that
\begin{align*}
    S_1^{\mathcal{M}(n,1+n)}(\vec{x};q,t_1,t_2) &= S_1^{\mathcal{M}(n,1)}(\vec{x};q,t_1,t_2) \\
    S_2^{\mathcal{M}(n,1+n)}(\vec{x}, \vec{u};q,t_1,t_2)\big|_{u_{i+1}=x_i} &= \prod_{a, b \in \lambda} \vartheta\left(\frac{x_b}{x_a t_1 t_2} \right)\cdot S_2^{\mathcal{M}(n,1)}(\vec{x},\vec{u};q,t_1,t_2) \\
   S_3^{\mathcal{M}(n,1+n)}(\vec{x};q,t_1,t_2) &=  S_3^{\mathcal{M}(n,1)}(\vec{x};q,t_1,t_2)
\end{align*}
where the superscript indicates the corresponding quiver variety. In the limit to $K$-theory, the discrepancy in the second line becomes
$$
\lim_{q\to 0}\prod_{a, b \in \lambda} \vartheta\left(\frac{x_b}{x_a t_1 t_2} \right) = \prod_{a, b \in \lambda} \left(1- \frac{x_a t_1 t_2}{x_b}\right) \left(\frac{x_b}{x_a t_1 t_2}\right)^{1/2} = \Delta_{\hbar} \cdot \left(\det \hbar \text{Hom}(V,V)\right)^{1/2}
$$
The fixed point $\lambda \in \mathcal{M}(n,1)^{\bT}$ is embedded in $\mathcal{M}(n,1+n)$ as the fixed point indexed by the tuple of partitions $(\lambda,\emptyset, \ldots,\emptyset)$. So, the formula from Theorem \ref{Ktreeformula} shows that the tree-dependent parts of the formulas for $\stab_{\bT\times \bA',\nabla}(\lambda)$ and $\stab_{\bT\times,\nabla}(\lambda)$ agree exactly. The discrepancy between the contributions from the polarizations is
$$
\left(\frac{\det T^{1/2}_{\mathcal{M}(n,1+n)}}{\det T^{1/2}_{\mathcal{M}(n,1)} }\right)^{1/2}= \left( \det \hbar \text{Hom}(V,V)\right)^{1/2}
$$
So
$$
\stab_{\bT \times \bA',\nabla}([\mu])|_{\star}=\Delta_{\hbar} \cdot \stab_{\bT}([\mu])
$$
So
$$
\textbf{f}^{\nabla}_{[\lambda]}\big|_{0}=\sum_{\mu}\stab_{\bT}^{-1}([\lambda])|_{\mu} \stab_{\bT}([\mu])
$$
By duality of stable envelopes, see 9.1.17 of \cite{pcmilect}, the inverse of a stable envelope is the transpose of another stable envelope with chamber, polarization, and slope switched to their opposites, up to multiplication by the diagonal matrix $\text{diag}(\Lambda^{\bullet}(T_{\mu}X)^{-1})_{\mu}$. This finishes the proof.
\end{proof}

We will denote
\begin{equation}\label{gdesc}
\textbf{g}^{\nabla}_{\lambda}:=\textbf{f}^{\nabla}_{[\lambda]}\big|_{0}
\end{equation}
From now on, we restore the notation $X=\mathcal{M}(n,1)$.

\subsection{Computation of vertex with descendants}
Next we compute the vertex with descendants for $X$. We apply localization to the moduli space $\qm_{\ns \, 0}$. By definition, a stable quasimap to $X$ consists of
\begin{itemize}
    \item A rank $n$ vector bundle $\mathscr{V}$ on $\mathbb{P}^1$.
    \item A section $f=(X_1,X_2,a,b)$ of the bundle $$
    \mathscr{M}=t_1^{-1} \text{Hom}(\mathscr{V},\mathscr{V})\oplus t_2^{-1} \text{Hom}(\mathscr{V},\mathscr{V})\oplus \mathscr{V} \oplus (t_1 t_2)^{-1} \mathscr{V}^*
    $$
    such that 
    \begin{itemize}
    \item The section satisfies the moment map equations:
    $$
    [X_1,X_2]+ab=0
    $$
        \item $f$ evaluates to a GIT-stable point at $0 \in \mathbb{P}^1$ and for all but finitely many $p \in \mathbb{P}^1$.
    \end{itemize}
\end{itemize}
As a consequence of the stability condition, one can prove that $b=0$ and so the moment map condition is $[X_1,X_2]=0$.

Points $p \in \mathbb{P}^1$ such that $f(p)$ is not stable are called singularities. We abuse notation and write $f \in \qm_{\ns \, 0}$, with the understanding that the data of a quasimap also consists of the bundle $\mathscr{V}$. By definition, the degree $d \in \mathbb{Z}$ of the quasimap is the degree of the bundle $\mathscr{V}$.

The torus $\mathbb{C}^{\times}_{q}$ acts on $\mathbb{P}^1$ in the usual way, and induces an action on $\qm_{\ns \, 0}$. The action of the torus $\bT$ also induces an action on $\qm_{\ns \, 0}$. We apply equivariant localization with respect to $\bT\times \mathbb{C}^{\times}_{q}$.

Suppose $f \in (\qm^{d}_{\ns \, 0})^{\bT \times \mathbb{C}^{\times}_{q}}$. Then $f$ provides a vector bundle $\mathscr{V}$. The singularities of $f$ are finite and invariant under $\mathbb{C}^{\times}_{q}$. Hence the only possible singular point is $\infty$, and the restriction of the quasimap to $\mathbb{P}^1\setminus \{\infty\}$ is a constant $\lambda \in X$. Since the quasimap is invariant under $\bT$, we must have $\lambda \in X^{\bT}$. Equivariantly with respect to $\mathbb{C}^{\times}_{q}$, we have
$$
\mathscr{V}\cong\bigoplus_{i=1}^{n} \mathcal{O}(d_i)
$$
and the linearization is such that the fiber of $\mathcal{O}(d)$ over $\infty$ is $q^{-d}$ and the fiber over $0$ has trivial $q$-weight.
So $f$ is a global section of
$$
\mathscr{M}\cong \bigoplus_{i,j=1}^{n} t_2^{-1} \mathcal{O}(d_i-d_j) \oplus  \bigoplus_{i,j=1}^{n} t_1^{-1} \mathcal{O}(d_i-d_j) \oplus \bigoplus_{i=1}^{n} \mathcal{O}(d_i) \oplus \bigoplus_{i=1}^{n} (t_1 t_2)^{-1} \mathcal{O}(-d_i)
$$
In order for the quasimap to be $\bT$-invariant, these bundles must be additionally twisted by certain $\bT$-weights.

Not all combinations of degrees $d_i$ are allowed. In order for the stability condition to be satisfied, at stable points $p$ of the quasimap, the image of $a(p)$ under $X_1(p)$ and $X_2(p)$ must generate the entire fiber $\mathscr{V}_{p}$. Hence the section must be nonvanishing at $0$. So, $a$ gives a global section of $\bigoplus_{i=1}^{n} \mathcal{O}(d_i)$ nonvanishing at 1 invariant under the action of $\mathbb{C}^{\times}_{q}$. This implies that $d_i\geq 0$ for all $i$, and there is only one such section.

Since $f(0)=\lambda$, the degrees must correspond to a reverse plane partition with shape $\lambda$, an example of which is given in Figure \ref{yng1}. We denote by $C_{\lambda}$ the set degrees corresponding to reverse plane partitions of shape $\lambda$. In this interpretation, each box of $\lambda$ corresponds to both a 1-dimensional $\bT$-weight space of $\mathscr{V}|_{0}$ and to a degree. We index the $\bT$-weights by $\varphi_1^{\lambda},\ldots, \varphi_{n}^{\lambda}$ and the degrees by $d_{1},\ldots,d_{n}$ such that $d_i$ is the degree for the box corresponding to $\varphi_i^{\lambda}$. Up to reordering, we have 
\begin{equation}\label{tbweights}
\{\varphi_i^{\lambda}\}= \{t_1^{j-1} t_2^{i-1}\}_{(i,j) \in \lambda}
\end{equation}

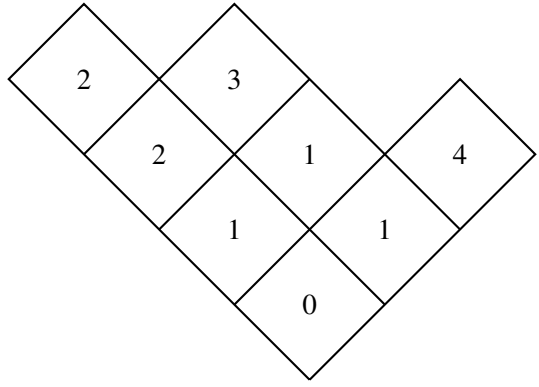
\begin{figure}[ht]
\centering
\begin{tikzpicture}[roundnode/.style={circle, draw=black, very thick, minimum size=5mm},squarednode/.style={rectangle, draw=black, very thick, minimum size=5mm}] 
\draw[thick,-](0,0)--(-4,4)--(-3,5)--(-2,4)--(-1,5)--(1,3)--(2,4)--(3,3)--(0,0);
\draw[thick,-](1,1)--(-2,4);
\draw[thick,-](2,2)--(1,3);
\draw[thick,-](-1,1)--(1,3);
\draw[thick,-](-2,2)--(0,4);
\draw[thick,-](-3,3)--(-2,4);
\node[] at (0,1){$0$};
\node[] at (-1,2){$1$};
\node[] at (-2,3){$2$};
\node[] at (-3,4){$2$};
\node[] at (1,2){$1$};
\node[] at (0,3){$1$};
\node[] at (-1,4){$3$};
\node[] at (2,3){$4$};
\end{tikzpicture}
\caption{A reverse plane partition of shape $\lambda=(4,3,1)$. The integers in the boxes must be non-negative and weakly increasing as one moves up-right or up-left.} \label{yng1}
\end{figure}

Thus equivariantly with respect to $\bT \times \mathbb{C}^{\times}_{q}$, we have
$$
\mathscr{M}\cong \bigoplus_{i,j=1}^{n} \frac{\varphi^{\lambda}_{i}}{\varphi^{\lambda}_{j} t_2} \mathcal{O}(d_i-d_j) \oplus  \bigoplus_{i,j=1}^{n} \frac{\varphi^{\lambda}_{j}}{\varphi^{\lambda}_{i} t_1} \mathcal{O}(d_i-d_j) \oplus \bigoplus_{i=1}^{n} \varphi_{i}^{\lambda} \mathcal{O}(d_i) \oplus \bigoplus_{i=1}^{n} \frac{1}{\varphi^{\lambda}_{i} t_1 t_2} \mathcal{O}(-d_i)
$$
Recall that with the linearization of $\mathcal{O}(d)$ from Proposition \ref{linearization},
$$
H^0\left(\mathbb{P}^1,\mathcal{O}(d) \right)=\sum_{i=1}^{d} q^i
$$
and the dualizing sheaf on $\mathbb{P}^1$ is equivariantly given by $q \mathcal{O}(-2)$.

Term in the reduced virtual tangent\footnote{The reduced virtual tangent space refers to the virtual $\bT\times\mathbb{C}^{\times}_{q}$-module $T_{\text{vir}}\big|_{f}-T_{\lambda}X$ where $f(0)=\lambda$. The second term is simply a convenient way to account for the denominators provided by computing a pushforward by localization.} space show up in pairs, and we use Serre duality to compute its character. For example, if $d_i-d_j\geq 0$, then
\begin{align*}
&H^*\left(\frac{\varphi^{\lambda}_i}{\varphi^{\lambda}_{j} t_2} \mathcal{O}(d_i-d_j)+\frac{\varphi^{\lambda}_j }{\varphi^{\lambda}_{i} t_1} \mathcal{O}(d_j-d_i) \right) -\frac{\varphi_i^{\lambda}}{\varphi^{\lambda}_{j} t_2}- \frac{\varphi_j^{\lambda}}{\varphi^{\lambda}_{i} t_1} \\
&=\frac{\varphi^{\lambda}_{i}}{\varphi^{\lambda}_{j} t_2} H^0(\mathcal{O}(d_i-d_j))-\frac{\varphi^{\lambda}_{i}}{\varphi^{\lambda}_{j} t_2} - \frac{\varphi^{\lambda}_{j}}{\varphi^{\lambda}_{i} t_1} H^1(\mathcal{O}(d_j-d_i)) -\frac{\varphi^{\lambda}_{j}}{\varphi^{\lambda}_{i} t_1} \\
&= \frac{\varphi^{\lambda}_{i}}{\varphi^{\lambda}_{j} t_2}(q+\ldots q^{d_i-d_j}) - \frac{\varphi^{\lambda}_{j}}{\varphi^{\lambda}_{i} t_1}H^0(q\mathcal{O}(d_i-d_j-2))^{\vee} -\frac{\varphi^{\lambda}_{j}}{\varphi^{\lambda}_{i} t_1} \\
&=\frac{\varphi^{\lambda}_{i}}{\varphi^{\lambda}_{j} t_2}(q+\ldots q^{d_i-d_j})- \frac{\varphi^{\lambda}_{j}}{\varphi^{\lambda}_{i} t_1}(1+q^{-1}+\ldots +q^{-(d_i-d_j)+1})
\end{align*}
The other terms contribute similarly. Overall, we obtain
\begin{align*}
    T_{\text{vir}}|_{f}&=\sum_{\substack{i,j=1\\ d_i-d_j\geq 0}}^{n} \frac{\varphi^{\lambda}_i}{\varphi^{\lambda}_j t_2}\left(q+\ldots q^{d_i-d_j}\right)- \frac{\varphi^{\lambda}_j}{\varphi^{\lambda}_i t_1}\left(1+q^{-1}+\ldots q^{-(d_i-d_j)+1}\right) \\
    &+ \sum_{\substack{i,j=1\\ d_i-d_j< 0}}^{n} -\frac{\varphi^{\lambda}_i}{\varphi^{\lambda}_j t_2}\left(1+q^{-1}+\ldots q^{d_i-d_j+1}\right)+ \frac{\varphi^{\lambda}_j}{\varphi^{\lambda}_i t_1}\left(q+\ldots q^{-(d_i-d_j)}\right) \\
    &+ \sum_{\substack{i,j=1\\ d_i-d_j\geq 0}}^{n} -\frac{\varphi^{\lambda}_i}{\varphi^{\lambda}_j}\left(q+\ldots q^{d_i-d_j}\right)+ \frac{\varphi^{\lambda}_j}{\varphi^{\lambda}_i t_1 t_2}\left(1+q^{-1}+\ldots q^{-(d_i-d_j)+1}\right) \\
    & + \sum_{\substack{i,j=1\\ d_i-d_j< 0}}^{n} -\frac{\varphi^{\lambda}_i}{\varphi^{\lambda}_j}\left(1+q^{-1}+\ldots q^{d_i-d_j+1}\right)+ \frac{\varphi^{\lambda}_j}{\varphi^{\lambda}_i t_1 t_2}\left(q+\ldots q^{-(d_i-d_j)}\right) \\
    &+ \sum_{i=1}^{n} \varphi^{\lambda}_{i}\left(q+\ldots +q^{d_i}\right) -\frac{1}{\varphi^{\lambda}_{i} t_1 t_2 }\left(1+ q^{-1}+\ldots q^{-d_i+1} \right)
\end{align*}
We define
$$
\{q\}_{k}=\begin{cases}
\sum\limits_{i=1}^{k} q^i & k \geq 0 \\
\sum_{i=k+1}^{0} q^i & k < 0
\end{cases}
$$
and can more compactly write
\begin{align*}
     T_{\text{vir}}|_{f}&=\sum_{\substack{i,j=1\\ d_i-d_j\geq 0}}^{n} \left(\frac{\varphi^{\lambda}_i}{\varphi^{\lambda}_j t_2}-\frac{\varphi^{\lambda}_{i}}{\varphi^{\lambda}_{j}} \right)\{q\}_{d_i-d_j}- \left(\frac{\varphi^{\lambda}_j}{\varphi^{\lambda}_i t_1} -\frac{\varphi^{\lambda}_j}{\varphi^{\lambda}_i t_1 t_2} \right)\{q\}_{-(d_i-d_j)}  \\ &+ \sum_{\substack{i,j=1\\ d_i-d_j< 0}}^{n} \left(-\frac{\varphi^{\lambda}_i}{\varphi^{\lambda}_j t_2}-\frac{\varphi^{\lambda}_i}{\varphi^{\lambda}_j} \right)\{q\}_{d_i-d_j}+ \left(\frac{\varphi^{\lambda}_j}{\varphi^{\lambda}_i t_1}+\frac{\varphi^{\lambda}_j}{\varphi^{\lambda}_i t_1 t_2} \right)\{q\}_{-(d_i-d_j)} \\
    &+ \sum_{i=1}^{n} \varphi^{\lambda}_{i}\{q\}_{d_i} -\frac{1}{\varphi^{\lambda}_{i} t_1 t_2 }\{q\}_{-d_i}
\end{align*}

A descendant $\alpha\in K_{\bT\times G}(pt)$ is a symmetric function in the Chern roots of the tautological bundle. For a fixed quasimap $f$ with $f(0)=\lambda$, we denote $\alpha(x(f))$ the result of substituting $x_i=\varphi_i^{\lambda} q^{d_i}$, where $d_i$ are the degrees given by $f$ as above. The vertex function with the polarization $T^{1/2}_{\text{opp}}$ in the fixed point basis is thus
\begin{align*}
\ver(\alpha)|_{\lambda}&= \sum_{d =0}^{\infty} z^d \sum_{\substack{f \in (\qm_{\ns \, 0})^{\bT \times \mathbb{C}^{\times}_{q}} \\ f(0)=\lambda}} \hat{a}\left(T_{\text{vir}}|_{f}\right) \left(\frac{\det \mathscr{T}^{1/2}_{\text{opp},0}\big|_{f}}{\det \mathscr{T}^{1/2}_{\text{opp},\infty}\big|_{f}} \right)^{1/2} \alpha(x(f)) \\
&=\sum_{d\in C_{\lambda}} z^{|d|} \hat{a}\left(T_{\text{vir}}|_{f_{d}}\right) \left(\frac{\det \mathscr{T}^{1/2}_{\text{opp},0}\big|_{f_{d}}}{\det \mathscr{T}^{1/2}_{\text{opp},\infty}\big|_{f_{d}}} \right)^{1/2} \alpha(x(f_{d})) 
\end{align*}
where $f_{d}$ is the $\bT\times\mathbb{C}^{\times}_{q}$-fixed quasimap indexed by the reverse plane partition given by $d$, $\hat{a}$ is the function defined by
$$
\hat{a}(t)=\frac{1}{t^{1/2}-t^{-1/2}}\quad \hat{a}(t_1+t_2-t_3)=\frac{\hat{a}(t_1) \hat{a}(t_2)}{\hat{a}(t_3)}
$$
and $\mathscr{T}_{\text{opp}}^{1/2}$ is the bundle on $\qm_{\ns \, 0}$ induced by the opposite of the chosen polarization, see \cite{pcmilect} Section 6.1. 

Furthermore,
$$
\left(\frac{\det \mathscr{T}^{1/2}_{\text{opp},0}}{\det \mathscr{T}^{1/2}_{\text{opp},\infty}} \right)^{1/2}= q^{\deg \mathscr{T}_{\text{opp}}^{1/2}/2}
$$
For our choice of polarization (\ref{hilbpol}), $\deg \mathscr{T}^{1/2}_{\text{opp}}= -|d|$. Denote the $q$-Pochammer symbol by
$$
(x)_{d}:=\begin{cases}
\prod\limits_{i=0}^{d-1} (1-x q^{i}) & d \geq 0 \\
\prod\limits_{i=d}^{-1}(1-x q^{i})^{-1} & d <0
\end{cases}
$$

Putting this all together, we obtain

\begin{Proposition}\label{verformula}
The restriction of the vertex function with descendant $\alpha$ for the polarization $T^{1/2}_{\text{opp}}$ to the fixed point $\lambda$ is given by
$$
\ver(\alpha)|_{\lambda}= \sum_{d \in C_{\lambda}}  \left(\frac{z}{\sqrt{t_1 t_2}}\right)^{|d|} \alpha(x_i)|_{x_i=\varphi_i^{\lambda} q^{d_i}} \prod_{i,j=1}^{n}\frac{\left(t_1 \frac{\varphi^{\lambda}_i}{\varphi^{\lambda}_j }\right)_{d_i-d_j}}{\left(\frac{q}{t_2} \frac{\varphi^{\lambda}_i}{\varphi^{\lambda}_j}\right)_{d_i-d_j}} \frac{\left(q \frac{\varphi^{\lambda}_i}{\varphi^{\lambda}_j}\right)_{d_i-d_j}}{\left(t_1 t_2 \frac{\varphi^{\lambda}_i}{\varphi^{\lambda}_j }\right)_{d_i-d_j}} \prod_{i=1}^{n}\frac{\left( \varphi^{\lambda}_i t_1 t_2\right)_{d_i}}{\left(q \varphi^{\lambda}_i\right)_{d_i}}  
$$
where $C_{\lambda}$ is the cone in $\mathbb{Z}^{n}$ corresponding to reverse plane partitions of shape $\lambda$.
\end{Proposition}

\subsection{Exotic solution in integral form}

This is the vertex that shows up in the right hand side of Theorem \ref{mainthm1}. Inserting the descendant $\textbf{f}^{\nabla}_{\alpha}$, we obtain $\ver^{\nabla}(\alpha)$. As in the proof of Theorem \ref{mainthm1}, one can check directly that $\lim_{q\to \infty} \ver^{\nabla}(\alpha)|_{\lambda}=\alpha$.

Recall the operator $\Omega^{\nabla}$ from (\ref{omega}). It is very similar to $\ver^{\nabla}$. The difference is that it is computed using the moduli space $\qm_{\ns \, \infty}$ instead of $\qm_{\ns \, 0}$ and the descendant insertion uses the opposite polarization and slope in the stable envelope. 

\begin{Proposition}
The restriction of the operator (\ref{omega}) to the fixed point basis is
$$
\Omega^{\nabla}(\alpha)|_{\lambda}=\sum_{d \in C_{\lambda}} \left(\sqrt{t_1 t_2} z\right)^{|d|} \textbf{f}^{\nabla}_{\text{opp},\alpha}(x)|_{x_i=\varphi_i^{\lambda} q^{-d_i}} \prod_{i,j=1}^{n} \frac{\left(\frac{\varphi^{\lambda}_j}{\varphi^{\lambda}_i t_1}\right)_{d_i-d_j}}{\left(q t_2 \frac{ \varphi^{\lambda}_j}{ \varphi^{\lambda}_i}\right)_{d_i-d_j}} \frac{\left(q \frac{\varphi^{\lambda}_j}{\varphi^{\lambda}_i}\right)_{d_i-d_j}}{\left(\frac{\varphi^{\lambda}_j}{ \varphi^{\lambda}_i  t_1 t_2}\right)_{d_i-d_j}} \prod_{i=1}^{n}\frac{\left( \frac{1}{\varphi^{\lambda}_i t_1 t_2}\right)_{d_i}}{\left(\frac{q}{ \varphi^{\lambda}_i}\right)_{d_i}} 
$$
\end{Proposition}

\begin{proof}
We repeat the analysis of the previous subsection. As above, a $\bT\times \mathbb{C}^{\times}_{q}$-fixed degree $d$ quasimap $f$ provides a vector bundle $\mathscr{V} \cong \bigoplus_{i=1}^{n} \mathcal{O}(d_i)$ where $\sum_{i} d_i =d$. By the stability condition and $\bT\times \mathbb{C}^{\times}_{q}$ invariance, we have $d_i \geq 0$ and the component of the section $f$ in $\mathscr{V}$ is uniquely determined. Applying $\mathbb{C}^{\times}_{q}$-invariance again, the bundles $\mathcal{O}(d_i)$ must be additionally twisted by $q^{-d_i}$. From the fact that
\begin{align*}
H^*(\mathbb{P}^1, q^{-d} \mathcal{O}(d)) = q^{-d}(1+q+\ldots +q^{d}) = 1+q^{-1}+\ldots +q^{-d} = H^*(\mathbb{P}^1,\mathcal{O}(d))|_{q=q^{-1}}
\end{align*}
we see that a formula $\Omega^{\nabla}(\alpha)|_{\lambda}$ can be obtained from $\ver^{\nabla}(\alpha)|_{\lambda}$ by substituting $q$ with $q^{-1}$ and shifting $z$ by a power of $q$. The shift of $z$ comes from the contribution of the polarization $\mathscr{T}^{1/2}_{\text{opp}}$ to the symmetrized virtual structure sheaf (\ref{symvrs}). In this case, the shift is $z \to z/q$.

We can compute the effect of switching $q$ to $q^{-1}$ in Proposition \ref{verformula}. By definition of $q$-Pochammer symbols, we have
$$
\frac{(\hbar x)_{d}}{(q x)_{d}}\bigg|_{q=q^{-1}}= \hbar^d q^d \frac{\left(\frac{1}{\hbar x}\right)_{d}}{\left(\frac{q}{x}\right)_{d}}, \quad d \in \mathbb{Z}
$$
Thus switching $q$ to $q^{-1}$ in the vertex gives
\begin{align*}
&\ver(\alpha)|_{\lambda,q=q^{-1}} \\
&=\sum_{d\in C_{\lambda}} \left(\frac{z}{\sqrt{t_1 t_2}}\right)^{|d|} \alpha(x)|_{x_i=\varphi_i^{\lambda} q^{-d_i}} \prod_{i,j=1}^{n} \left(\hbar q\right)^{d_i-d_j}\frac{\left(\frac{\varphi^{\lambda}_j}{\varphi^{\lambda}_i t_1}\right)_{d_i-d_j}}{\left(q t_2 \frac{ \varphi^{\lambda}_j}{ \varphi^{\lambda}_i}\right)_{d_i-d_j}} \left(\hbar q\right)^{-(d_i-d_j)}\frac{\left(q \frac{\varphi^{\lambda}_j}{\varphi^{\lambda}_i}\right)_{d_i-d_j}}{\left(\frac{\varphi^{\lambda}_j}{t_1 t_2 \varphi^{\lambda}_i }\right)_{d_i-d_j}} \prod_{i=1}^{n} \left(\hbar q\right)^{d_i}\frac{\left( \frac{1}{\varphi^{\lambda}_i t_1 t_2}\right)_{d_i}}{\left(\frac{q}{ \varphi^{\lambda}_i}\right)_{d_i}}  \\
&=\sum_{d\in C_{\lambda}} \left(\frac{z}{\sqrt{t_1 t_2}}\right)^{|d|} \alpha(x)|_{x_i=\varphi_i^{\lambda} q^{-d_i}} \prod_{i,j=1}^{n} \frac{\left(\frac{\varphi^{\lambda}_j}{\varphi^{\lambda}_i t_1}\right)_{d_i-d_j}}{\left(q t_2 \frac{ \varphi^{\lambda}_j}{ x^{\lambda}_i}\right)_{d_i-d_j}} \frac{\left(q \frac{\varphi^{\lambda}_j}{\varphi^{\lambda}_i}\right)_{d_i-d_j}}{\left(\frac{\varphi^{\lambda}_j}{t_1 t_2 \varphi^{\lambda}_i }\right)_{d_i-d_j}} \prod_{i=1}^{n} \left(\hbar q\right)^{d_i}\frac{\left( \frac{1}{\varphi^{\lambda}_i t_1 t_2}\right)_{d_i}}{\left(\frac{q}{ \varphi^{\lambda}_i}\right)_{d_i}}  \\
&=\sum_{d\in C_{\lambda}} \left(\sqrt{\hbar} q z\right)^{|d|} \alpha(x)|_{x_i=\varphi_i^{\lambda} q^{-d_i}}  \prod_{i,j=1}^{n} \frac{\left(\frac{\varphi^{\lambda}_j}{\varphi^{\lambda}_i t_1}\right)_{d_i-d_j}}{\left(q t_2\frac{ \varphi^{\lambda}_j}{ \varphi^{\lambda}_i}\right)_{d_i-d_j}} \frac{\left(q \frac{\varphi^{\lambda}_j}{\varphi^{\lambda}_i}\right)_{d_i-d_j}}{\left(\frac{\varphi^{\lambda}_j}{t_1 t_2 \varphi^{\lambda}_i }\right)_{d_i-d_j}} \prod_{i=1}^{n}\frac{\left( \frac{1}{\varphi^{\lambda}_i t_1 t_2}\right)_{d_i}}{\left(\frac{q}{ \varphi^{\lambda}_i}\right)_{d_i}} 
\end{align*}

Shifting $z \to z/q$ gives the result.
\end{proof}

By Theorem \ref{mainthm2}, we can also obtain an explicit formula for $\Psi^{\nabla}$ in the fixed point basis. Recall that the pushforward under the inclusion $\iota: X^{\bT} \hookrightarrow X$ induces an isomorphism $K_{\bT}(X^{\bT})_{loc}\cong K_{\bT}(X)_{loc}$. Since $X^{\bT}$ is finite, a basis of $K_{\bT}(X^{\bT})_{loc}$ is given by the classes of the structure sheaves of the fixed points. The class of a fixed point $\lambda \in X^{\bT}$ is defined to be $[\lambda]=\iota_{*}(\mathcal{O}_{\lambda})$ so that $[\lambda]|_{\mu}=\delta_{\lambda,\mu} \Lambda^{\bullet}(T_{\lambda} X)^{*}$. The classes of the fixed points provides a basis for $K_{\bT}(X)_{loc}$. 

The matrix of the operator $\Psi^{\nabla}$ in the basis of fixed points is given by
$$
\Psi^{\nabla}_{\lambda,\mu}:=\Psi^{\nabla}(\mu)|_{\lambda}
$$
so that
$$
\Psi^{\nabla}([\mu])=\sum_{\lambda} \Psi^{\nabla}_{\lambda,\mu} \cdot [\lambda] 
$$

\begin{Theorem}\label{hilbsolution}
The matrix $\Psi^{\nabla}_{\lambda,\mu}$ given by
$$
\Psi^{\nabla}_{\lambda,\mu}=\sum_{d \in C_{\mu}} \left(\sqrt{t_1 t_2} z\right)^{|d|} \textbf{g}^{\nabla}_{\text{opp},\lambda}(x)|_{x_i=\varphi_i^{\mu} q^{-d_i}} \prod_{i,j=1}^{n} \frac{\left(\frac{\varphi^{\mu}_j}{\varphi^{\mu}_i t_1}\right)_{d_i-d_j}}{\left(q t_2 \frac{ \varphi^{\mu}_j}{ \varphi^{\mu}_i}\right)_{d_i-d_j}} \frac{\left(q \frac{\varphi^{\mu}_j}{\varphi^{\mu}_i}\right)_{d_i-d_j}}{\left(\frac{\varphi^{\mu}_j}{ \varphi^{\mu}_i  t_1 t_2}\right)_{d_i-d_j}} \prod_{i=1}^{n}\frac{\left( \frac{1}{\varphi^{\mu}_i t_1 t_2}\right)_{d_i}}{\left(\frac{q}{ \varphi^{\mu}_i}\right)_{d_i}} 
$$
is a solution of the exotic quantum difference equation for the alcove $\nabla$. Here $\textbf{g}^{\nabla}_{\text{opp},\lambda}$ is from (\ref{gdesc}) where the stable envelopes use the opposite polarization and slope.
\end{Theorem}
\begin{proof}
By Theorem \ref{mainthm2}, the exotic solution $\Psi^{\nabla}$ is the adjoint of $\Omega^{\nabla}$ of Definition \ref{adjoint}. By definition of adjoint,
$$
(\Psi^{\nabla}([\lambda]),[\mu])_{X}= ([\lambda],\Omega^{\nabla}([\mu]))_{X}
$$
By definition of the pairing $(\cdot,\cdot)_{X}$, we have
$$
(\Psi^{\nabla}([\lambda]),[\mu])_{X}=\chi(\Psi^{\nabla}([\lambda])\otimes [\mu])= \sum_{\nu} \frac{\Psi^{\nabla}([\lambda])|_{\nu} [\mu]|_{\nu}}{\Lambda^{\bullet}(T_{\nu}X)^{*}} = \Psi^{\nabla}([\lambda])|_{\mu}
$$
Similarly, 
$$
([\lambda],\Omega^{\nabla}([\mu]))= \Omega^{\nabla}([\mu])_{\lambda}
$$
So in the fixed point basis, $\Psi^{\nabla}$ is simply the transpose of $\Omega^{\nabla}$.
\end{proof}

There is a convenient way to write the matrix of $\Psi^{\nabla}$ in the fixed point basis as a contour integral. We view the Chern roots as complex variables. Assume $|q|<1$. Let 
$$
(x)_{\infty}:=(x;q)_{\infty}:=\prod_{i=0}^{\infty}(1-x q^i)
$$
Define
$$
\Phi(\vec{x})=\prod_{i,j=1}^{n}\frac{\left(q \frac{t_2 x_j}{x_i}\right)_{\infty}}{\left(\hbar^{-1}\frac{t_2 x_j}{x_i}\right)_{\infty}} \frac{\left(\hbar^{-1}\frac{x_j}{x_i}\right)_{\infty}}{\left(q\frac{x_j}{x_i}\right)_{\infty}} \prod_{i=1}^{n} \frac{\left(q \frac{1}{x_i}\right)_{\infty}}{\left(\hbar^{-1}\frac{1}{ x_i}\right)_{\infty}}
$$
\begin{Theorem}\label{intform}
We have the equality
$$
\Psi^{\nabla}_{\lambda,\mu}=\frac{1}{2 \pi i \alpha_{\mu}}\int_{C_{\mu}} e^{-\frac{\ln(\sqrt{\hbar} z) \ln(x_1\ldots x_n)}{\ln(q)} } \Phi(\vec{x}) \textbf{g}^{\nabla}_{\text{opp},\lambda}(\vec{x}) \prod_{i=1}^{n} d x_i
$$
where
$$
\alpha_{\mu}=\left(e^{-\frac{\ln(\sqrt{\hbar}z) \ln(x_1 \ldots x_n)}{\ln(q)}} \Phi(\vec{x})\right)\Big|_{x_i=\varphi_i^{\mu}}
$$
and the contour $C_{\mu}$ encloses all poles of the form
$$
x_i=\varphi_i^{\mu} q^{-d_i}, \quad d_i \geq 0
$$
\end{Theorem}
\begin{proof}
This follows by a straightforward computation of the residues.
\end{proof}
Alternatively, one can avoid any analytic issues by interpreting the above integral as a Jackson $q$-integral, see Section 2.4 of \cite{dinksmir2}.

\section{Quantum toroidal algebra}\label{qta}

Having written formulas for the solutions of the exotic quantum difference equations, we now proceed to write the equations explicitly. We will do so using the algebra known as the quantum toroidal $\mathfrak{gl}_1$, written $\eha$. It was first discovered in \cite{miki}, see also \cite{FeiginBethe,FJMM, Negqta, SchifVas}. 

\subsection{Generators and relations}
This section follows \cite{SchifVas} closely. Let $\textbf{Z}=\mathbb{Z}^2$, $\textbf{Z}^{*}=\mathbb{Z} \setminus (0,0)$, 
$$
\textbf{Z}^{+}= \{(i,j) \in \textbf{Z} \, \mid \, i >0 \text{ or } i=0, j>0\}
$$
and $\textbf{Z}^{-}=-\textbf{Z}^{+}$.
Let 
$$
n_k=\frac{(t_1^{k/2}-t_1^{-k/2})(t_2^{k/2}-t_2^{-k/2})(\hbar^{-k/2}-\hbar^{k/2})}{k}
$$
where $\hbar=t_1t_2$. For a vector $\textbf{x}=(x_1,x_2) \in \textbf{Z}^*$, let $\deg(\textbf{x}):=\gcd(x_1,x_2)$. Let $\epsilon_{\textbf{x}}=1$ if $x \in \textbf{Z}^{+}$ and $\epsilon_{\textbf{x}}=-1$ if $x \in \textbf{Z}^{-}$. For noncollinear vectors $\textbf{x},\textbf{y} \in \textbf{Z}^*$, let $\epsilon_{\textbf{x},\textbf{y}}=\text{sign}(\det(\textbf{x},\textbf{y}))$.

\begin{Definition}
Let $\mathscr{U}_{\hbar}(\doublehat{\mathfrak{gl}}_1)$ be the associative unital algebra over $\mathbb{C}(t_1^{1/2},t_2^{1/2})$ generated by $e_{\textbf{x}}$ and $K_{\textbf{x}}$ for $\textbf{x} \in \textbf{Z}^*$ with the following relations:
\begin{enumerate}
    \item $K_{\textbf{x}}$ is central for all $\textbf{x} \in \textbf{Z}^*$, $K_{\textbf{x}} K_{\textbf{y}}=K_{\textbf{x}+\textbf{y}}$, and $K_{0}=1$.
    \item If $\textbf{x}$ and $\textbf{y}$ are collinear then $$[e_{\textbf{x}},e_{\textbf{y}}]=\delta_{\textbf{x}+\textbf{y}} \frac{K_{\textbf{x}}^{-1}-K_{\textbf{x}}}{n_{\deg(\textbf{x})}}$$
    \item If $\deg(\textbf{x})=1$ and the triangle with vertices $\{(0,0),\textbf{x},\textbf{y}\}$ has no interior lattice points, then 
    $$[e_{\textbf{x}},e_{\textbf{y}}]=\epsilon_{\textbf{y},\textbf{x}} K_{\alpha(\textbf{x},\textbf{y})} \frac{\Phi_{\textbf{x}+\textbf{y}}}{n_1}$$
    where
    $$
    \alpha(\textbf{x},\textbf{y})=\begin{cases}
    \epsilon_{\textbf{x}}(\epsilon_{\textbf{x}} \textbf{x} + \epsilon_{\textbf{y}} \textbf{y}-\epsilon_{\textbf{x}+\textbf{y}}(\textbf{x}+\textbf{y}))/2 & \epsilon_{\textbf{x},\textbf{y}}=1 \\
    \epsilon_{\textbf{y}}(\epsilon_{\textbf{x}}\textbf{x} +\epsilon_{\textbf{y}}\textbf{y} -\epsilon_{\textbf{x}+\textbf{y}}(\textbf{x}+\textbf{y}))/2 & \epsilon_{\textbf{x},\textbf{y}}=-1
    \end{cases}
    $$
    and the elements $\Phi_{\textbf{z}}$ for $\textbf{z} \in \textbf{Z}$ are defined by
$$
\sum_{k=0}^{\infty} \Phi_{k\textbf{w}} t^k= \exp\left( \sum_{m=1}^{\infty} n_m e_{m \textbf{w}} t^m\right)
$$ 
for any $\textbf{w}\in \textbf{Z}^*$ with $\deg(\textbf{w})=1$.
\end{enumerate}

\end{Definition}

The algebra $\mathscr{U}_{\hbar}(\doublehat{\mathfrak{gl}}_1)$ is called the quantum toroidal algebra. 

\subsection{Heisenberg subalgebras}\label{Heis}
For $w \in \mathbb{Q} \cup \{\infty\}$, we denote by $n(w)$ and $d(w)$ the numerator and denominator of the reduced expression for $w$. We define $n(\infty)=1$ and $d(\infty)=0$. We also assume $d(w)\geq 0$. Let
$$
\alpha^{w}_{k}:=e_{k d(w),k n(w)}
$$
for $k \in \mathbb{Z}\setminus \{0\}$. These generate Heisenberg subalgebras. The subalgebra for the wall $w=0$ is called the horizontal subalgebra, and the subalgebra for $w=\infty$ is called the vertical subalgebra. 

The upper and lower triangular universal $R$-matrices of the Heisenberg subalgebras for a wall $w\in \mathbb{Q}$ are given explicitly by
$$
R_{w}^{+}=\exp\left(-\sum_{k=1}^{\infty}n_k\alpha^{w}_{k} \otimes \alpha^{w}_{-k} \right)
$$
and
$$
R_{w}^{-}=\exp\left(-\sum_{k=1}^{\infty}n_k\alpha^{w}_{-k} \otimes \alpha^{w}_{k} \right)
$$
The universal $R$-matrix of the infinite slope Heisenberg subalgebra is given by
$$
R_{\infty}= \exp\left(-\sum_{k=1}^{\infty} n_k \alpha^{\infty}_{-k} \otimes \alpha^{\infty}_{k}\right)
$$

The universal $R$-matrix of $\eha$ for slope $s$ can be expressed as an infinite product of these wall $R$-matrices
$$
\mathscr{R}^{s}= \left(\rprod_{\substack{w \in \mathbb{Q} \\ w <s}} R_{w}^{-} \right) R_{\infty} \left( \lprod_{\substack{w \in \mathbb{Q} \\w >s}}R_{w}^{+} \right) 
$$

\subsection{Fock space representation}
It is convenient to consider the equivariant $K$-theory of all $\text{Hilb}^{n}(\mathbb{C}^2)$ at once. We consider equivariance with respect to the torus $\bT=\mathbb{C}^{\times}_{t_1^{1/2}}\times \mathbb{C}^{\times}_{t_2^{1/2}} \times \mathbb{C}_{u}$, where $t_1$ and $t_2$ are as in Section \ref{hilbqv} and $u$ is the framing parameter. According to the general theory of Section \ref{algebras}, the algebra $\hopf$ acts on the localized $K$-theory
$$
\bigoplus_{n=0}^{\infty} K_{\bT}(\text{Hilb}^{n}(\mathbb{C}^2))_{loc}
$$
Furthermore, the above direct sum decomposition is a weight space decomposition for a Cartan subalgebra of $\hopf$. It is well-known that there is an isomorphism of graded vector spaces
$$
\bigoplus_{n=0}^{\infty} K_{\bT}(\text{Hilb}^{n}(\mathbb{C}^2))_{loc}=\mathbb{Q}[p_1,p_2,\ldots] \otimes_{\mathbb{Z}} \mathbb{Q}\left(t_1^{1/2},t_2^{1/2},u\right)=:\mathsf{F}
$$
where the grading on $\mathsf{F}$ is defined by $\deg p_k=k$. This space can be identified with the ring of symmetric functions in infinitely many variables, where $p_i$ is viewed as the $i$th power sum function. According to the grading, we write
$$
\mathsf{F}=\bigoplus_{n=0}^{\infty} \mathsf{F}_n
$$

There is an important basis of $\mathsf{F}$ known as the Macdonald polynomials, see \cite{mac}. We state our conventions here. Define an inner product on $\mathsf{F}$ by
$$
\langle p_{\lambda}, p_{\mu}\rangle_{t_1,t_2}= \delta_{\lambda,\mu}\prod_{n \geq 1} n^{m_n} m_n!\prod_{i=1}^{l(\lambda)} \frac{1-t_1^{-\lambda_i}}{1-t_2^{\lambda_i}}  \quad  \text{where} \quad m_n=|\{k \, \mid \, \lambda_k=n \}|
$$
The usual parameters $q$ and $t$ of Macdonald theory are related to $t_1$ and $t_2$ by $q=t_1^{-1}$, $t=t_2^{1}$. The Macdonald polynomials, written $P_{\lambda}$, are the unique basis of $\mathsf{F}$ orthogonal with respect to this inner product with a strictly lower triangular expansion in the basis of monomial symmetric functions with respect to the dominance ordering on partitions. The integral form of the Macdonald polynomials is defined as
$$
\widetilde{H}_{\lambda}=t_2^{|\lambda|} \prod_{\square \in \lambda}\left( t_2^{-\text{leg}(\square)-1}-t_1^{-\text{arm}(\square)}\right) P_{\lambda}
$$
The modified Macdonald polynomials in \cite{Haim} are defined as
\begin{equation}\label{mac}
H_{\lambda}=\widetilde{H}_{\lambda}\left[\frac{p_k}{1-t_2^{k}}\right]
\end{equation}
where the square brackets stand for the substitution of $p_k$ by $\frac{p_k}{1-t_2^{k}}$ for all $k$.

\begin{Proposition}[\cite{NegFlags}]\label{ehaaction}
There is an action of the algebra $\eha$ on $\mathsf{F}$ in which the horizontal subalgebra acts via
$$
e_{m,0}= \begin{cases}
\frac{1}{(t_1^{m/2}-t_1^{-m/2})(t_2^{m/2}-t_2^{-m/2})} p_{-m} & m<0 \\
- m \frac{\partial}{\partial p_{m}} & m>0
\end{cases}
$$
and the vertical subalgebra acts diagonally on $H_{\lambda}$ as
$$
e_{0,m} H_{\lambda} =u^{-m}  \text{sign}(m) \left(\frac{1}{1-t_1^{m}} \sum_{i=1}^{\infty}t_1^{m \lambda_i} t_2^{m(i-1)} \right) H_{\lambda}
$$
\end{Proposition}

Since the horizontal and vertical subalgebras generate $\eha$, this completely determines the action. Furthermore, it is known that the generators of the Heisenberg subalgebras shift the grading as
\begin{equation}\label{gradingshift}
\alpha^{w}_{k}:\mathsf{F}_n \to \mathsf{F}_{n-k d(w)}
\end{equation}

\begin{Conjecture}
There exists an isomorphism of algebras $\hopf \cong \eha$, compatible with the actions of these algebras on $\mathsf{F}$. Under this isomorphism, 
\begin{itemize}
    \item the action of the $R$-matrix (\ref{slopeR}) on the tensor product of two Fock spaces coincides with the action of the universal $R$-matrix of $\eha$.
    \item  the wall $R$-matrices (\ref{wallR}) coincide with the $R$-matrices of the Heisenberg subalgebras of $\eha$.
    \item the wall subalgebras of $\hopf$ coincide with the Heisenberg subalgebras of $\eha$ from Section \ref{Heis}.
\end{itemize}
\end{Conjecture}

This conjecture is discussed in Chapter 1 of \cite{Neg} and Section 7.1.7 of \cite{OS}.

\begin{Remark}
For the remainder of this thesis, we assume that the above conjecture holds.
\end{Remark}

\subsection{Wall-crossing operators}
It follows from (\ref{gradingshift}) that for a fixed $n$ and $w$, there exists some $k>\!\!>0$ such that $\alpha^{w}_{k}$ acts by 0 on $\mathsf{F}_n$. The elements $\alpha^{w}_{k}$ for $k>0$ are called annihilation operators. The elements $\alpha^{w}_{k}$ for $k<0$ are called creation operators. For a monomial $\alpha^{w_1}_{k_1} \ldots \alpha^{w_m}_{k_m}$, let $\{1,\ldots, m\}=I^{+} \sqcup I^{-}$ be the partition into subsets such that $k_i>0$ for $i \in I^{+}$ and $k_i<0$ for $i \in I^{-}$. Define the normal ordering $::$ by 
$$
:\alpha^{w_1}_{k_1} \ldots \alpha^{w_m}_{k_m}: = \prod_{i \in I^{-}} \alpha^{w_i}_{k_i} \prod_{i \in I^{+}} \alpha^{w_i}_{k_i}
$$
In other words, all annihilation operators act first. We extend $::$ by linearity.

\begin{Proposition}[\cite{OS} Section 8]
For a wall $w\in \mathbb{Q} \cup \{\infty\}$, the wall-crossing operator is equal to the element
$$
\textbf{B}_{w}(z)= \, \, :\exp\left(\sum_{k=1}^{\infty} \frac{n_k \hbar^{d(w)k/2}}{1-z^{-d(w)k} q^{n(w)k} \hbar^{d(w)k/2}} \alpha^{w}_{-k} \alpha^{w}_{k}\right):
$$
of a completion of $\eha$.
\end{Proposition}

From the definition of normal ordering and from (\ref{gradingshift}), it is clear that $\textbf{B}_{w}(z)$ has a well-defined action on $\mathsf{F}$.

\begin{Proposition}\label{locfinite}
Let $\Wall_n=\{w \in \mathbb{Q} \, \mid \, 1\leq d(w) \leq n\}$. Then the operator $\textbf{B}_{w}(z)$ acts nontrivially on $\mathsf{F}_n$ only if $w \in \Wall_n$.
\end{Proposition}
\begin{proof}
By (\ref{gradingshift}), the element $\alpha^{w}_{k}$ for $k>0$ sends $\mathsf{F}_n$ to $\mathsf{F}_{n-kd(w)}$. So if $d(w)>n$, then it must act by $0$.
\end{proof}


Since we have explicit formulas for the wall crossing operators, we can verify Propositions \ref{Bshift}, \ref{Blim}, and \ref{Bdegrees} directly.

\begin{Proposition}
For any $w \in \mathbb{Q} \cup \{\infty\}$
$$
\Bw_{w}(z q^{w})
$$
is independent of $q$.
\end{Proposition}
\begin{proof}
Follows directly from the definition.
\end{proof}

\begin{Proposition}
Let $w \in \mathbb{Q} \cup \{\infty\}$ and $s \in \mathbb{R} \setminus \mathbb{Q}$. Then
$$
\lim_{q \to 0} \Bw_{w}(z q^s) = \begin{cases}
1 & s >w \\
\text{exists} & s<w
\end{cases}
$$
Equivalently, if $s<w$, then $\lim_{q\to \infty} \Bw_{w}(z q^s)=1$
\end{Proposition}

\begin{Proposition}
$\textbf{B}_{w}(z)$ can be expanded as a formal power series in positive powers of $z$ with constant term equal to 1.
\end{Proposition}

\subsection{Quantum difference equations}
Here we discuss the structure of the $q$-difference equation. It was first identified in \cite{OS}. The monodromy of the difference equation and of the differential equation in cohomology was studied using elliptic stable envelopes in \cite{qdehilb}. For $X=\text{Hilb}^{n}(\mathbb{C}^{2})$, the Picard group is isomorphic to $\mathbb{Z}$, with generator given by $\mathscr{L}$ the dual of the tautological line bundle, whose fiber at the fixed point $\lambda$ has $\bT$-weights $\prod_{(i,j) \in \lambda} t_{1}^{1-j} t_2^{1-i}$

The quantum difference equation has the form
$$
\Psi(zq) \mathscr{L} = \mathscr{L} \textbf{M}_{\mathscr{L}}(z) \Psi(z), \quad \Psi(0)=1
$$
for $\Psi(z) \in \text{End}(\mathsf{F})(q)[[z]]$. 


In \cite{OS}, the quantum difference operator is identified as
$$
\textbf{M}_{\mathscr{L}}(z) = \const_{X}  \cdot \mathscr{L} \rprod_{w \in [-1,0)} \textbf{B}_{w}(z) \in \text{End}(\mathsf{F})[q^{\pm 1}][[z]]
$$
where the arrow means that the indices increase moving from left to right and $\const_{X}$ is some constant. It is conjectured in \cite{OS} that $\const_{X}=1$. It follows from the definition that the wall-crossing operators preserve the direct sum decomposition of $\mathsf{F}$. By Proposition \ref{locfinite}, $\textbf{B}_{w}(z)$ acts by $1$ unless $w \in \Wall_n$. Thus on $\mathsf{F}_n$, this operator is
$$
\rprod_{w \in [-1,0)} \textbf{B}_{w}(z) = \rprod_{w \in [-1,0)\cap \Wall_n} \textbf{B}_{w}(z)
$$
 
\begin{Example}
If $n=3$, then 
$$
\rprod_{w \in [-1,0) \cap \Wall_{3}} \textbf{B}_{w}(z)= \textbf{B}_{-1}(z) \textbf{B}_{-\frac{2}{3}}(z) \textbf{B}_{-\frac{1}{2}}(z) \textbf{B}_{-\frac{1}{3}}(z)
$$
\end{Example}

Let $\nabla \subset \mathbb{R} \setminus \Wall_n$ be an alcove and define
$$
\textbf{B}^{\nabla}_{\mathscr{L}}(z)=\rprod_{w \in [s-1,s)} \textbf{B}_{w}(z)
$$
for some $s \in \nabla$. This is independent of the choice of $s \in \nabla$.

Then the exotic $q$-difference equation for alcove $\nabla \subset \mathbb{R} \setminus \Wall_n$ is the equation
\begin{equation}\label{exoticqde3}
\Psi^{\nabla}(qz) \mathscr{L} = \const_{X} \cdot \mathscr{L} \textbf{B}^{\nabla}_{\mathscr{L}}(z) \Psi^{\nabla}(z)
\end{equation}

\begin{Theorem}
The matrix of Theorem \ref{hilbsolution} solves the $q$-difference equation (\ref{exoticqde3}).
\end{Theorem}


\chapter{Bethe Ansatz}\label{CH5}

This final chapter is a conjectural application of our results. In it, we employ the integral formula for the solutions of the exotic $q$-difference equations to solve the Bethe ansatz problem for $\eha$. Although there are several key results that we will leave as conjectures, we emphasize that the results of this section have been thoroughly tested computationally, and we believe there is strong evidence that they are correct.

\section{Bethe subalgebras}
First, we explain the Bethe ansatz problem. Recall the universal $R$-matrices of the Heisenberg subalgebras of $\eha$:
$$
R_{w}^{+}=\exp\left(-\sum_{k=1}^{\infty}n_k\alpha^{w}_{k} \otimes \alpha^{w}_{-k} \right)
$$
$$
R_{w}^{-}=\exp\left(-\sum_{k=1}^{\infty}n_k\alpha^{w}_{-k} \otimes \alpha^{w}_{k} \right)
$$
$$
R_{\infty}= \exp\left(-\sum_{k=1}^{\infty} n_k \alpha^{\infty}_{-k} \otimes \alpha^{\infty}_{k}\right)
$$

The universal $R$-matrix of $\eha$ for slope $s$ can be expressed as an infinite product of these wall $R$-matrices
$$
\mathscr{R}^{s}= \left(\rprod_{\substack{w \in \mathbb{Q} \\ w <s}} R_{w}^{-} \right) R_{\infty} \left( \lprod_{\substack{w \in \mathbb{Q} \\w >s}}R_{w}^{+} \right) 
$$
It is an element of a completion of $\eha\otimes \eha$.


Recall that $\eha$ acts on the Fock space 
$$
\mathsf{F}=\mathbb{Q}[p_1,p_2,\ldots]\otimes \mathbb{Q}\left(t_1^{1/2},t_2^{1/2},u\right) = \bigoplus_{n=0}^{\infty} \mathsf{F}_n
$$

Consider the operator on $\mathsf{F}$ that scales the degree $d$ summand by $z^{d}$. We will write this as $z^{\deg}$. Evaluating the first tensor factor of $\mathscr{R}^s$ on the representation $\mathsf{F}$, we can consider the transfer matrix, which is defined as
$$
\mathsf{T}^{s}(u)=\text{Tr}_{\mathsf{F}}\left((z^{\deg} \otimes 1) \mathscr{R}^{s} \right) \in \hopf[[u,z]]
$$
where the trace is taken over the first copy of $\mathsf{F}$. The Yang-Baxter equation for $\mathscr{R}^{s}(u)$ implies the following.

\begin{Proposition}
$$
[\mathsf{T}^{s}(u_1),\mathsf{T}^{s}(u_2)]=0
$$
\end{Proposition}

Expanding $\mathsf{T}^{s}(u)$ as a series in $u$, we write
$$
\mathsf{T}^{s}(u)=\sum_{k=0}^{\infty} \mathsf{H}^{s}_{k} u^k
$$
The coefficients $\mathsf{H}^{s}_{k}$ are elements of a completion of $\eha[[z]]$ and are called the Hamiltonians of the transfer matrix. The previous proposition implies that
$$
[\mathsf{H}^{s}_{k_1},\mathsf{H}^{s}_{k_2}]=0
$$

\begin{Definition}
The Bethe subalgebra $\mathscr{B}^{s}$ for the slope $s \in \mathbb{R}\cup \{\infty\}\setminus \mathbb{Q}$ is the commutative subalgebra of a completion of $\eha[[z]]$ generated by all the Hamiltonians of the transfer matrix $\mathsf{T}^{s}(u)$.
\end{Definition}

Since the elements of $\mathscr{B}^{s}$ commute, they can be simultaneously diagonalized on any representation of $\eha$.

\begin{Definition}
The eigenvectors of $\mathscr{B}^{s}$ acting on $\mathsf{F}[[z]]$ are called the Bethe vectors.
\end{Definition}

\begin{Proposition}
The Bethe vectors of $\mathscr{B}^{s}$ are one parameter (in $z$) deformations of the modified Macdonald polynomials of (\ref{mac}).
\end{Proposition}
\begin{proof}
Consider the $z\to 0$ limit of the transfer matrix $\mathsf{T}^{s}(u)$, written $\mathsf{T}^{s}(u)|_{z=0}$. The degree $0$ part of $\mathsf{F}$ is one dimensional, spanned by $1$. So, if $\mathscr{R}^{s}(u):1\otimes x\mapsto 1\otimes y +\ldots$, then $\mathsf{T}^{s}(u)|_{z=0}: x\mapsto y$. By (\ref{gradingshift}), we have $\alpha^{w}_{k}(1)=0$ if $k >0$. So $R^{+}_{w}(1\otimes x)=1\otimes x$ for all $w$. By Proposition \ref{ehaaction}, we have 
$$
\alpha_{-k}^{\infty}(1)=-u^{k}\left(\frac{1}{1-t_1^{k}}\sum_{i=1}^{\infty} t_2^{k(i-1)}\right)=- \frac{u^{k}}{(1-t_1^{k})(1-t_2^{k})}
$$
for $k \geq 0$. So
$$
R_{\infty}(1 \otimes x)= 1 \otimes \exp \left(\sum_{k=0}^{\infty}\frac{n_k u^{k}}{(1-t_1^{k})(1-t_2^{k})} \alpha^{\infty}_{k}\right)(x)
$$
By (\ref{gradingshift}), each $R^{-}_{w}$ maps $1 \otimes y$ to $1\otimes y+\ldots$, where the dots stand for terms with degree greater than $0$ in the first tensor factor. Hence
$$
\mathsf{T}^{s}(u)|_{z=0}= \exp \left(\sum_{k=0}^{\infty}\frac{n_k u^{k}}{(1-t_1^{k})(1-t_2^{k})} \alpha^{\infty}_{k}\right)
$$
All $\alpha^{\infty}_{k}$, and hence $\mathsf{T}^{s}(u)|_{z=0}$, are diagonal in the basis of Macdonald polynomials $H_{\lambda}$. Thus the $z\to 0 $ limit of the Bethe eigenvectors are Macdonald polynomials.
\end{proof}

One of the main results of \cite{MO} is the identification of the quantum cohomology ring of a quiver variety with the Bethe subalgebra of the Yangian associated to the quiver. We conjecture that a similar result holds in $K$-theory.

\begin{Conjecture}\label{conj}
Let $s \in \mathbb{R}\cup\{\infty\}\setminus \mathbb{Q}$. Consider the element $\textbf{B}^{s}(z)$ of $\text{End}(\mathsf{F})[[z]]$ which acts on $\text{End}(\mathsf{F}_{n})[[z]]$ as
$$
\left(\mathscr{L} \textbf{B}^{\nabla}_{\mathscr{L}}(z)\right)\big|_{q=1}
$$
where $\mathscr{L}$ is the dual of the tautological line bundle on $\text{Hilb}^{n}(\mathbb{C}^2)$ and $\nabla$ is the alcove containing $s$. Then $\textbf{B}^{s}(z)$
is the image of an element of $\mathscr{B}^{s}$ acting on $\mathsf{F}[[z]]$.
\end{Conjecture}

In the case of the small ample alcove, this result was conjectured earlier for the cotangent bundles of partial flag varieties in \cite{RTV3}.

\begin{Conjecture}\label{eigenvalues}
The element $\textbf{B}^{s}(z)$ has distinct eigenvalues. Hence, to diagonalize the action of $\mathscr{B}^s$ on $\mathsf{F}[[z]]$, it is sufficient to find the eigenvectors of $\textbf{B}^{s}(z)$.
\end{Conjecture}

\section{Saddle point approximation}
We apply the saddle point approximation to the integral of Theorem \ref{intform} to study the $q\to 1$ limit. Rewriting Theorem \ref{intform} gives
$$
\Psi^{\nabla}_{\lambda,\mu}= \frac{1}{2\pi i \alpha_{\mu}} \int_{C_{\mu}}  e^{S(\vec{x})} \textbf{g}^{\nabla}_{\text{opp},\lambda}(\vec{x}) \prod_{i=1}^{n} d x_i
$$
where 
$$
S(\vec{x})= -\frac{\ln(\sqrt{\hbar} z) \ln(x_1 \ldots x_{n})}{\ln(q)} + \ln\left(\Phi(\vec{x})\right)
$$
The critical points of the function $S(\vec{x})$ are defined by the equations
$$
\frac{\partial }{\partial x_i} S(\vec{x})=0,\quad i=1,\ldots,n,  \quad x_j \in \mathbb{C}(t_1^{1/2},t_2^{1/2},q)[[z]]
$$
This is equivalent to
\begin{equation}\label{saddle}
0=-\frac{\ln(\sqrt{\hbar} z)}{\ln(q)}+ x_i \frac{\partial}{\partial x_i}\ln\left(\Phi(\vec{x})\right), \quad i=1.\ldots,n
\end{equation}

\begin{Lemma}[\cite{Pushk1} Lemma 2]
Let $C$ be constant in $x$ such that $C_{1}:=\lim_{q\to 1} C$ exists and $|Cx|<1$. Suppose that $|q|<1$. Then
$$
x \frac{\partial}{\partial x} \ln((C x;q)_{\infty}) = -\frac{\ln(1-C_1 x)}{\ln(q)} +o\left(\frac{1}{\ln(q)}\right), \quad q\to 1
$$
Here, the little $o$ notation means
$$
\lim_{q\to 1} \ln(q)\left( \frac{\partial}{\partial x} \ln((C x;q)_{\infty}) +\frac{\ln(1-C x)}{\ln(q)}\right)=0
$$
\end{Lemma}
\begin{proof}
We calculate
\begin{align*}
x \frac{\partial}{\partial x} \ln((Cx;q)_{\infty})&=-\sum_{i=0}^{\infty} \frac{Cx q^i}{1-Cxq^i} = -\sum_{i=0}^{\infty} \sum_{j=0}^{\infty} (Cx)^{j+1} q^{i(j+1)} \\
&=-\sum_{j=0}^{\infty} \frac{(Cx)^{j+1}}{1-q^{j+1}} =-\sum_{j=0}^{\infty}\frac{(Cx)^{j+1}}{1-e^{(j+1)\ln(q)}}
\end{align*}
So
$$
\lim_{q\to 1} \ln(q) x \frac{\partial}{\partial x}\ln((Cx;q)_{\infty})=-\ln(1-C_1 x)
$$
\end{proof}

\begin{Lemma}
The $q\to 1$ limit of the saddle points of $S(\vec{x})$ satisfy the equations
$$
\sqrt{\hbar} z \left(1-\frac{1}{x_i t_1 t_2}\right)\prod_{j\neq i} w(x_i/x_j)= \left(1-\frac{1}{x_i}\right)\prod_{j \neq i}w(x_j/x_i), \quad i=1,\ldots, n
$$
where 
$$
w(x)=\frac{(1-x)(t_1t_2-x)}{(t_2-x)(t_1-x)}
$$
These are known as the Bethe equations.
\end{Lemma}
\begin{proof}
Using the previous Lemma the equation (\ref{saddle}) for $x_i$ is equivalent to
\begin{multline*}
0=\ln\left(\sqrt{\hbar} z\right) +\ln\left(1-\frac{1}{t_1 t_2 x_i}\right) - \ln\left(1-\frac{1}{x_i}\right)+ \\ \sum_{j \neq i}\bigg(-\ln\left(1-\frac{x_j}{x_i}t_2\right)+\ln\left(1-\frac{x_j}{t_1 x_i}\right)-\ln\left(1-\frac{x_j}{x_i t_1 t_2}\right) + \ln\left(1-\frac{x_j}{x_i}\right) \\ +\ln\left(1-\frac{x_i}{x_j}t_2\right)-\ln(1-\frac{x_i}{t_1 x_j})+\ln\left(1-\frac{x_i}{x_j t_1 t_2}\right) -\ln\left(1-\frac{x_i}{x_j}\right) \bigg)
\end{multline*}
Exponentiating gives 
$$
\sqrt{\hbar} z \left(1-\frac{1}{x_i t_1 t_2}\right)\prod_{j\neq i} \frac{\left(1-\frac{x_j}{x_i}\right) \left(1-\frac{x_i}{x_j t_1 t_2}\right)}{\left(1-\frac{x_j}{x_i}t_2\right)\left(1-\frac{x_i}{x_j t_1}\right)}= \left(1-\frac{1}{x_i}\right)\prod_{j \neq i}\frac{\left(1-\frac{x_i}{x_j}\right)\left(1-\frac{x_j}{x_i t_1 t_2}\right)}{\left(1-\frac{x_j}{x_i t_1}\right)\left(1-\frac{x_i}{x_j} t_2\right)}
$$
which can be rewritten as
$$
\sqrt{\hbar} z \left(1-\frac{1}{x_i t_1 t_2}\right)\prod_{j\neq i} \frac{\left(1-\frac{x_i}{x_j}\right)\left( t_1 t_2-\frac{x_i}{x_j}\right)}{\left(t_2-\frac{x_i}{x_j}\right)\left( t_1-\frac{x_i}{x_j}\right)}= \left(1-\frac{1}{x_i}\right)\prod_{j \neq i}\frac{\left(1-\frac{x_j}{x_i}\right)\left( t_1 t_2-\frac{x_j}{x_i}\right)}{\left( t_1-\frac{x_j}{x_i}\right)\left(t_2-\frac{x_j}{x_i}\right)}
$$
\end{proof}

\begin{Proposition}
The $q\to 1$ limit of the saddle point approximation of $\Psi^{\nabla}_{\lambda,\mu}$ is 
$$
\frac{1}{2\pi i \alpha_{\mu}} \left(2 \pi \ln(q) \right)^{n/2} \frac{ e^{S(\vec{x})}}{ \sqrt{\det \frac{\partial^2 S(\vec{x})}{\partial x_j \partial x_i} }} \textbf{g}^{\nabla}_{\text{opp},\lambda}(\vec{x})
$$
where $x_i\in\mathbb{C}(t_1^{1/2},t_2^{1/2})[[z]]$ satisfies the Bethe equations with initial condition $x_i|_{z=0}=\varphi^{\mu}_{i}$.
\end{Proposition}

\begin{Conjecture}
The $q\to 1$ asymptotics of the function $\Psi^{\nabla}_{\lambda,\mu}$ are dominated by its saddle point approximation for some saddle point satisfying the Bethe equations with initial condition $x_i|_{z=0}=\varphi^{\mu}_{i}$.
\end{Conjecture}

\begin{Remark}
There are several complications involved in the above conjecture. As the Bethe equations are a system of nonlinear equations, it is first of all not obvious that solutions even exist. Secondly, the solutions may not be uniquely determined by the initial condition. In this case, there may be several saddle points that contribute to the $q\to 1 $ limit of the function $\Psi^{\nabla}_{\lambda,\mu}$. Part of the conjecture is that there is a unique saddle point that dominates.
\end{Remark}

Our final result is an explicit description of the eigenbasis of the action of the Bethe subalgebra $\mathscr{B}^{s}$ on $\mathsf{F}[[z]]$. Fix $n \in \mathbb{N}$ and pick an alcove for $\text{Hilb}^{n}(\mathbb{C}^2)$ containing $s$. Let $\lambda_1,\ldots, \lambda_{m}$ be the partitions of $n$. Define the column vector $\textbf{g}^{\nabla}_{\text{opp}}(\vec{x})=(\textbf{g}^{\nabla}_{\text{opp},\lambda_1},\ldots,\textbf{g}^{\nabla}_{\text{opp},\lambda_{m}})^{T}$ in the fixed point basis of $\mathsf{F}$.
\begin{Proposition}
Let $\vec{x}=(x_1,\ldots,x_n)$ be a solution to the Bethe equations with initial condition $x_i|_{z=0}=\varphi^{\mu}_{i}$. Then the vector $\textbf{g}^{\nabla}_{\text{opp}}(\vec{x})$ is an eigenvector of the action of the Bethe subalgebra $\mathscr{B}^{s}$ on the degree $n$ component of $\mathsf{F}[[z]]$. The collection of these vectors, as $\mu$ runs over all partitions of $n$ and $n$ runs over all natural numbers, gives an eigenbasis of the action of $\mathscr{B}^{s}$ on $\mathsf{F}[[z]]$.
\end{Proposition}
\begin{proof}
Write $\Psi^{\nabla}_{\mu}$ for the $\mu$th column of the fundamental solution matrix $\Psi^{\nabla}_{\lambda,\mu}$. The saddle point approximation of this column vector is
$$
\Psi^{\nabla}_{\mu}=\frac{1}{2\pi i \alpha_{\mu}} \left(2 \pi \ln(q) \right)^{n/2} \frac{ e^{S(\vec{x})}}{ \sqrt{\det \frac{\partial^2 S(\vec{x})}{\partial x_j \partial x_i} }} \left(\textbf{g}^{\nabla}_{\text{opp}}(\vec{x})+\ldots\right)
$$
where the $x_i$ solve the Bethe equations with initial condition $x_i|_{z=0}=\varphi^{\mu}_{i}$. The $q$-difference equation is
$$
\Psi^{\nabla}_{\mu}( z q) \mathscr{L} =\const_{X}  \mathscr{L} \textbf{B}^{\nabla}_{\mathscr{L}}(z) \Psi^{\nabla}_{\mu}(z)
$$
where $\mathscr{L}$ is the operator of multiplication by the dual of the tautological line bundle, whose fiber $\mathscr{L}_{\mu}$ over $\mu$ has weights given by $\prod_{i} (\varphi^{\mu}_{i})^{-1}$. Substituting the saddle point approximation into the $q$-difference equation and writing the $z$ dependence explicitly, we have
$$
\frac{\left(2 \pi \ln(q) \right)^{n/2}}{2\pi i \alpha_{\mu}(z q)}  \frac{ e^{S(\vec{x},zq)}}{ \sqrt{\det \frac{\partial^2 S(\vec{x},z q)}{\partial x_j \partial x_i} }} \left(\textbf{g}^{\nabla}_{\text{opp}}(\vec{x})+\ldots\right) \mathscr{L}_{\mu} = \const_{X}  \mathscr{L} \textbf{B}^{\nabla}_{\mathscr{L}}(z) \frac{\left(2 \pi \ln(q) \right)^{n/2}}{2\pi i \alpha_{\mu}(z)}  \frac{ e^{S(\vec{x},z)}}{ \sqrt{\det \frac{\partial^2 S(\vec{x},z)}{\partial x_j \partial x_i} }} \left(\textbf{g}^{\nabla}_{\text{opp}}(\vec{x})+\ldots\right)
$$
which is equivalent to
$$
\frac{\mathscr{L}_{\mu}^{-1}}{\alpha_{\mu}(z)} \frac{ e^{S(\vec{x},z q)}}{ \sqrt{\det \frac{\partial^2 S(\vec{x},z q)}{\partial x_j \partial x_i} }} \left(\textbf{g}^{\nabla}_{\text{opp}}(\vec{x})+\ldots\right) \mathscr{L}_{\mu} =\const_{X}   \mathscr{L} \textbf{B}^{\nabla}_{\mathscr{L}}(z) \frac{1}{\alpha_{\mu}(z)} \frac{ e^{S(\vec{x},z)}}{ \sqrt{\det \frac{\partial^2 S(\vec{x},z)}{\partial x_j \partial x_i} }} \left(\textbf{g}^{\nabla}_{\text{opp}}(\vec{x})+\ldots\right)
$$
One can show that
$$
\lim_{q\to 1} \frac{\det \frac{\partial^2 S(\vec{x},zq)}{\partial x_j \partial x_i}}{\det \frac{\partial^2 S(\vec{x},z)}{\partial x_j \partial x_i}}=1 \quad \text{and} \quad \frac{e^{S(\vec{x},zq)}}{e^{S(\vec{x},z)}}= (x_1\ldots x_n)^{-1}
$$
from the definition of $S(\vec{x},z)$. So
$$
\left(\const_{X} \mathscr{L} \textbf{B}^{\nabla}_{\mathscr{L}}(z)\right)\big|_{q=1} \textbf{g}^{\nabla}_{\text{opp}}(\vec{x})= (x_1\ldots x_n)^{-1} \textbf{g}^{\nabla}_{\text{opp}}(\vec{x})
$$

So $\textbf{g}^{\nabla}_{\text{opp}}(\vec{x})$ is an eigenvector of $(\mathscr{L} \textbf{B}^{\nabla}_{\mathscr{L}}(z))|_{q=1}$. By Conjecture \ref{eigenvalues}, it is also an eigenvector of the Bethe subalgebra. Since we have one of these eigenvectors for each $\mu$, there are sufficiently many to provide an eigenbasis for the Bethe subalgebra.

\end{proof}


\backmatter


\singlespacing
\setlength\bibitemsep{2\itemsep}
\printbibliography[title=REFERENCES]



\end{document}